\documentclass[11pt,reqno]{amsart}
\usepackage{amssymb,dsfont,enumerate}

\usepackage[letterpaper,hmargin=1.0in,vmargin=1.0in]{geometry} 

\newcommand{\red}[1]{{\color{red} #1  }}
\usepackage{mathtools}
\mathtoolsset{showonlyrefs}
\usepackage[colorlinks]{hyperref}
\usepackage[american]{babel}
\usepackage{amsmath}
\usepackage{mathtools}
\usepackage{tikz}
\usepackage{subfigure}
\usepackage{color}
\usepackage{mathrsfs}
\mathtoolsset{showonlyrefs,showmanualtags}

\numberwithin{equation}{section}
\numberwithin{figure}{section}

\DeclareMathSymbol{\leqslant}{\mathalpha}{AMSa}{"36} 
\DeclareMathSymbol{\geqslant}{\mathalpha}{AMSa}{"3E} 
\DeclareMathSymbol{\eset}{\mathalpha}{AMSb}{"3F}     
\renewcommand{\leq}{\;\leqslant\;}                   
\renewcommand{\geq}{\;\geqslant\;}                   
\newcommand{\dd}{\,\text{\rm d}}             
\newcommand{\Av}{\,\text{\rm Av}\,}  
\newcommand{\VSpairs}{\bbX} 
\DeclareMathOperator*{\union}{\bigcup}       
\DeclareMathOperator*{\inter}{\bigcap}       
\newcommand{\maxtwo}[2]{\max_{\substack{#1 \\ #2}}} 
\newcommand{\mintwo}[2]{\min_{\substack{#1 \\ #2}}} 
\newcommand{\suptwo}[2]{\sup_{\substack{#1 \\ #2}}} 
\newcommand{\inftwo}[2]{\inf_{\substack{#1 \\ #2}}} 
\newcommand{\sumtwo}[2]{\sum_{\substack{#1 \\ #2}}} 
\newcommand{\sumthree}[3]{\sum_{\substack{#1 \\ #2 \\ #3}}} 
\newcommand{\uniontwo}[2]{\union_{\substack{#1 \\ #2}}} 
\newcommand{\intertwo}[2]{\inter_{\substack{#1 \\ #2}}} 
\newcommand{\inttwo}[2]{\int_{\substack{#1 \\ #2}}}     
\newcommand{\limtwo}[2]{\lim_{\substack{#1 \\ #2}}}     
\newcommand{\liminftwo}[2]{\liminf_{\substack{#1 \\ #2}}} 
\newcommand{\limsuptwo}[2]{\limsup_{\substack{#1 \\ #2}}} 
\newcommand{\prodtwo}[2]{\prod_{\substack{#1 \\ #2}}}     
\newcommand{\prodthree}[3]{\prod_{\substack{#1 \\ #2 \\ #3}}} 

\newcommand{\grad}{\nabla} 
\newcommand{\gap}{{\rm gap}} 
\newcommand{\what}{\ensuremath{\clubsuit\otimes\clubsuit}} 
\newcommand{\df}{:=}
\newcommand{\ind}{\mathbf{1}}
\newcommand\at[2]{\left.#1\right|_{#2}}

\newcommand{\E}{\mathds{E}}
\newcommand{\Prob}{\mathds{P}}
\newcommand{\e}{\varepsilon}
\newcommand{\setone}{\mathds{1}}
\newcommand{\skp}[2]{\left\langle  #1 , #2 \right\rangle}
\newcommand{\h}{\frac{1}{2}}
\newcommand{\R}{\mathbb{R}}
\newcommand{\N}{\mathbb{N}}
\newcommand{\IND}{{\bf 1}}
\newcommand{\tmix}{T_{\rm mix}}

\newcommand{\vecx}{\vec{X}}
\newcommand{\si}{\sigma} 
\newcommand{\ent}{{\rm Ent} } 
\newcommand{\var}{{\rm Var} } 
\newcommand{\wt}{\widetilde } 
\newcommand{\tc}{\, |\, } 
\newcommand{\scr}{\mathscr}  
\newcommand{\cov}{{\rm Cov} } 
\newcommand{\dert}{\frac{{\rm d}}{{\rm d} t}}  
\newcommand{\scalar}[2]{\langle #1 , #2\rangle}

\newtheorem{theorem}{Theorem}[section]
\newtheorem{proposition}[theorem]{Proposition}
\newtheorem{lemma}[theorem]{Lemma}
\newtheorem{corollary}[theorem]{Corollary}
\newtheorem{remark}[theorem]{Remark}
\newtheorem{definition}[theorem]{Definition}

\newcommand{\cA}{\ensuremath{\mathcal A}} 
\newcommand{\cB}{\ensuremath{\mathcal B}} 
\newcommand{\cC}{\ensuremath{\mathcal C}} 
\newcommand{\cD}{\ensuremath{\mathcal D}} 
\newcommand{\cE}{\ensuremath{\mathcal E}} 
\newcommand{\cF}{\ensuremath{\mathcal F}} 
\newcommand{\cG}{\ensuremath{\mathcal G}} 
\newcommand{\cH}{\ensuremath{\mathcal H}} 
\newcommand{\cI}{\ensuremath{\mathcal I}} 
\newcommand{\cJ}{\ensuremath{\mathcal J}} 
\newcommand{\cK}{\ensuremath{\mathcal K}} 
\newcommand{\cL}{\ensuremath{\mathcal L}} 
\newcommand{\cM}{\ensuremath{\mathcal M}} 
\newcommand{\cN}{\ensuremath{\mathcal N}} 
\newcommand{\cO}{\ensuremath{\mathcal O}} 
\newcommand{\cP}{\ensuremath{\mathcal P}} 
\newcommand{\cQ}{\ensuremath{\mathcal Q}} 
\newcommand{\cR}{\ensuremath{\mathcal R}} 
\newcommand{\cS}{\ensuremath{\mathcal S}} 
\newcommand{\cT}{\ensuremath{\mathcal T}} 
\newcommand{\cU}{\ensuremath{\mathcal U}} 
\newcommand{\cV}{\ensuremath{\mathcal V}} 
\newcommand{\cW}{\ensuremath{\mathcal W}} 
\newcommand{\cX}{\ensuremath{\mathcal X}} 
\newcommand{\cY}{\ensuremath{\mathcal Y}} 
\newcommand{\cZ}{\ensuremath{\mathcal Z}}

\newcommand{\calA}{\mathcal{A}}
\newcommand{\calB}{\mathcal{B}}
\newcommand{\calC}{\mathcal{C}}
\newcommand{\calD}{\mathcal{D}}
\newcommand{\calE}{\mathcal{E}}
\newcommand{\calF}{\mathcal{F}}
\newcommand{\calG}{\mathcal{G}}
\newcommand{\calH}{\mathcal{H}}
\newcommand{\calI}{\mathcal{I}}
\newcommand{\calJ}{\mathcal{J}}
\newcommand{\calK}{\mathcal{K}}
\newcommand{\calL}{\mathcal{L}}
\newcommand{\calM}{\mathcal{M}}
\newcommand{\calN}{\mathcal{N}}
\newcommand{\calO}{\mathcal{O}}
\newcommand{\calP}{\mathcal{P}}
\newcommand{\calQ}{\mathcal{Q}}
\newcommand{\calR}{\mathcal{R}}
\newcommand{\calS}{\mathcal{S}}
\newcommand{\calT}{\mathcal{T}}
\newcommand{\calU}{\mathcal{U}}
\newcommand{\calV}{\mathcal{V}}
\newcommand{\calW}{\mathcal{W}}
\newcommand{\calX}{\mathcal{X}}
\newcommand{\calY}{\mathcal{Y}}
\newcommand{\calZ}{\mathcal{Z}}

\newcommand{\fra}{\mathfrak{a}}
\newcommand{\frb}{\mathfrak{b}}
\newcommand{\frc}{\mathfrak{c}}
\newcommand{\frd}{\mathfrak{d}}
\newcommand{\fre}{\mathfrak{e}}
\newcommand{\frf}{\mathfrak{f}}
\newcommand{\frg}{\mathfrak{g}}
\newcommand{\frh}{\mathfrak{h}}
\newcommand{\fri}{\mathfrak{i}}
\newcommand{\frj}{\mathfrak{j}}
\newcommand{\frk}{\mathfrak{k}}
\newcommand{\frl}{\mathfrak{l}}
\newcommand{\frm}{\mathfrak{m}}
\newcommand{\frn}{\mathfrak{n}}
\newcommand{\fro}{\mathfrak{o}}
\newcommand{\frp}{\mathfrak{p}}
\newcommand{\frqqq}{\mathfrak{q}}
\newcommand{\frr}{\mathfrak{r}}
\newcommand{\frs}{\mathfrak{s}}
\newcommand{\frt}{\mathfrak{t}}
\newcommand{\fru}{\mathfrak{u}}
\newcommand{\frv}{\mathfrak{v}}
\newcommand{\frw}{\mathfrak{w}}
\newcommand{\frx}{\mathfrak{x}}
\newcommand{\fry}{\mathfrak{y}}
\newcommand{\frz}{\mathfrak{z}}
\newcommand{\frA}{\mathfrak{A}}
\newcommand{\frB}{\mathfrak{B}}
\newcommand{\frC}{\mathfrak{C}}
\newcommand{\frD}{\mathfrak{D}}
\newcommand{\frE}{\mathfrak{E}}
\newcommand{\frF}{\mathfrak{F}}
\newcommand{\frG}{\mathfrak{G}}
\newcommand{\frH}{\mathfrak{H}}
\newcommand{\frI}{\mathfrak{I}}
\newcommand{\frJ}{\mathfrak{J}}
\newcommand{\frK}{\mathfrak{K}}
\newcommand{\frL}{\mathfrak{L}}
\newcommand{\frM}{\mathfrak{M}}
\newcommand{\frN}{\mathfrak{N}}
\newcommand{\frO}{\mathfrak{O}}
\newcommand{\frP}{\mathfrak{P}}
\newcommand{\frQ}{\mathfrak{Q}}
\newcommand{\frR}{\mathfrak{R}}
\newcommand{\frS}{\mathfrak{S}}
\newcommand{\frT}{\mathfrak{T}}
\newcommand{\frU}{\mathfrak{U}}
\newcommand{\frV}{\mathfrak{V}}
\newcommand{\frW}{\mathfrak{W}}
\newcommand{\frX}{\mathfrak{X}}
\newcommand{\frY}{\mathfrak{Y}}
\newcommand{\frZ}{\mathfrak{Z}}

\newcommand{\bbA}{\mathbb{A}}
\newcommand{\bbB}{\mathbb{B}}
\newcommand{\bbC}{\mathbb{C}}
\newcommand{\bbD}{\mathbb{D}}
\newcommand{\bbE}{\mathbb{E}}
\newcommand{\bbF}{\mathbb{F}}
\newcommand{\bbG}{\mathbb{G}}
\newcommand{\bbH}{\mathbb{H}}
\newcommand{\bbI}{\mathbb{I}}
\newcommand{\bbJ}{\mathbb{J}}
\newcommand{\bbK}{\mathbb{K}}
\newcommand{\bbL}{\mathbb{L}}
\newcommand{\bbM}{\mathbb{M}}
\newcommand{\bbN}{\mathbb{N}}
\newcommand{\bbO}{\mathbb{O}}
\newcommand{\bbP}{\mathbb{P}}
\newcommand{\bbQ}{\mathbb{Q}}
\newcommand{\bbR}{\mathbb{R}}
\newcommand{\bbS}{\mathbb{S}}
\newcommand{\bbT}{\mathbb{T}}
\newcommand{\bbU}{\mathbb{U}}
\newcommand{\bbV}{\mathbb{V}}
\newcommand{\bbW}{\mathbb{W}}
\newcommand{\bbX}{\mathbb{X}}
\newcommand{\bbY}{\mathbb{Y}}
\newcommand{\bbZ}{\mathbb{Z}}

\newcommand{\sfa}{\mathsf a}
\newcommand{\sfb}{\mathsf b}
\newcommand{\sfc}{\mathsf c}
\newcommand{\sfd}{\mathsf d}
\newcommand{\sfe}{\mathsf e}
\newcommand{\sff}{\mathsf f}
\newcommand{\sfg}{\mathsf g}
\newcommand{\sfh}{\mathsf h}

\newcommand{\sfi}{\mathsf i}
\newcommand{\sfj}{\mathsf j}
\newcommand{\sfk}{\mathsf k}
\newcommand{\sfl}{{\mathsf l}}
\newcommand{\sfm}{{\mathsf m}}
\newcommand{\sfn}{{\mathsf n}}
\newcommand{\sfo}{{\mathsf o}}
\newcommand{\sfp}{{\mathsf p}}

\newcommand{\sfr}{{\mathsf r}}
\newcommand{\sfs}{{\mathsf s}}
\newcommand{\sft}{{\mathsf t}}
\newcommand{\sfu}{{\mathsf u}}
\newcommand{\sfv}{{\mathsf v}}
\newcommand{\sfw}{{\mathsf w}}
\newcommand{\sfx}{{\mathsf x}}
\newcommand{\sfy}{{\mathsf y}}
\newcommand{\sfz}{{\mathsf z}}

\newcommand{\sfA}{\mathsf{A}}
\newcommand{\sfB}{\mathsf{B}}
\newcommand{\sfC}{\mathsf{C}}
\newcommand{\sfD}{\mathsf{D}}
\newcommand{\sfE}{\mathsf{E}}
\newcommand{\sfF}{\mathsf{F}}
\newcommand{\sfG}{\mathsf{G}}
\newcommand{\sfH}{\mathsf{H}}
\newcommand{\sfI}{\mathsf{I}}
\newcommand{\sfJ}{\mathsf{J}}
\newcommand{\sfK}{\mathsf{K}}
\newcommand{\sfL}{\mathsf{L}}
\newcommand{\sfM}{\mathsf{M}}
\newcommand{\sfN}{\mathsf{N}}
\newcommand{\sfO}{\mathsf{O}}
\newcommand{\sfP}{\mathsf{P}}
\newcommand{\sfQ}{\mathsf{Q}}
\newcommand{\sfR}{\mathsf{R}}
\newcommand{\sfS}{\mathsf{S}}
\newcommand{\sfT}{\mathsf{T}}
\newcommand{\sfU}{\mathsf{U}}
\newcommand{\sfV}{\mathsf{V}}
\newcommand{\sfW}{\mathsf{W}}
\newcommand{\sfX}{\mathsf{X}}
\newcommand{\sfY}{\mathsf{Y}}
\newcommand{\sfZ}{\mathsf{Z}}

\newcommand{\ua}{\underline{a}}
\newcommand{\ub}{\underline{b}}
\newcommand{\uc}{\underline{c}}
\newcommand{\ud}{\underline{d}}
\newcommand{\ue}{\underline{e}}
\newcommand{\uf}{\underline{f}}
\newcommand{\ug}{\underline{g}}
\newcommand{\uh}{\underline{h}}
\newcommand{\ur}{\underline{r}}
\newcommand{\us}{\underline{s}}
\newcommand{\ut}{\underline{t}}
\newcommand{\uv}{\underline{v}}
\newcommand{\uu}{\underline{u}}
\newcommand{\ux}{\underline{x}}
\newcommand{\uy}{\underline{y}}
\newcommand{\uz}{\underline{z}}
\newcommand{\uw}{\underline{w}}

\newcommand{\ueta}{\underline{\eta}}
\newcommand{\unu}{\underline{\nu}}
\newcommand{\urho}{\underline{\rho}}
\newcommand{\ukappa}{\underline{\kappa}}
\newcommand{\uxi}{\underline{\xi}}
\newcommand{\uzeta}{\underline{\zeta}}
\newcommand{\uchi}{\underline{\chi}}
\newcommand{\udelta}{\underline{\delta}}

\newcommand{\uA}{\underline{A}}
\newcommand{\uB}{\underline{B}}
\newcommand{\uX}{\underline{X}}
\newcommand{\uY}{\underline{Y}}
\newcommand{\uZ}{\underline{Z}}

\newcommand{\uset}[2]{\underline{#1}^{(#2 )}}

\newcommand{\lb}{\left(}
\newcommand{\rb}{\right)}
\newcommand{\lbr}{\left\{}
\newcommand{\rbr}{\right\}}
\newcommand{\lbs}{\left[}
\newcommand{\rbs}{\right]}

\newcommand{\lra}{\leftrightarrow}
\newcommand{\slra}[1]{\stackrel{#1}{\longleftrightarrow}}

\newcommand{\1}{\mathbbm{1}}

\newcommand{\smo}[1]{{\mathrm o}\lb #1\rb }
\newcommand{\so}{\mathrm (1)}

\newcommand{\eqvs}{\stackrel{\sim}{=}}
\newcommand{\les}{\lesssim}            
\newcommand{\ges}{\gtrsim}             

\newcommand{\be}[1]{\begin{equation}\label{#1}}
\newcommand{\ee}{\end{equation}}

\newcommand{\ep}{\varepsilon}

\newcommand{\Hla}{H_\lambda}
\newcommand{\RWP}{\sfp}
\newcommand{\An}{\bbA_n^+}
\newcommand{\Anl}{\bbA_{n, \lambda}^{+}}
\newcommand{\Anr}{\bbA_n^{+, {\mathsf r}}}
\newcommand{\Anlr}{\bbA_{n, \lambda}^{+, {\mathsf r}}}

\newcommand{\fkg}{\stackrel{\mathsf{FKG}}{\prec}}

\newcommand{\pfs}{\mathbb{P}_{{\rm FS}}}
\newcommand{\efs}{\mathbb{E}_{{\rm FS}}}

\definecolor{darkblue}{rgb}{0,0.3,0.9}
\newcommand{\cb}{\color{darkblue}}
\newcommand{\cn}{\normalcolor}
\newcommand{\eqd}{\stackrel{d}{=}}
\newcommand{\uell}{\underline{\ell}}

\let\a=\alpha \let\b=\beta  \let\c=\chi \let\d=\delta  \let\e=\varepsilon
\let\f=\varphi \let\g=\gamma \let\h=\eta    \let\k=\kappa  \let\l=\lambda
\let\m=\mu   \let\n=\nu   \let\o=\omega    \let\p=\pi  \let\ph=\varphi
\let\r=\rho  \let\s=\sigma \let\t=\tau   \let\th=\vartheta
\let\y=\upsilon \let\x=\xi \let\z=\zeta
\let\D=\Delta \let\F=\Phi  \let\G=\Gamma  \let\L=\Lambda \let\The=\Theta
\let\O=\Omega \let\P=\Pi   \let\Ps=\Psi \let\Si=\Sigma \let\X=\Xi
\let\Y=\Upsilon

\def\({\left(}
\def\){\right)}
\def\gap{\mathop{\rm gap}\nolimits}

\begin{document}
\date{\today} 

\setcounter{tocdepth}{1}

\title[Uniqueness, mixing, and optimal tails for area tilted line ensembles]{Uniqueness, mixing, and optimal tails for Brownian line ensembles with geometric area tilt}
\author{Pietro Caputo and Shirshendu Ganguly}
\address{Pietro Caputo\\ Universit\`{a} Roma Tre}
\email{pietro.caputo@uniroma3.it}
\address{Shirshendu Ganguly\\ University of California, Berkeley}
\email{sganguly@berkeley.edu}

\thispagestyle{empty}

\begin{abstract}
We consider non-colliding Brownian lines above a hard wall,
which are subject to geometrically growing (given by a parameter $\l>1$) self-potentials of tilted area type, which we call the $\l$-\emph{tilted line ensemble} (LE). The model was introduced in \cite{CIW18, CIW19} 
as a putative scaling limit for the level lines of (2+1)-dimensional 
Solid-On-Solid random interfaces above a hard wall, which in turn is an approximate model for low-temperature 3D Ising interfaces. While the LE has infinitely many lines, the case of  the single line,  known as the Ferrari-Spohn (FS) diffusion, is one of the canonical interfaces appearing in the Kardar-Parisi-Zhang (KPZ) universality class and is well studied \cite{ferrarispohn2005}. 

The $\l$-{tilted} LE can also be viewed as a Gibbs measure. Much of the recent work on Gibbsian line ensembles has focussed on the Airy LE constructed in \cite{corwinhammond}, another central object in the KPZ class with determinantal structure.
In contrast with the Airy LE, the presence of geometrically growing area tilts renders the $\l$-{tilted} LE non-integrable. In \cite{CIW18,CIW19}, a stationary infinite volume Gibbs measure was constructed as a limit of finite LEs on finite intervals with zero boundary conditions, and initial control on its fluctuations was given in terms of first moment estimates for one-point marginals and for suitable curved maxima. We refer to this as the zero boundary $\l$-{tilted} LE. Subsequently, \cite{ DLZ} revisited the case of finitely many lines  
and established an equivalence between the free and the zero boundary LEs. 

In this article, we develop probabilistic arguments  to resolve several  questions that remained open.
We prove that the infinite volume zero boundary $\l$-{tilted} LE is mixing and hence ergodic and establish a quantitative decay of correlation. Further, we prove an optimal upper tail estimate for the top line showing that it is essentially the same as the FS diffusion, i.e., the presence of infinitely many lines doesn't significantly push the top line up. Finally, we prove uniqueness of the Gibbs measure in the sense that 
any \emph{uniformly tight}  $\l$-{tilted} LE (a notion which includes all stationary $\l$-{tilted} LE) must be the zero boundary $\l$-{tilted} LE. This immediately implies that the LE with free boundary conditions, as the number of lines and the domain size are taken to infinity in an arbitrary fashion, converges to this unique LE.
\end{abstract}

\maketitle
\tableofcontents

\section{Introduction}
The goal of this paper is to develop a unified probabilistic framework that will enable us to answer multiple open questions about area tilted line ensembles. Before diving into the precise model of interest, we start with a broad overview of what line ensembles are, why they are important, and some of the recent developments in their study.  

Line ensembles are a collection of random curves which occur rather naturally in many models of interest, encoding useful information. They could be both discrete or continuous. In certain cases they are natural in the pre-limit as well as admit canonical scaling limits. 
A classical example is Dyson Brownian motion (DBM) which describes the motion of eigenvalues of an  $n\times n$ Gaussian unitary ensemble (GUE) as their complex Gaussian entries perform independent (up to Hermitianness) Brownian motions. Remarkably, an alternate description which is significantly more probabilistic states (see \cite{oconnelyor}) that DBM is a collection of independent Brownian motions conditioned to not intersect. 
This immediately implies that they possess a simple yet powerful Markov property, namely, the conditional law of a finite subset of the curves on a compact interval given everything else is given by 
independent Brownian bridges between the given boundary endpoints, constrained to not intersect and to avoid the remainder of the boundary data.

It turns out that this model is determinantal via the Karlin-McGregor formula which opens up the door to performing asymptotic analysis. Indeed such formulas along with probabilistic methods were employed in \cite{corwinhammond} where the Airy line ensemble was constructed as a scaling limit of the DBM. The aforementioned resampling invariance property passes to the limit which provides a varied set of tools to study this canonical family of random curves whose top line is the parabolic Airy$_2$ process.  Via the Robinson–Schensted–Knuth correspondence, the various lines in the Airy LE together are expected to describe the scaling limit of the asymptotic free energy of disjoint polymers in various $1+1$ dimensional models expected to be in the Kardar-Parisi-Zhang (KPZ) universality class. However, currently, the rigorous proof of such connections exist only for a handful of models possessing exactly solvable properties (see \cite{ dauvergne2018directed, dauvergne2021scaling}).  One of the key phenomenon observed in such models is a competition between Brownian fluctuations and the non-intersection constraint leading to new fluctuation exponents characteristic of KPZ universality.

The Markovian property, in this context referred to as BG property,  also casts such LE as special cases of  infinite volume Gibbs measures which are ubiquitous in statistical mechanics and probability theory, where, though infinite, the local conditional distributions given
boundary conditions, are functions of local energy contributions.  Gibbsian line ensembles are a special class of Gibbs measures, which
have received considerable attention in the past two decades owing, in part, to their occurrence in
integrable probability. A general treatment appears in \cite{barraquand2023spatial}. 

Often in applications one needs a precise understanding of various observables including one point tail estimates, decay of correlation and so on. Such has been studied in great detail recently for the Airy LE, see e.g. \cite{corwinhammond, calvert2019brownian, corwin2014ergodicity,dauvergne2021bulk}. Besides probabilistic ideas, as already indicated, integrable features such as being determinantal which stems from an underlying exchangeabilty is usually a crucial input in the analyses. For a recent set of results, which bypasses such reliance see \cite{ganguly2022sharp}.\\

\noindent
\textbf{Low temperature level curves and entropic repulsion:} Another class of line ensembles, which is indeed the focus of this article, arises naturally by considering  the local restrictions of level curves of discrete random interfaces. 
Perhaps the most canonical such an example comes from low temperature three dimensional Ising model with a hard floor, with all positive boundary conditions except the floor where it is entirely negative, leading to an  interface between the two phases. This is an example of entropic repulsion, which for low temperature $(2+1)$-dimensional crystals above a hard wall has been
the subject of extensive study in statistical physics. While the unconstrained 
surface of the crystal would typically be rigid at height $0$, the presence of a
wall pushes the surface  upwards to increase its entropy (i.e., to allow
downward fluctuations), to a height which typically diverges
logarithmically with the side length $L$ of the box. 
A rigorous study of this phenomenon in the $(2+1)$- dimensional  Solid-On-Solid (SOS) model—a
low temperature approximation of the $3$D Ising model—dates back to Bricmont, El
Mellouki and Fr\"ohlich in 1986 \cite{bricmont1986random}. A rather refined picture was established later in \cite{caputoetal2016} where among other things it was shown that the model admits a sequence of nested level lines each
encompassing a large macroscopic fraction of the sites (see also \cite{caputo2017entropic} for counterpart study of other gradient interface models). As a result of the entropic repulsion, the energy cost of the $i$-{th} curve is linear in the area enclosed between the $i$-th and $(i+1)$-th curve with a pre-factor $\l^i$, for some constant $\l>1$, which grows geometrically with $i$.  

While such curves are analyzable to some degree using cluster expansion techniques, to understand refined properties and construct candidates for the scaling limit, in a series of two papers, \cite{CIW18,CIW19}, the authors introduced and initiated the study of a putative scaling limit, a model consisting of an unbounded number of non-intersecting
Brownian bridges, above a hard wall and subject to geometrically increasing area tilts. 
Thus the individual lines face stronger and stronger pressure towards the wall as one goes down the stack (that is as their index is increased). We will call this model the {$\l$-tilted LE}. The very recent work \cite{serio2023scaling} proves that  the $\l$-tilted LE is the scaling limit of the discrete counterpart consisting  of area tilted non intersecting random walks.

A natural way to create infinite volume limits is to take finite systems on compact domains with boundary conditions and then pass to a limit provided it exists. 
Thus important questions concern existence of such limits and their dependence on boundary conditions, and their properties such as ergodicity, decay of correlation and tail behavior. For the Airy LE many of the above questions have been thoroughly investigated, with most of the arguments relying substantially on integrable inputs. See also the recent results in \cite{
dimitrov2023uniform, ferrari2023airy2} showing convergence to the Airy LE of the $1$-tilted LE, i.e.\ when $\lambda=1$, which admits a determinantal structure as observed in \cite{ioffevelenikwachtel}. In contrast, in absence of any algebraic structure,  for the $\l$-tilted LE with $\l>1$, despite the initial results in \cite{CIW18,CIW19} and further progress made recently in \cite{DLZ}, all of the above questions and others have remained largely open.  

{ In this article we develop a unified probabilistic framework primarily relying on resampling and coupling ideas to resolve several of the above questions.} 

Before moving on to the next section  devoted to setting up the definitions leading to the statements of our main theorems, for the ease of readability we offer a quick glimpse of what the paper aims to establish. 
 This includes:
 \begin{itemize}
 \item Sharp tail estimates (addressing both upper tail as well as entropic repulsion).  
 \item Ergodicity and mixing properties of the infinite LE constructed in \cite{CIW19} as well as quantitative decay of correlation estimates.
 \item Uniqueness of $\l-$tilted LEs under natural tightness assumptions, which in particular implies that there exists a unique stationary Gibbs measure. 
 \end{itemize}

We now move on to the precise description of the model and the main theorems. 
In the forthcoming section we set up the technical apparatus, introduce the Gibbs property, as well as review previously known facts, 
to provide the necessary background 
for the main results that will be presented in Section \ref{mainresults}.

\section{Setup and known facts}
\label{gibbsdef}
We start with the formal definition of LEs. 
\subsection{Line ensembles with geometric area tilts.}
\label{sec:geopol}
For  $\ell < r$, $n\in\bbN$, and $\ux,\uy\in\bbR^n$, 
let 
${\mathbf B}^{\ux,\uy}_{\ell , r} $ be the unnormalized path measure of $n$ independent 
standard Brownian bridges $\uB=(B^1(s),\dots,B^n(s))$, $s\in[\ell , r]$, with boundary data 
$B^i(\ell)=x_i$ and $B^i(r)=y_i$, $i=1,\dots,n$.
%
%
As a convention, we write ${\mathbf B}^{\ux , \uy}_{n; \ell , r}$ also for 
the corresponding expectation, so that the total mass of ${\mathbf B}^{\ux , \uy}_{n; \ell , r}$ satisfies
\begin{equation}\label{eq:q-xy} 
{\mathbf B}^{\ux , \uy}_{n; \ell , r}\left[1\right]=q_{r-\ell} (\ux ,\uy ) := 	
 \tfrac1{{(2\pi(r-\ell ))^{n/2}}}\, e^{- \frac{\|\uy-\ux\|_2^2}{2(r-\ell )}}.
\end{equation}
For  $n\in\bbN$, define the simplex
\begin{equation}\label{eq:Aplus}
\bbA_n^+
= \{ \ux \in \bbR^n \,:\,  x_1 >\dots > x_n > 0\}\,.
\end{equation}
We consider path measures that  are  
supported on the set $\Omega^{+}_{n;  \ell, r}$ of non-intersecting $n$-tuples  $\uX$, 
\begin{equation}\label{eq:Omega-Set} 
\Omega^{+}_{n; \ell , r} = \lbr \uX~:~ \uX (t )\in \bbA_n^+\ \ \forall\, t\in (\ell , r)\rbr,
\end{equation}
and such that the $i$-th path $X^i$ is subject to a potential of area tilt type of the form 
 \begin{equation}\label{eq:areatilt} 
\exp\left(-a\l^{i-1}\int_{\ell}^rX^i(s)\dd s
\right)\,,\quad i=1,\dots,n,
\end{equation}
where $a>0$ and $\l>1$ are fixed constants. 
Thus, 
given 
 $n\ge 1$,  $a>0,\lambda >1$ and $\ux,\uy\in\bbA_n^+$, we consider the
partition  function
\begin{equation}\label{eq:PF-BC-T01} 
 Z_{n; \ell , r}^{ \ux , \uy } (a , \lambda ) 
:=  {\mathbf B}^{\ux , \uy}_{\ell , r} 
\left[
\ind_{\Omega_{n; \ell , r}^+}
e^{-a\sum_{i=1}^n \lambda^{i-1}\!\int_{\ell}^rX^i(s)\dd s )}
\right], 
\end{equation}
and the associated probability measure $\bbP^{\ux,\uy}_{n; \ell , r}\left[ \cdot  ~|a, \lambda  \right]$ defined by the expectations
\begin{equation}\label{eq:PF-BC-T02} 
\bbP^{\ux,\uy}_{n; \ell , r}
\left[ F (\uX ) ~|a, \lambda  \right]  
:= \frac{1}{Z_{n; \ell , r}^{ \ux , \uy } (a , \lambda ) } \,
{\mathbf B}^{\ux , \uy}_{\ell , r} 
\left[ F (\uX )\,
\ind_{{\Omega_{n; \ell, r }^+}}
{\rm e}^{-a\sum_1^n \lambda^{i-1}\!\int_{\ell}^rX^i(s)\dd s} 
\right] ,
\end{equation}
where $F$ is any bounded measurable function over the set of $n$-tuples of continuous functions from $[\ell,r]$ to $\bbR$. 
The measure $\bbP^{\ux,\uy}_{n; \ell , r}\left[ \cdot  ~|a, \lambda  \right]$ will be referred to as 
the $n$-LE with 
 $(a,\lambda)$-geometric area tilts with boundary conditions $(\ux,\uy)$ on the interval $[\ell,r]$.  
We remark that $\bbP^{\ux,\uy}_{n; \ell , r}\left[\, \cdot  \tc a, \lambda  \right]$ is well defined for all $\ux,\uy\in\bbA_n^+$. Indeed, if $\ux,\uy\in\bbA_n^+$
one has 
\begin{equation}\label{eq:PF-BC-T03} 
 {\mathbf B}^{\ux , \uy}_{\ell , r} 
\lb 
{\Omega_{n; \ell , r}^+}
\rb >0 , 
\end{equation}
 and therefore $Z_{n; \ell , r}^{ \ux , \uy } (a , \lambda )\in (0,\infty)$. While this does not apply to all $\ux,\uy\in\bar \bbA_n^+$, 
 the closure of the set $ \bbA_n^+\subset \bbR^n$, it is still possible to define $\bbP^{\ux,\uy}_{n; \ell , r}\left[ \cdot  ~|a, \lambda  \right]$
  in these cases by a limiting procedure, see e.g.\ \cite[Definition 2.13]{corwinhammond} for the case $a=0$. If $\lambda=1$, then $\bbP^{\ux,\uy}_{n; \ell , r}\left[ \cdot  ~|a, \lambda  \right]$ has an explicit determinantal form in the limit $[\ell,r]\to \bbR$, and is referred to as the Dyson-Ferrari-Spohn line ensemble; see \cite{ioffevelenikwachtel}.

Two cases of boundary conditions are of special interest. 
 The first one corresponds to {\em zero boundary conditions}, and is obtained by taking $\ux=\uy=\underline 0$ above: 
 \begin{equation}\label{eq:zerobp}
 \bbP^0_{n; \ell , r}\left[ \cdot  \tc a, \lambda  \right]:=\bbP^{\underline 0,\underline 0}_{n; \ell , r}\left[\cdot  \tc a, \lambda   \right].
 \end{equation}
 The second is the case of {\em free boundary conditions} defined by  
\begin{equation}\label{eq:PolMeas02} 
\bbP^f_{n; \ell , r}
\left[ F (\uX ) ~|a, \lambda  \right]  
:= \frac{1}{{\mathcal Z}^f_{n ;\ell , r} (a, \lambda )} 
\int_{\bbA_{n}^+\times \bbA_{n}^+}
{\mathbf B}^{\ux , \uy}_{\ell , r} 
\left[ F (\uX )
\,\ind_{{\Omega_{n; \ell, r }^+}}
{\rm e}^{-a\sum_1^n \lambda^{i-1}\!\int_{\ell}^rX^i(s)\dd s} 
\right]
\dd \ux \dd\uy ,
\end{equation}
where 
\begin{equation}\label{eq:freepartfct}
{\mathcal Z}^f_{n ;\ell , r} (a, \lambda )
:= 
\int_{\bbA_{n}^+\times \bbA_{n}^+}
Z_{n ;\ell , r}^{ \ux , \uy}  (a, \lambda)  
\dd\ux 
\dd 
\uy ,
\end{equation}
and $\dd\ux, \dd 
\uy$ denote Lebesgue measure on $\bbR^n$.
For a proof that $\bbP^f_{n; \ell , r}
\left[ \cdot~|a, \lambda  \right] $ is well defined, that is ${\mathcal Z}^f_{n ;\ell , r} (a, \lambda )\in(0,\infty)$, for all $a>0,\lambda>1$, see \cite[Appendix~A]{CIW18}. 

To simplify the notation, when $[\ell,r]=[-T,T]$ we will often write 
\begin{equation}\label{eq:notation} 
\mu^f_{n,T}:=\bbP^f_{n;-T,T }\left[ \cdot  \tc a, \lambda  \right]\,,\qquad 
\mu^0_{n,T}:=\bbP^0_{n;-T,T }\left[ \,\cdot  \tc a, \lambda  \right]\,,
\end{equation}
 for the above probability measures.
 In what follows $a$ and $\l$ will often be omitted from our notation. We assume the parameter $\l>1$ to be fixed once for all.
  Moreover, whenever $a$ is not explicitly mentioned, it will be tacitly assumed that $a=1$.

\subsection{The stationary measure}\label{sec:stationary}
For $n=1,$ note that $\mu^0_{1,T}$ is simply a Brownian excursion $B(s)$,  $s\in[-T,T]$, with an area tilt 
\begin{equation}\label{areatiltfactor}
\exp{\(-\int_{-T}^T B(s)\dd t\)},
\end{equation}
 and is known to converge weakly, as $T\to\infty$, to the law of a stationary process known as Ferrari-Spohn diffusion, which we denote $\{Y_{FS}(t),\,t\in\bbR\}$, such that $Y_{FS}(0)$
has density proportional to 
\begin{equation}\label{eq:airyfs} 
{\rm Ai}\(\sqrt[3]{2}\,x-\o_1\)^2\ind_{x>0},
\end{equation} 
where ${\rm Ai}(\cdot)$ is the Airy function and $-\o_1$ denotes its largest zero. This was first constructed in  \cite{ferrarispohn2005} as a scaling limit of the relative height of a Brownian bridge constrained to be above circular or parabolic barriers, and then it was shown to be the scaling limit of a large class of area tilted random walk models \cite{ioffeshlosmanvelenik}. Further, across \cite{prewet1,prewet2} this was shown to also arise as the limit of interfaces in low temperature Ising model in the critical pre-wetting regime.

However, we will focus primarily on the case of large $n$ or $n=\infty$. 
Results from \cite{CIW18,CIW19} 
ensure the tightness of the measures $\mu^f_{n,T}$ and $\mu^0_{n,T}$ as $n\to\infty$ and $T\to\infty$. 
As observed therein,  monotonicity (the precise statement appears in Lemma \ref{lem:mono} below)  implies that there is a well defined, time stationary,  infinite line ensemble $\mu^0$ describing the weak limit of $\mu^0_{n,T}$, the limit  $n,T\to\infty$ being taken in an arbitrary fashion, 
\begin{gather}\label{eq:latile} 
\mu^0= \lim_{\substack {n,T\to\infty}} \mu^0_{n,T}.
\end{gather} 
We often refer to $\mu^0$ as the {\em zero boundary condition  $\l$-tilted line ensemble}, or simply the 
zero boundary $\l$-tilted LE or sometimes, for brevity, the zero boundary LE.  In particular, for every fixed $n\in\bbN$, there exists a time stationary, $n$-LE 
\begin{gather}\label{eq:CIWa} 
\mu^0_n = \lim_{T\to\infty} \mu^0_{n,T}
\end{gather}
 describing the weak limit of $\mu^0_{n,T}$ as $T\to\infty$, and such that 
 $\mu^0_n\to\mu^0$ weakly. 
Recently, Dembo, Lubetzky and Zeitouni \cite{DLZ} showed that for every fixed $n\in\bbN,$ one has the weak convergence 
\begin{gather}\label{eq:DLZ} 
\mu^0_n = \lim_{T\to\infty} \mu^f_{n,T}\,,
\end{gather}
for the free boundary ensemble as well. It follows that, taking $T\to\infty$ first, and then $n\to\infty$, the free boundary measures $\mu^f_{n,T}$ also converge to the stationary field $\mu^0$. We refer to Theorem \ref{tightLE12} and Theorem \ref{th:DLZ} for more formal statements.  

As a consequence of our uniqueness result mentioned in the introduction, it follows that sending $n,T$ to infinity in any arbitrary fashion also yields the same result; see Corollary \ref{freeunique}. Concerning the measure $\mu^0$, we shall also address other questions that were listed as open problems in \cite{CIW18,CIW19}, namely ergodicity and decay of correlations.  

The above facts, combined with  PDE theory, also show that the measure $\mu^0_n$ can be identified with the stationary  diffusion process on $\bbA_n^+$ associated to the Sturm Liouville operator 
\begin{gather}\label{eq:SturmLiouville}
\sfL_n =\sum_{i=1}^n\(\tfrac12\partial_{x_i}^2 - \l^{i-1} x_i\)\,,\qquad\;\; \ux\in \bbA_n^+\,,
\end{gather} 
where $\bbA_n^+$ is defined in \eqref{eq:Aplus}. More precisely, $\mu^0_n$ is the law of the stationary Langevin diffusion with invariant distribution $\Phi_n^2(\ux)\dd \ux$ on $\bbA_n^+$, 
where $\Phi_n$ is the Krein-Rutman eigenfunction of $\sfL_n$  on $\bbA_n^+$ with Dirichlet boundary condition on $\partial \bbA_n^+$; see e.g.\ the discussion in \cite[Section 2.3]
{ioffevelenikwachtel} 
for the special case $\l=1$. It is thus natural to identify $\mu^0$ as the law of an infinite dimensional diffusion process, a point of view that we hope to investigate further in future work.

\subsection{Brownian-Gibbs (BG) property for area tilted measures}\label{bgsec}
As alluded to multiple times already, a crucial tool for us will be a sampling invariance property enjoyed by the ensembles defined above. This extends the notion of Brownian-Gibbs property introduced  in \cite{corwinhammond} to paths with area tilts. We refer to it as the Brownian-Gibbs property of the line ensemble.
We start by recalling the basic definitions.  A more comprehensive treatment can be found in \cite{corwinhammond,corwin2016kpz,dauvergne2021bulk}.

Let the spaces $\bbA_\infty^+$
be defined as the $n=\infty$ version of the set $\bbA_n^+$
given in \eqref{eq:Aplus}.
The sample space is   
$\Omega =\sfC\lb \bbR,\bbA_\infty^+\rb$, the set of continuous functions $f:\bbR\mapsto\bbA_\infty^+$, equipped with the topology of uniform convergence of any finite number of paths on compact subsets, and   with the corresponding Borel $\sigma$-field $\calB$. The coordinate maps 
$\uX\in \Omega \mapsto X^i (t )$ are viewed as position of $i$-th particle at time $t$.  
Following \cite{corwinhammond}, for each $n\in\bbN$, and time interval $[\ell,r]\subset\bbR$, define
the {internal} and external $\sigma$-algebras
\begin{gather}\label{eq:Bext} 
\calB_{n; \ell , r}^{\sfi} 
= \sigma\lb X^i (t ) :\, \text{{$t{\in} (\ell , r)$ {and} $i \le n$}}\rb,
\\
\calB_{n; \ell , r}^{\sfe} 
= \sigma\lb X^i (t ):\, \text{{either $t{\notin} (\ell , r)$ {or} $i >n$}}\rb. 
\end{gather}
Given $a>0,\l>1$, an interval $[\ell,r]$, an integer $n\in\bbN$ and a continuous function $h:\bbR\mapsto \bbR_+$, we write $\bbE_{n ;\ell , r}^{\ux, \uy}[\cdot\tc h]$
for the expectation of the $n$-line ensemble with floor $h$, which is defined as in \eqref{eq:PF-BC-T01}-\eqref{eq:PF-BC-T02}  with the set $\Omega_{n; \ell , r}^+$ replaced by 
\begin{gather}\label{eq:ohset} 
\Omega_{n; \ell , r}^{+,h}=\Omega_{n; \ell , r}^+\cap \{X^n(s)> h(s)\,,\;\forall s\in(\ell,r)\}.
\end{gather}
Let also $\Omega_{\ell , r}=\sfC\lb [\ell,r],\bbA_\infty^+\rb$ denote the set of
paths in the interval $[\ell,r]$.  
\begin{definition}[$\l$-tilted LE]\label{def:BG}
A probability measure $\bbP$ on $\Omega$ is said to have the 
{\em Brownian-Gibbs (BG) property} with respect to $(a , \lambda )$-geometric area tilts (or in short simply the BG property) 
if 
for any bounded measurable $F:\Omega_{\ell,r}\mapsto \bbR$, the corresponding conditional expectations $\bbE[\cdot\tc\calB^{\sfe}_{n ;\ell , r}] $ satisfy 
\begin{equation}\label{eq:BG-prop} 
\bbE\left[ \, F\tc
\calB^{\sfe}_{n ;\ell , r} \right] 
= 
\bbE_{n ;\ell , r}^{\uX^{\le n} (\ell) , \uX^{\le n} (r )} \left[\,F(\,\cdot\,,X^{>n}) 
\tc 
X^{n+1}\right], 
\end{equation}
$\bbP$-a.s for any $-\infty <\ell <r <\infty$ and $n\in \bbN$.
In \eqref{eq:BG-prop}, we use the notation $\uX^{\le n}=(X^1,\dots,X^n)$ and $\uX^{>n}=(X^{n+1},X^{n+2},\dots)$.  A probability measure $\bbP$ on $\Omega$ with the above Brownian-Gibbs property is called a $\l$-tilted line ensemble, or simply  $\l$-tilted LE.
\end{definition}

If $\bbP=\mu^0$ denotes the zero boundary $\l$-tilted LE, that is the weak limit of $\mu^0_{n,T}$ as discussed in Section \ref{sec:stationary}, then it was shown in \cite{CIW19} that  
$\bbP$ has the BG property. A standard argument allows one to show that $\l$-tilted measures also satisfy the
{\em strong BG property}, namely the property \eqref{eq:BG-prop} when the deterministic domain identified by $n$ lines and the time interval $[\ell,r]$ is replaced by the {\em stopping domain} identified by $n$  lines and the time interval $[\t_\ell,\t_r]$, where $\t_\ell$ and $\t_r$ are left and right stopping times respectively (i.e., the event $\{\t_\ell \le t\}\cap\{\t_r \ge s\}$ is in $\calB_{n; t , s}^{\sfe}$); see \cite[Lemma 2.5]{corwinhammond}.

Finally, note that the BG property only specifies the conditional law of finitely many paths on a finite domain. Thus $\underline{X}\sim \mu^0$ satisfying the BG property immediately implies that for any constant $c>0,$  the random element $\underline{X}+c=(X^{i}+c)_{i\ge 1}$ also satisfies the BG property (this uses the fact that the area increase for each curve on a given finite domain is deterministic along with the fact that the Brownian bridge density is invariant under shifting by a constant). This can be also thought of as raising the hard floor from $0$ to $c$. Generalizing further, one can raise the floor to any given $g: \bbR\to \bbR_{\ge 0},$ and in particular if, say, $g$ is bounded and smooth enough, one can in principle construct a BG measure where the floor is $g$ by monotonicity and tightness arguments.

However, postponing the investigation of such LEs to the future, \emph{throughout this article we will be only considering LEs satisfying the following property.} 

\begin{definition}[Asymptotically pinned to zero]\label{def:asympin} A probability measure $\bbP$ on $\Omega$ or the corresponding LE is said to be asymptotically pinned to zero if the following holds. For any $\e>0$ and $T\in \bbR,$ there exists $k=k(T,\e)$ such that 
\begin{equation}\label{eq:asymppin}
\bbP\left(\sup_{s\in [-T,T]}X^{k}(s)\le \e\right)\ge 1-\e.
\end{equation}
\end{definition}

Note that any finite LE $(X^{i}(\cdot))_{1\le i\le n}$ can be naturally seen to satisfy the asymptotic pinning condition by introducing an auxiliary curve $X^{n+1}(\cdot)\equiv 0.$ Further, estimates recorded later (see e.g. Remark \ref{freetozero}) immediately imply the unsurprising fact that $\mu^0$ is indeed an example of BG asymptotically pinned to zero.

\subsection{Monotonicity and scaling}\label{sub:pmeas}
Two key tools for the analysis of $\l$-tilted line ensembles are monotonicity and scaling. 
Monotonicity is expressed by the following stochastic domination properties.
For vectors of continuous functions $\underline f=(f^1,\dots,f^n)$, $\underline g=(g^1,\dots,g^n)$ on $[\ell,r]\subset \bbR$ we define the partial order 
  \begin{equation}\label{eq:partialorder} 
\underline f\prec\underline g\;\;\iff \;\;f^i(t)\le g^i(t)\,,\;\quad\forall t\in[\ell,r]\,,\quad \forall \,i\in[n].
  \end{equation}
 For every $n\in \bbN$ 
and $\ell < r$, consider the ensemble $\bbP_{n ;\ell , r}^{\ux , \uy}\left[\, \cdot \tc \underline h^- , \underline h^+, \urho\right]$ of $n$ lines on the interval $[\ell,r]$, parametrized by 
 boundary conditions 
 $\ux, \uy \in \bbA_n^+$, 
 a pair $\uh=(\underline h^- , \underline h^+)$, $\underline h^\pm=(h_{i}^\pm, i=1,\dots,n)$, where for each $i$, $h_{i}^-, h_{i}^+$ are non-negative continuous functions, called respectively the $i$-th  
 floor and the $i$-th ceiling, satisfying $h_{i}^-\prec  h_{i}^+$, and \[ h_{i+1}^-\prec h_{i}^-, \qquad h_{i+1}^+\prec h_{i}^+,\qquad 1=1,\dots,n-1,\]
and an $n$-tuple of nonnegative 
 continuous functions  $\urho = (\rho_1, \dots , \rho_n)$, 
 called the area tilts. 
The line ensemble $\bbP_{n ;\ell , r}^{\ux , \uy}\left[\, \cdot \tc \underline h^- , \underline h^+, \urho\right]$ is defined by the partition function
\begin{equation}\label{eq:PFgen} 
 Z_{n; \ell , r}^{ \ux , \uy } (\underline h^- , \underline h^+, \urho ) 
:=  {\mathbf B}^{\ux , \uy}_{\ell , r} 
\left[
\ind_{\underline h^-\prec \uX\prec   \underline h^+}
\ind_{\Omega_{n; \ell , r}^+}
{\rm e}^{-\sum_{i=1}^n  \int_{\ell}^r \rho_i(t) X^i (t)\dd t}
\right], 
\end{equation}
Stochastic domination is defined w.r.t.\ the partial order \eqref{eq:partialorder}. 
For two measures $\mu,\nu$ on paths $\uX,\uY$, we write $\mu\succ \nu$ if there exists a coupling $\Gamma$  of $(\mu,\nu)$ such that $\Gamma(\uX\succ\uY)=1$.
The following lemma is proved in \cite[Lemma 1.4]{CIW18}, which in turn is based on  \cite[Lemma 2.6]{corwinhammond}, where 
the basic monotonicity property of non-intersecting brownian line ensembles were first established. 
\begin{lemma}\label{lem:mono}
If,  $\ux\prec \uu$, $\uy\prec \uv$, $\underline h^-\prec \underline g^-$, $\underline h^+\prec \underline g^+$, and  $\urho\succ \ukappa$, 
then 
\begin{equation}\label{eq:FKG-2} 
\bbP_{n ;\ell , r}^{\ux , \uy}\left[\, \cdot \tc \underline h^- , \underline h^+, \urho\right]\prec
\bbP_{n ;\ell , r}^{\uu , \uv}\left[\, \cdot \tc \underline g^- , \underline g^+, \ukappa\right].
\end{equation} 
\end{lemma}
 \begin{remark}\label{rem:nonct}
 The statement in Lemma \ref{lem:mono} extends, with the same proof, to the case of non-continuous ceilings of the form $\bar h^+ = h^+\ind_{\G} + \infty \ind_{\G^c}$ where $h^+$ is a continuous function and $\G$ is a finite union of intervals contained in $[\ell,r]$, and $\G^c=[\ell,r]\setminus\G$. The same extension holds for floors of the form $\bar h^- = h^-\ind_{\G}$ where $h^-$ is a continuous function and $\G$ is as above. Moreover, the lemma generalizes easily to the case of pinned fields obtained by taking ceilings of the form $\bar h^+ = \infty \ind_{\G^c}$ where $\G=\cup_i \{s_i\}$ is a finite collection of points $s_i\in  [\ell,r]$. The latter case corresponds to independent ensembles with zero boundary conditions at the points $s_i$. 
 \end{remark}

Brownian scaling induces a useful scaling relation for $\l$-tilted line ensembles that may be summarized as follows. 
Consider the following mapping of an $n$-tuple 
$\uX$ of paths  on an interval $[-\lambda^{2/3}T , \lambda^{2/3}T]$ 
to $n$-tuple $\uY$ of paths on  $[-T , T]$: 
\begin{equation}\label{eq:BrownScale}  
\uY (\cdot  ) = \frac{1}{\lambda^{1/3}} \uX (\lambda^{2/3}\cdot ) .
\end{equation}
The next lemma says in particular that if $\uY$ is related to $\uX$ via \eqref{eq:BrownScale}, then 
$\uY$ has distribution $\bbP^0_{n, T} \left[ \cdot   ~|a\lambda , \lambda  \right]$ if and only if 
$\uX$ has distribution $\bbP^0_{n, T\lambda^{2/3}} \left[ \cdot   ~|a , \lambda  \right]$, and the same holds for the free boundary ensemble. See \cite[Lemma 1.1]{CIW18} for a proof.
\begin{lemma} 
\label{lem:scaling} 
For all $n\in\bbN,T>0$, $a>0,\lambda>1$, and $ \ux , \uy\in\bbA_n^+$,
\begin{equation}\label{eq:BS1} 
Z_{n,  \ell\lambda^{2/3},r\lambda^{2/3}}^{ \ux , \uy}( a, \lambda) = \lambda^{-\frac{n}3}
Z_{n, \ell,r}^{\lambda^{-1/3} \ux , \lambda^{-1/3}\uy} ( a\lambda , \lambda). 
\end{equation}
Moreover, for any bounded measurable function $F$ on $\Omega_{n,\ell,r}^+$,
\begin{gather}\label{eq:BS2} 
\bbP^0_{n,\ell,r}
\left[ F (\uX ) ~|a\lambda, \lambda  \right]  = \bbP^0_{n, \ell\lambda^{2/3},r\lambda^{2/3}}
\left[ F (\l^{-1/3}\uX ) ~|a, \lambda  \right],
\end{gather}
where $\bbP^0_{n, \ell,r}\left[ \cdot ~|a, \lambda  \right] $ is the zero boundary measure defined in \eqref{eq:zerobp}.  The same expression holds for the free boundary measure $\bbP^f_{n, \ell,r}\left[ \cdot ~|a, \lambda  \right]$ defined in \eqref{eq:PolMeas02}.
\end{lemma}
The following remark will be useful in several later applications.
\begin{remark}\label{rem:mon}
Let $X^i_{n,T}$ denote the $i$-th line of the ensemble $\mu^f_{n,T}=\bbP^f_{n,-T,T}\left[ \cdot~|a, \lambda  \right]$.
The above lemma, combined with the monotonicity in Lemma \ref{lem:mono}, shows that $X^2_{n,T}$ is stochastically dominated by $\l^{-1/3}X^1_{n-1,\l^{2/3}T}$, even conditioned on $X^1_{n,T}$. This follows by first removing the ceiling $X^1_{n,T}$ imposed on $X^2_{n,T}$ and then by applying the scaling relation. More generally, by removing all top $k$ paths and applying $k$ times the scaling relation,  one has  that  for any $n>k\ge 1$, 
$X^{k+1}_{n,T}$ is stochastically dominated by $\l^{-k/3}X^1_{n-k,\l^{2k/3}T}$. The same statement continues to hold for $\mu^0_{n,T}.$
\end{remark}

\subsection{Confinement estimates}
Let $X^1_{n,T}$ denote the top path in the free boundary ensemble $\mu^f_{n,T}$. We use $[\cdot]_+$ to denote the positive part. 
The main confinement estimate from \cite{CIW18} 
can be rephrased as follows.
 \begin{theorem}\label{th:logtight}
There exists a constant $C>0$ such that for all $n\in\bbN$, $T>0$, 
\begin{equation}\label{eq:tight1}
\bbE\left[\max_{t\in[-T,T]}[X^1_{n,T}(t) - \psi(t)]_+
\right]\le C
\end{equation}
where $\psi(t)=C\log(1+|t|)$. In particular, for all $n\in\bbN$, and $T\ge S>0$,
\begin{equation}\label{eq:tight_max}
\bbE\left[\max_{t\in[-S,S]}X^1_{n,T}(t)
\right]\le C[1+\log(1+S)].
\end{equation}
\end{theorem}
The statement of Theorem \ref{th:logtight} was proven in \cite{CIW18} with the function $\psi(t)$ replaced by $|t|^\a$, where $\a>0$ can be taken arbitrarily small, see \cite[Theorem 3.1]{CIW18}. However, a careful check of the steps in that proof reveals that  the upgrading
presented in \eqref{eq:tight1} requires only minor adjustments; this is indeed a consequence of the exponential tail established in \cite[Section 3.3]{CIW18}. Moreover, minor modifications of the argument leading to \eqref{eq:tight1} can be shown to prove the slightly stronger bound: for every $S,T>0$, 
\begin{equation}\label{eq:tightao1}
\bbE\left[\max_{t\in[-S,T]}[X^1_{n;-S,T}(t) - \psi(t)]_+
\right]\le C\,,
\end{equation}
where $X^1_{n;-S,T}$ is the top path in the free boundary condition ensemble $\bbP^f_{n;-S,T }\left[ \cdot  \tc a, \lambda  \right]$ on the interval $[-S,T]$. 
As a corollary, we state bounds on the height of the $k$-th path $X^k_{n,T}(t)$ in the free boundary ensemble $\mu^f_{n,T}$, for any $k\ge 1$, which exploit the uniformity along translations expressed by the estimate \eqref{eq:tightao1}. The next statement follows directly from the stochastic domination in Remark \ref{rem:mon} and \eqref{eq:tightao1}. 
As a convention, we set $X^{k+1}_{n,T}(u)=0$ if $u\notin[-T,T]$. 
 \begin{corollary}\label{cor:logtight}
There exists a constant $C>0$ such that for all integers $k\ge 0$, $n\ge k+1$, and for all $T>0$, 
\begin{equation}\label{eq:tight2}
\sup_{t\in\bbR}\,\bbE\left[X^{k+1}_{n,T}(t)\right]\le C\l^{-k/3} .
\end{equation}
Moreover, for all $T\ge S>0$,
\begin{equation}\label{eq:tight2a}
\sup_{t\in\bbR}\,
\bbE\left[\max_{u\in[-S,S]}X^{k+1}_{n,T}(t+u) \right]\le C\l^{-k/3}[1+\log(1+|S\l^{2k/3}|)].
\end{equation}
\end{corollary}
\begin{remark}\label{freetozero}
Note that by monotonicity, 
Corolloary \ref{cor:logtight} applies to the zero boundary ensemble $\mu^0_{n,T}$ as well.  Moreover, by uniformity in $n,T$, and appealing for instance to Fatou's lemma,   the same uniform estimates continue to hold for the $(k+1)$-th line of the zero boundary $\l$-tilted LE $\mu^0$.
\end{remark}

Finally, for later reference we collect below the 
the main findings from \cite{CIW19} and \cite{DLZ} concerning infinite volume measures that we already alluded to in Section \ref{sec:stationary}.

\begin{theorem}\cite[Th.\,1.3, Th.\,1.4, Th.\,1.5]{CIW19}\label{tightLE12} For any $k,$ the joint law of $\{X^i_{n,T}\}_{1\le i\le k}$ under $\mu^f_{n,T}$ or $\mu^0_{n,T}$ forms a tight sequence as $n,T\to \infty$,
and any weak limit point is a Gibbs measure in the sense of Definition \ref{def:BG}. Moreover, $\mu^0_{n,T}$  has a well defined stationary limit $\mu^0$, the zero boundary $\l$-tilted LE. 
\end{theorem}

\begin{theorem}\cite[Th.\,1.1]{DLZ}\label{th:DLZ} For any $n\in\bbN$,  $\mu^f_{n,T}$ and $\mu^0_{n,T}$ have the same weak limit $\mu^0_n$, as $T\to \infty$. In particular, taking first $T\to\infty$ and then $n\to\infty$ one has the weak convergence of $\mu^f_{n,T}$ to the zero boundary LE  $\mu^0$.
\end{theorem}
 
Given the above preparation, we are now in a position to state our main results.
\section{Main results}\label{mainresults} We first start with the properties of $\mu^0$ before discussing other boundary conditions.

\subsection{Optimal tail behavior}
The first result concerns the one point tail behavior which will also serve as a crucial input in the proofs of several of the forthcoming results.
It is worth beginning by drawing an analogy with the known tail estimates for the top line of the Airy LE, namely the Airy$_2$ process. It is known that the one point distribution in this case is the GUE Tracy-Widom distribution $F_{TW}$ (see \cite{tracy1994level, ramirez2011beta}) which has the following tail behavior:
\begin{align}\label{twtails}
F_{TW}((t,\infty)) &= \exp(-\tfrac4{3} t^{3/2}+ O(\log t)), \,\,\text{as\,\,\,} t\to \infty,\\
F_{TW}((-\infty,-t)) &= \exp(-\tfrac1{12} t^{3}+ O(\log t)), \,\,\text{as\,\,\,} t\to \infty. 
\end{align}

Moreover, letting  $Y_{FS}(\cdot)$ denote the stationary  Ferrari-Spohn diffusion with one-point marginal given by \eqref{eq:airyfs}, the asymptotic formula for the Airy function ${\rm Ai}(z)=\exp{[(-2/3 +o(1))z^{3/2}]}$, as $z\to\infty$, 
shows that 
\begin{equation}\label{eq:domino1}
\bbP\left(Y_{FS}(0)> t\right) = \exp{\left(-\left(\tfrac{2\sqrt 2}3 
+ o(1)\right)t^{3/2}\right)}\,,\qquad t\to\infty.
\end{equation}

In the case of infinitely many lines, as recorded in Theorem \ref{th:logtight} and Corollary \ref{cor:logtight}, estimates from \cite{CIW18,CIW19} captured the right scale of fluctuation for the $\l$-tilted ensemble $\mu^0$. Here we establish an essentially optimal upper tail estimate for the height of the top path, showing that the Tracy-Widom/Ferrari-Spohn upper tail \eqref{eq:domino1} \emph{continues to hold at least up to first order even for $\mu^0$}, and using scaling arguments this extends to every subsequent path. 
On the other hand, the lower tail behavior for $\mu^0$ 
cannot be expected to be the same as that of Airy$_2$ because of the entropic repulsion induced by the hard floor constraint. For the Ferrari-Spohn distribution, it is not too difficult to see  that the probability of getting close to the floor decays as: 
\begin{equation}\label{FSlowertail}
\bbP(Y_{FS}(0)\le \varepsilon)\;\asymp\;
\varepsilon^{3} ,
\end{equation}
 since essentially it behaves in the same way as a Brownian excursion on a unit order interval. With multiple lines, we establish that the lower tail diminishes faster than polynomial; see Theorem \ref{lowertail}.

Let $X^i$ denote the $i$-th line in the zero boundary $\l$-tilted LE $\mu^0$. 
By monotonicity, $X^1(0)$ dominates $Y_{FS}(0)$, 
and therefore \eqref{eq:domino1} is a lower bound on the tail probability $\bbP\left(X^1(0)> t\right)$. The next result proves that, asymptotically, it is also an upper bound. 
\begin{theorem}[Upper tail]\label{th:stretched} 
The top line in the zero boundary $\l$-tilted LE satisfies  
\begin{equation}\label{eq:stretchedexp}
\bbP\left(X^1(0)> t\right)= \exp{\left(-\left(\tfrac{2\sqrt 2}3 
+ o(1)\right)t^{3/2}\right)}\,,\qquad t\to\infty.
\end{equation}
\end{theorem}

We remark that a minor 
refinement of our ideas would allow us to improve the above statement  with more precise information on the coefficient $c=\tfrac{2\sqrt 2}3 +o(1)$, by including logarithmic correction terms as in \eqref{twtails}. However,  we refrain from pursuing that for the sake of exposition. Also, as a straightforward consequence of Theorem \ref{th:stretched} and Remark \ref{rem:mon} one obtains the 
the following tail estimate for the $k$-th path.
\begin{equation}\label{eq:stretchedexpa}
\bbP\left(X^{k+1}(0)> t\l^{-k/3}\right)\le  \exp{\left(-\left(\tfrac{2\sqrt 2}3 
+ o(1)\right)t^{3/2}\right)}\,,\qquad t\to\infty.
\end{equation}
Note however that one should expect a sharper bound than the above for larger $k$, since the deviation $X^{k+1}(0)> t$ for the $(k+1)$-th curve forces the $k$ curves above it to deviate by the same amount. We record this improvement in the following corollary whose proof is also relatively straightforward from Theorem \ref{th:stretched}. 
\begin{corollary}\label{cor:stretched}
For any $\lambda >1,$ there exists an increasing sequence $c_k=c_k(\lambda)$ with  $c_0=\tfrac{2\sqrt 2}3$ and $c_k\to c_\infty = \tfrac{2\sqrt {2}}{3}\tfrac{\sqrt \l}{\sqrt {\l}-1}$, such that for all fixed $k\ge 0$,
the $(k+1)$-th line in the zero boundary $\l$-tilted LE satisfies 
\begin{equation}\label{eq:stretchedimproved}
\bbP\left(X^{k+1}(0)> t\l^{-k/3}\right)\le  \exp{\left(-\left(c_k 
+ o(1)\right)t^{3/2}\right)}\,,\qquad t\to\infty.
\end{equation}
\end{corollary}

Our next result provides a lower tail estimate quantifying the repulsion induced on the top path $X^1$ by the multiple lines below it. 
\begin{theorem}[Entropic repulsion]\label{lowertail} There exists a constant $C=C_\l>0$ such that 
$$\bbP(X^{1}(0)\le \varepsilon)\le \varepsilon^{C\log (1/\varepsilon)}. $$
\end{theorem}
It is worth noticing the contrast with the polynomial tail displayed by  the Ferrari-Spohn diffusion in \eqref{FSlowertail}, where the only repulsive effect is due to the hard wall.

\subsection{Ergodicity and mixing properties}
We next turn to the study of the ergodic properties of $\mu^0,$ and let the corresponding 
stationary, infinite line ensemble be $\uX$. For any $t\in\bbR $, let $T_t$ denote the shift operator  
defined by $T_t\uX (\cdot ) = \uX (t + \cdot)$. Let $\cB$ denote the Borel $\si$-field generated by the finite dimensional cylinder sets (see e.g., sets appearing later in \eqref{eq:monpis}), and write $T_t B = \{ T_{-t}\uX \in B\}$, $B\in\cB$. 

For every fixed $n$ the stationary line ensemble $\mu^0_n=\lim_{T\to\infty}\mu^0_{n,T}$ with $n$ lines is known to have an exponential decay of correlations, but no quantitative dependence on $n$ is known; see \cite{CIW18, DLZ}. Here we address the mixing and ergodic properties in the case of infinitely many lines. 

\begin{theorem}\label{th:mixing}
The line ensemble  $\mu^0$ is mixing, that is for every $A,B\in\cB$,
\begin{gather}\label{eq:ergo1} 
\lim_{t\to\infty}\mu^0
\left(T_t A \cap B\right) = \mu^0
\left(A \right)\mu^0\left( B\right)  
\,.
\end{gather} 
In particular, $\mu^0$ is ergodic. 
\end{theorem}

The corresponding result for the Airy LE was established in \cite{corwin2014ergodicity}.
A natural quantitative version of the above result would seek to establish the decay of correlation of, say to begin with, the top line. Parallel inquiries have been undertaken in the case of exactly solvable models; see for instance \cite{widom2004asymptotics, shinault2011asymptotics} for sharp correlation estimates for the Airy$_2$ process. A more recent work \cite{basu2022exponent} considers the related Airy$_1$ process. Towards this we have the next result. 
\begin{theorem}\label{th:corrdecay}
There exist constants $c,C>0$ such that, for all $t>0$,
\begin{equation}\label{eq:corre1}
\left|\cov\left(X^1(0),X^1(t)\right)\right|\le C\exp\left[-c(\log t)^{3/7}\right].
\end{equation}
\end{theorem}
Here we use the notation $\cov\left(f,g\right)=\bbE[fg]-\bbE[f]\bbE[g]$ for the covariance. We note that by translation invariance  Theorem \ref{th:stretched} implies that all moments of $X^1(s)$ are finite for all $s\in\bbR$. 
In particular, the covariance in \eqref{eq:corre1} is well defined. Our proof will in fact show that the above estimate holds for $X^{k}(\cdot)$ for any fixed $k.$

Note that the above is significantly weaker than the exponential decay of correlation exhibited by the Ferrari-Spohn diffusion. While one may speculate on the basis of geometrically decaying nature of the curves, that a similar correlation structure might be expected for $\mu^0$ as well, pinning down the correct order of correlations in this case remains an attractive open problem. 

\subsection{Uniqueness}
Classifying Gibbs measures and in certain cases proving their uniqueness is an important problem in statistical mechanics. As already alluded to in the introduction, taking limits of finite volume measures with varying boundary conditions is a natural recipe to construct infinite volume measures which leads to important questions about whether the limits depend on the sequence of boundary conditions. For instance, for the $\l$-tilted LE, the question of uniqueness of the limit points of the free boundary measures $ \mu^f_{n,T}$ as $n,T \to \infty$ has received some attention \cite{CIW19,DLZ}. Our final main result addresses these questions in rather general terms. 
We begin with a definition. Recall that a $\l-$tilted LE is a BG measure as in Definition \ref{def:BG}.
\begin{definition}\label{ut} A $\l-$tilted LE  $\uX=\{X^i\}_{i\ge 1}$ is said to be \emph{uniformly tight} (UT) if, 
given any $\varepsilon>0$, there exists $C>0$ such that, \emph{for all} $t>0$ 
$$\bbP(X^{1}(t)\ge C)\le \varepsilon.$$
\end{definition}
Note that while any stationary line ensemble is by definition uniformly tight, the above definition also allows non-stationary ensembles.
Finally, recall the notion of being asymptotically pinned to zero from Definition \ref{def:asympin}.
\begin{theorem}\label{th:uniqueness}
Suppose $\nu$ is the law of a UT $\lambda-$tilted LE which is also asymptotically pinned to zero. 
Then $\nu=\mu^0$. 
\end{theorem}
An obvious consequence of the above theorem is that there exists a unique stationary, asymptotically pinned to zero, $\lambda-$tilted LE. 

Finally, we record that a straightforward consequence of the above result is the uniqueness of the line ensemble with free boundary conditions.
\begin{corollary}\label{freeunique} The finite LE with free boundary conditions $\mu^f_{n,T}$ from \eqref{eq:notation} satisfy $$\lim_{n, T\to \infty}\mu^f_{n,T}=\mu^{0},$$
with $n, T\to \infty$ arbitrarily. 
\end{corollary}

Note that it was mentioned in the discussion preceding Definition \ref{def:asympin} that for any bounded continuous function $g,$  a BG measure with floor $g$ could be made sense of if $g$ is smooth enough and bounded. This would then be UT as well. Further, such a measure could also be  constructed if $g$ is not growing too fast. In that case it would yield a BG but a non-UT measure.  However, perhaps more intriguing is the possibility of the existence of non-UT, yet asymptotically pinned to zero,  $\lambda-$tilted LEs. We anticipate that our arguments are robust enough to 
treat such cases as well,
and this will be the subject of forthcoming work.  

\subsection{Proof ideas}\label{iop} We end this section with a brief overview of the proofs. As already indicated, in contrast to past work on Airy LE, in absence of any integrability, our arguments do not have access to algebraic inputs and are completely probabilistic in nature. In the upcoming section, we present a key observation which is the starting point of many of our arguments, and already in its vanilla form yields a quick alternate proof of Theorem \ref{th:DLZ}.

In its core the observation is simple. Note that by monotonicity, for any $n, T,$ there exists a coupling of the measures $\mu^f_{n,T}$ and  $\mu^0_{n,T}$  such that deterministically the former dominates the latter. The crucial observation then is that one can also simultaneously construct another coupling such that the \emph{latter dominates the former}, at least on a given compact interval, say around the origin, with high probability, provided that $T$ is large enough compared to $n$. To accomplish this we sample two independent copies of $\mu^f_{n,T}$ and  $\mu^0_{n,T}$. Given the samples, all one needs to ensure is the existence of a stopping domain (see Figure \ref{fig:key105}) 
where the boundary data for the latter dominates the former. Since then one can simply resample both  ensembles on this domain (the strong BG property from Section \ref{bgsec} allows this) under the monotone coupling and on account of the reverse ordering of the boundary data, the zero boundary ensemble will dominate the free boundary counterpart.

\begin{figure}[h]
\centering
\begin{tikzpicture}[scale=0.7]
\colorlet{colora}{magenta}
\colorlet{colorb}{teal}
\draw[dashed,thick] (-2.47475,-1.1) -- (-2.47475,4);
\draw[dashed,thick] (1.9697,-1.1) -- (1.9697,4);
\draw[thick] (-5,-0.5) -- (5,-0.5);
\draw (-2.77475,-0.8) node {$\tau_\ell$};
\draw (2.2697,-0.8) node {$\tau_r$};
\draw[thick,colora] (-5,0) -- (-4.89899,-0.07914) -- (-4.79798,-0.02241) -- (-4.69697,-0.05041) -- (-4.59596,0.00655) -- (-4.49495,-0.02758) -- (-4.39394,0.07223) -- (-4.29293,0.0737) -- (-4.19192,0.14667) -- (-4.09091,-0.03266) -- (-3.9899,-0.01653) -- (-3.88889,-0.03666) -- (-3.78788,-0.05354) -- (-3.68687,-0.02013) -- (-3.58586,0.05229) -- (-3.48485,0.08212) -- (-3.38384,0.07378) -- (-3.28283,0.15761) -- (-3.18182,0.21682) -- (-3.08081,0.15884) -- (-2.9798,0.2592) -- (-2.87879,0.27165) -- (-2.77778,0.20364) -- (-2.67677,0.12366) -- (-2.57576,0.08241) -- (-2.47475,0.22533) -- (-2.37374,0.09353) -- (-2.27273,0.09352) -- (-2.17172,0.08446) -- (-2.07071,0.12118) -- (-1.9697,0.0675) -- (-1.86869,0.09931) -- (-1.76768,0.18577) -- (-1.66667,0.20232) -- (-1.56566,0.17366) -- (-1.46465,0.09764) -- (-1.36364,0.17895) -- (-1.26263,0.29786) -- (-1.16162,0.30819) -- (-1.06061,0.2885) -- (-0.9596,0.30159) -- (-0.85859,0.22901) -- (-0.75758,0.16479) -- (-0.65657,0.18196) -- (-0.55556,0.20894) -- (-0.45455,0.13651) -- (-0.35354,0.15288) -- (-0.25253,0.18165) -- (-0.15152,0.16866) -- (-0.05051,0.16932) -- (0.05051,0.09739) -- (0.15152,0.16887) -- (0.25253,0.29791) -- (0.35354,0.35609) -- (0.45455,0.3333) -- (0.55556,0.27764) -- (0.65657,0.24499) -- (0.75758,0.33675) -- (0.85859,0.24113) -- (0.9596,0.23521) -- (1.06061,0.16546) -- (1.16162,0.17603) -- (1.26263,0.14545) -- (1.36364,0.09195) -- (1.46465,0.11509) -- (1.56566,0.18776) -- (1.66667,0.11902) -- (1.76768,0.00407) -- (1.86869,0.04758) -- (1.9697,0.16003) -- (2.07071,0.29282) -- (2.17172,0.36183) -- (2.27273,0.21199) -- (2.37374,0.1817) -- (2.47475,0.1505) -- (2.57576,0.07604) -- (2.67677,0.08291) -- (2.77778,0.0168) -- (2.87879,-0.03994) -- (2.9798,-0.0222) -- (3.08081,-0.15706) -- (3.18182,-0.14493) -- (3.28283,-0.17105) -- (3.38384,-0.05926) -- (3.48485,-0.13212) -- (3.58586,-0.05611) -- (3.68687,-0.26411) -- (3.78788,-0.2938) -- (3.88889,-0.20986) -- (3.9899,-0.2011) -- (4.09091,-0.17243) -- (4.19192,-0.14853) -- (4.29293,-0.17532) -- (4.39394,-0.20191) -- (4.49495,-0.14987) -- (4.59596,-0.08922) -- (4.69697,0.05924) -- (4.79798,0.0801) -- (4.89899,-0.01881) -- (5,0.06988);
\draw[thick,colorb] (-5,0) -- (-4.89899,-0.03513) -- (-4.79798,-0.12263) -- (-4.69697,-0.04131) -- (-4.59596,-0.19071) -- (-4.49495,-0.21967) -- (-4.39394,-0.23424) -- (-4.29293,-0.29077) -- (-4.19192,-0.39942) -- (-4.09091,-0.41873) -- (-3.9899,-0.28739) -- (-3.88889,-0.18827) -- (-3.78788,-0.14035) -- (-3.68687,-0.16009) -- (-3.58586,-0.149) -- (-3.48485,-0.10989) -- (-3.38384,-0.02762) -- (-3.28283,0.02569) -- (-3.18182,-0.07797) -- (-3.08081,-0.07541) -- (-2.9798,-0.12636) -- (-2.87879,-0.1086) -- (-2.77778,-0.18277) -- (-2.67677,-0.12656) -- (-2.57576,-0.17731) -- (-2.47475,-0.05998) -- (-2.37374,-0.14761) -- (-2.27273,-0.09458) -- (-2.17172,-0.16609) -- (-2.07071,-0.10656) -- (-1.9697,-0.03819) -- (-1.86869,-0.13326) -- (-1.76768,-0.15349) -- (-1.66667,-0.12581) -- (-1.56566,-0.02323) -- (-1.46465,-0.05432) -- (-1.36364,-0.07715) -- (-1.26263,-0.0366) -- (-1.16162,-0.05863) -- (-1.06061,-0.14982) -- (-0.9596,-0.14343) -- (-0.85859,-0.13871) -- (-0.75758,-0.01581) -- (-0.65657,0.07262) -- (-0.55556,0.02843) -- (-0.45455,-0.05262) -- (-0.35354,-0.06262) -- (-0.25253,-0.09183) -- (-0.15152,-0.0377) -- (-0.05051,-0.14965) -- (0.05051,-0.1356) -- (0.15152,-0.03984) -- (0.25253,-0.12542) -- (0.35354,-0.14374) -- (0.45455,-0.0633) -- (0.55556,-0.03828) -- (0.65657,0.18303) -- (0.75758,0.3365) -- (0.85859,0.40148) -- (0.9596,0.36731) -- (1.06061,0.30856) -- (1.16162,0.27512) -- (1.26263,0.33319) -- (1.36364,0.20503) -- (1.46465,0.12227) -- (1.56566,0.09331) -- (1.66667,0.09731) -- (1.76768,-0.01829) -- (1.86869,-0.09596) -- (1.9697,-0.02658) -- (2.07071,0.05377) -- (2.17172,-0.00878) -- (2.27273,0.15913) -- (2.37374,0.1617) -- (2.47475,0.25056) -- (2.57576,0.22359) -- (2.67677,0.32266) -- (2.77778,0.25263) -- (2.87879,0.18062) -- (2.9798,0.24331) -- (3.08081,0.08891) -- (3.18182,-0.02803) -- (3.28283,-0.13581) -- (3.38384,-0.13902) -- (3.48485,-0.12597) -- (3.58586,-0.21186) -- (3.68687,-0.24398) -- (3.78788,-0.17209) -- (3.88889,-0.20537) -- (3.9899,-0.18179) -- (4.09091,-0.13994) -- (4.19192,-0.15419) -- (4.29293,-0.16685) -- (4.39394,-0.22336) -- (4.49495,-0.22277) -- (4.59596,-0.18498) -- (4.69697,-0.21961) -- (4.79798,-0.18946) -- (4.89899,-0.09895) -- (5,-0.28498);
\fill[fill=magenta] (-2.47475,0.22533) circle[radius=0.07];
\fill[fill=magenta] (1.9697,0.16003) circle[radius=0.07];
\fill[fill=teal] (-2.47475,-0.05998) circle[radius=0.07];
\fill[fill=teal] (1.9697,-0.02658) circle[radius=0.07];
\draw[thick,colora] (-5,1) -- (-4.89899,0.95811) -- (-4.79798,0.97491) -- (-4.69697,0.99605) -- (-4.59596,0.96936) -- (-4.49495,1.0253) -- (-4.39394,0.95602) -- (-4.29293,1.09908) -- (-4.19192,1.01587) -- (-4.09091,1.08926) -- (-3.9899,1.09107) -- (-3.88889,1.09796) -- (-3.78788,1.15283) -- (-3.68687,1.18289) -- (-3.58586,1.10533) -- (-3.48485,1.14243) -- (-3.38384,1.12249) -- (-3.28283,1.08819) -- (-3.18182,1.11577) -- (-3.08081,1.15754) -- (-2.9798,1.10094) -- (-2.87879,1.09887) -- (-2.77778,1.01731) -- (-2.67677,0.93324) -- (-2.57576,1.01365) -- (-2.47475,1.09816) -- (-2.37374,1.05757) -- (-2.27273,0.94565) -- (-2.17172,0.96451) -- (-2.07071,0.74661) -- (-1.9697,0.69311) -- (-1.86869,0.77789) -- (-1.76768,0.65556) -- (-1.66667,0.702) -- (-1.56566,0.77557) -- (-1.46465,0.90679) -- (-1.36364,0.79122) -- (-1.26263,0.70657) -- (-1.16162,0.70933) -- (-1.06061,0.79532) -- (-0.9596,0.80877) -- (-0.85859,0.8294) -- (-0.75758,1.02649) -- (-0.65657,0.99607) -- (-0.55556,1.01342) -- (-0.45455,0.91333) -- (-0.35354,0.95574) -- (-0.25253,0.89276) -- (-0.15152,0.84744) -- (-0.05051,0.97245) -- (0.05051,0.93757) -- (0.15152,1.03611) -- (0.25253,1.0603) -- (0.35354,1.04985) -- (0.45455,0.9872) -- (0.55556,0.9833) -- (0.65657,1.10922) -- (0.75758,1.14541) -- (0.85859,1.10881) -- (0.9596,1.09339) -- (1.06061,1.01804) -- (1.16162,1.09376) -- (1.26263,1.05955) -- (1.36364,1.07203) -- (1.46465,1.0204) -- (1.56566,0.8925) -- (1.66667,0.9778) -- (1.76768,1.06593) -- (1.86869,1.20209) -- (1.9697,1.19047) -- (2.07071,1.19173) -- (2.17172,1.0141) -- (2.27273,0.96793) -- (2.37374,1.05082) -- (2.47475,1.11656) -- (2.57576,1.05146) -- (2.67677,0.98389) -- (2.77778,1.03852) -- (2.87879,1.04246) -- (2.9798,1.1166) -- (3.08081,1.01988) -- (3.18182,1.03532) -- (3.28283,1.04575) -- (3.38384,0.969) -- (3.48485,1.03654) -- (3.58586,1.17411) -- (3.68687,1.14241) -- (3.78788,1.25302) -- (3.88889,1.28298) -- (3.9899,1.13516) -- (4.09091,1.06261) -- (4.19192,1.04696) -- (4.29293,1.12372) -- (4.39394,1.1081) -- (4.49495,1.2323) -- (4.59596,1.20369) -- (4.69697,1.24253) -- (4.79798,1.27642) -- (4.89899,1.28739) -- (5,1.27768);
\draw[thick,colorb] (-5,1) -- (-4.89899,0.90614) -- (-4.79798,0.73849) -- (-4.69697,0.79498) -- (-4.59596,0.78689) -- (-4.49495,0.65636) -- (-4.39394,0.65322) -- (-4.29293,0.73808) -- (-4.19192,0.68456) -- (-4.09091,0.71279) -- (-3.9899,0.67792) -- (-3.88889,0.79087) -- (-3.78788,0.89769) -- (-3.68687,1.00866) -- (-3.58586,0.9729) -- (-3.48485,0.86579) -- (-3.38384,0.91158) -- (-3.28283,1.01724) -- (-3.18182,0.96226) -- (-3.08081,0.97263) -- (-2.9798,0.97131) -- (-2.87879,1.09551) -- (-2.77778,0.9801) -- (-2.67677,0.91527) -- (-2.57576,0.91239) -- (-2.47475,0.90417) -- (-2.37374,0.85682) -- (-2.27273,0.86075) -- (-2.17172,0.68361) -- (-2.07071,0.61908) -- (-1.9697,0.62751) -- (-1.86869,0.71871) -- (-1.76768,0.72402) -- (-1.66667,0.78098) -- (-1.56566,0.77528) -- (-1.46465,0.74518) -- (-1.36364,0.67442) -- (-1.26263,0.6932) -- (-1.16162,0.81493) -- (-1.06061,0.83021) -- (-0.9596,0.87436) -- (-0.85859,0.95014) -- (-0.75758,0.98457) -- (-0.65657,0.90347) -- (-0.55556,0.89635) -- (-0.45455,0.91329) -- (-0.35354,0.93259) -- (-0.25253,0.93566) -- (-0.15152,1.0016) -- (-0.05051,0.92232) -- (0.05051,0.96476) -- (0.15152,0.9523) -- (0.25253,1.03983) -- (0.35354,1.132) -- (0.45455,1.221) -- (0.55556,1.15825) -- (0.65657,1.05504) -- (0.75758,1.1188) -- (0.85859,1.01741) -- (0.9596,0.99756) -- (1.06061,0.9674) -- (1.16162,0.85844) -- (1.26263,0.85052) -- (1.36364,0.86771) -- (1.46465,0.79795) -- (1.56566,0.7866) -- (1.66667,0.83769) -- (1.76768,0.86486) -- (1.86869,1.00843) -- (1.9697,1.03535) -- (2.07071,1.09312) -- (2.17172,1.04077) -- (2.27273,1.03258) -- (2.37374,1.08002) -- (2.47475,1.12456) -- (2.57576,1.17326) -- (2.67677,1.25912) -- (2.77778,1.24291) -- (2.87879,1.20194) -- (2.9798,1.3128) -- (3.08081,1.33978) -- (3.18182,1.27716) -- (3.28283,1.36171) -- (3.38384,1.31857) -- (3.48485,1.23505) -- (3.58586,1.20327) -- (3.68687,1.2487) -- (3.78788,1.15019) -- (3.88889,1.04962) -- (3.9899,1.11691) -- (4.09091,1.11014) -- (4.19192,0.96222) -- (4.29293,0.95917) -- (4.39394,0.96685) -- (4.49495,1.11247) -- (4.59596,1.12363) -- (4.69697,1.21684) -- (4.79798,1.21563) -- (4.89899,1.1747) -- (5,1.15081);
\fill[fill=magenta] (-2.47475,1.09816) circle[radius=0.07];
\fill[fill=magenta] (1.9697,1.19047) circle[radius=0.07];
\fill[fill=teal] (-2.47475,0.90417) circle[radius=0.07];
\fill[fill=teal] (1.9697,1.03535) circle[radius=0.07];
\draw[thick,colora] (-5,2) -- (-4.89899,1.92105) -- (-4.79798,1.98195) -- (-4.69697,2.17228) -- (-4.59596,2.18787) -- (-4.49495,2.23316) -- (-4.39394,2.35502) -- (-4.29293,2.32874) -- (-4.19192,2.32405) -- (-4.09091,2.29398) -- (-3.9899,2.27969) -- (-3.88889,2.16662) -- (-3.78788,2.1358) -- (-3.68687,2.14898) -- (-3.58586,2.12682) -- (-3.48485,2.05589) -- (-3.38384,2.01989) -- (-3.28283,1.95135) -- (-3.18182,1.98941) -- (-3.08081,2.00954) -- (-2.9798,1.97954) -- (-2.87879,1.81824) -- (-2.77778,1.87033) -- (-2.67677,1.97199) -- (-2.57576,1.94668) -- (-2.47475,2.0082) -- (-2.37374,2.04592) -- (-2.27273,2.12169) -- (-2.17172,2.04472) -- (-2.07071,1.90743) -- (-1.9697,1.93065) -- (-1.86869,1.89127) -- (-1.76768,1.84453) -- (-1.66667,1.83955) -- (-1.56566,1.8769) -- (-1.46465,1.94072) -- (-1.36364,1.87982) -- (-1.26263,1.97841) -- (-1.16162,1.96444) -- (-1.06061,1.96607) -- (-0.9596,1.925) -- (-0.85859,1.94543) -- (-0.75758,1.81979) -- (-0.65657,1.78995) -- (-0.55556,1.83491) -- (-0.45455,1.86954) -- (-0.35354,1.86723) -- (-0.25253,1.7912) -- (-0.15152,1.80721) -- (-0.05051,1.76762) -- (0.05051,1.84407) -- (0.15152,1.88295) -- (0.25253,1.93204) -- (0.35354,1.87644) -- (0.45455,1.96421) -- (0.55556,2.09481) -- (0.65657,2.06888) -- (0.75758,1.92569) -- (0.85859,1.89207) -- (0.9596,1.85621) -- (1.06061,1.92705) -- (1.16162,2.04612) -- (1.26263,2.16414) -- (1.36364,2.12964) -- (1.46465,2.11536) -- (1.56566,2.13088) -- (1.66667,2.05142) -- (1.76768,2.05) -- (1.86869,2.08138) -- (1.9697,2.10271) -- (2.07071,2.14494) -- (2.17172,2.0135) -- (2.27273,2.00346) -- (2.37374,1.86928) -- (2.47475,1.89125) -- (2.57576,1.83225) -- (2.67677,1.77149) -- (2.77778,1.79649) -- (2.87879,1.7203) -- (2.9798,1.79316) -- (3.08081,1.88618) -- (3.18182,1.84822) -- (3.28283,1.76733) -- (3.38384,1.82243) -- (3.48485,1.85982) -- (3.58586,1.93749) -- (3.68687,1.99624) -- (3.78788,2.01309) -- (3.88889,2.04876) -- (3.9899,2.00743) -- (4.09091,1.93172) -- (4.19192,1.83035) -- (4.29293,1.83953) -- (4.39394,1.89433) -- (4.49495,1.75597) -- (4.59596,1.80239) -- (4.69697,1.78571) -- (4.79798,1.67782) -- (4.89899,1.65453) -- (5,1.78188);
\draw[thick,colorb] (-5,2) -- (-4.89899,1.97783) -- (-4.79798,1.91389) -- (-4.69697,2.08033) -- (-4.59596,2.07625) -- (-4.49495,2.1533) -- (-4.39394,2.00684) -- (-4.29293,2.0066) -- (-4.19192,1.97531) -- (-4.09091,1.98428) -- (-3.9899,2.05416) -- (-3.88889,2.11913) -- (-3.78788,2.32307) -- (-3.68687,2.45284) -- (-3.58586,2.40648) -- (-3.48485,2.37392) -- (-3.38384,2.24361) -- (-3.28283,2.20925) -- (-3.18182,2.17734) -- (-3.08081,2.20224) -- (-2.9798,1.99398) -- (-2.87879,1.99471) -- (-2.77778,1.97932) -- (-2.67677,1.8988) -- (-2.57576,1.75118) -- (-2.47475,1.89352) -- (-2.37374,1.89333) -- (-2.27273,1.92452) -- (-2.17172,1.97761) -- (-2.07071,1.95747) -- (-1.9697,1.94261) -- (-1.86869,1.90626) -- (-1.76768,1.84006) -- (-1.66667,1.88481) -- (-1.56566,1.85922) -- (-1.46465,1.91215) -- (-1.36364,2.13676) -- (-1.26263,2.26335) -- (-1.16162,2.16894) -- (-1.06061,2.11045) -- (-0.9596,2.03307) -- (-0.85859,1.94266) -- (-0.75758,1.97753) -- (-0.65657,1.97939) -- (-0.55556,2.11203) -- (-0.45455,2.09778) -- (-0.35354,2.14652) -- (-0.25253,2.2154) -- (-0.15152,2.11444) -- (-0.05051,2.08588) -- (0.05051,2.06786) -- (0.15152,2.07819) -- (0.25253,2.16331) -- (0.35354,2.06151) -- (0.45455,2.07223) -- (0.55556,2.0045) -- (0.65657,2.11991) -- (0.75758,2.00274) -- (0.85859,1.92489) -- (0.9596,2.04343) -- (1.06061,2.01713) -- (1.16162,2.03595) -- (1.26263,1.99021) -- (1.36364,2.05644) -- (1.46465,1.87121) -- (1.56566,1.89933) -- (1.66667,1.93826) -- (1.76768,1.95883) -- (1.86869,1.98948) -- (1.9697,1.91426) -- (2.07071,1.95974) -- (2.17172,1.89747) -- (2.27273,1.79465) -- (2.37374,1.75373) -- (2.47475,1.61105) -- (2.57576,1.64544) -- (2.67677,1.80823) -- (2.77778,1.77022) -- (2.87879,1.8925) -- (2.9798,1.95956) -- (3.08081,1.97212) -- (3.18182,2.0396) -- (3.28283,2.03395) -- (3.38384,2.06828) -- (3.48485,2.07288) -- (3.58586,2.08699) -- (3.68687,2.00482) -- (3.78788,2.061) -- (3.88889,2.04703) -- (3.9899,2.04235) -- (4.09091,2.02793) -- (4.19192,1.99598) -- (4.29293,1.97776) -- (4.39394,1.97936) -- (4.49495,2.03649) -- (4.59596,2.00169) -- (4.69697,1.97747) -- (4.79798,1.97051) -- (4.89899,1.98188) -- (5,1.95588);
\fill[fill=magenta] (-2.47475,2.0082) circle[radius=0.07];
\fill[fill=magenta] (1.9697,2.10271) circle[radius=0.07];
\fill[fill=teal] (-2.47475,1.89352) circle[radius=0.07];
\fill[fill=teal] (1.9697,1.91426) circle[radius=0.07];
\draw[thick,colora] (-5,3) -- (-4.89899,2.94297) -- (-4.79798,2.8128) -- (-4.69697,2.833) -- (-4.59596,2.88338) -- (-4.49495,2.7553) -- (-4.39394,2.78597) -- (-4.29293,2.82327) -- (-4.19192,2.90075) -- (-4.09091,2.81606) -- (-3.9899,2.93204) -- (-3.88889,3.07821) -- (-3.78788,3.0485) -- (-3.68687,3.17675) -- (-3.58586,3.1403) -- (-3.48485,3.08373) -- (-3.38384,3.09733) -- (-3.28283,3.15799) -- (-3.18182,3.19138) -- (-3.08081,3.15084) -- (-2.9798,3.13138) -- (-2.87879,3.01541) -- (-2.77778,2.99298) -- (-2.67677,3.09027) -- (-2.57576,3.14074) -- (-2.47475,3.10858) -- (-2.37374,3.1101) -- (-2.27273,3.13348) -- (-2.17172,3.06442) -- (-2.07071,3.09771) -- (-1.9697,3.08069) -- (-1.86869,3.16434) -- (-1.76768,3.09395) -- (-1.66667,3.11375) -- (-1.56566,3.08835) -- (-1.46465,3.0783) -- (-1.36364,3.0507) -- (-1.26263,3.07824) -- (-1.16162,3.13911) -- (-1.06061,3.20937) -- (-0.9596,3.15177) -- (-0.85859,2.98928) -- (-0.75758,3.03338) -- (-0.65657,2.97014) -- (-0.55556,2.96673) -- (-0.45455,2.88349) -- (-0.35354,2.99853) -- (-0.25253,2.99899) -- (-0.15152,3.01379) -- (-0.05051,3.00951) -- (0.05051,2.9861) -- (0.15152,2.89459) -- (0.25253,2.97019) -- (0.35354,2.90423) -- (0.45455,2.87049) -- (0.55556,2.84231) -- (0.65657,2.90586) -- (0.75758,2.92025) -- (0.85859,2.92935) -- (0.9596,3.01649) -- (1.06061,2.85483) -- (1.16162,2.921) -- (1.26263,2.98176) -- (1.36364,2.93968) -- (1.46465,2.86155) -- (1.56566,2.87255) -- (1.66667,2.91731) -- (1.76768,3.06234) -- (1.86869,3.10392) -- (1.9697,3.16457) -- (2.07071,3.12164) -- (2.17172,3.10441) -- (2.27273,3.19174) -- (2.37374,3.1243) -- (2.47475,3.09522) -- (2.57576,2.96667) -- (2.67677,2.98204) -- (2.77778,2.98577) -- (2.87879,3.0284) -- (2.9798,3.0982) -- (3.08081,3.13808) -- (3.18182,3.09762) -- (3.28283,3.11817) -- (3.38384,3.04829) -- (3.48485,3.05073) -- (3.58586,3.01549) -- (3.68687,3.02829) -- (3.78788,2.87358) -- (3.88889,2.88586) -- (3.9899,2.89037) -- (4.09091,2.75496) -- (4.19192,2.87104) -- (4.29293,2.80198) -- (4.39394,2.85684) -- (4.49495,3.03303) -- (4.59596,3.00562) -- (4.69697,2.93692) -- (4.79798,2.95515) -- (4.89899,3.00088) -- (5,2.89474);
\draw[thick,colorb] (-5,3) -- (-4.89899,3.04426) -- (-4.79798,3.0339) -- (-4.69697,3.10163) -- (-4.59596,3.00095) -- (-4.49495,3.11372) -- (-4.39394,3.13436) -- (-4.29293,3.09077) -- (-4.19192,3.0227) -- (-4.09091,2.96097) -- (-3.9899,2.98766) -- (-3.88889,2.90391) -- (-3.78788,2.93068) -- (-3.68687,2.99752) -- (-3.58586,3.03593) -- (-3.48485,3.03516) -- (-3.38384,3.05231) -- (-3.28283,2.96511) -- (-3.18182,2.94408) -- (-3.08081,2.90698) -- (-2.9798,2.88833) -- (-2.87879,2.85712) -- (-2.77778,2.95143) -- (-2.67677,2.97224) -- (-2.57576,2.90848) -- (-2.47475,2.87402) -- (-2.37374,3.00444) -- (-2.27273,3.0884) -- (-2.17172,3.03844) -- (-2.07071,2.99361) -- (-1.9697,3.0903) -- (-1.86869,2.9978) -- (-1.76768,2.9279) -- (-1.66667,2.92517) -- (-1.56566,2.88424) -- (-1.46465,2.78818) -- (-1.36364,2.84633) -- (-1.26263,2.74373) -- (-1.16162,2.75766) -- (-1.06061,2.82685) -- (-0.9596,2.87887) -- (-0.85859,2.81772) -- (-0.75758,2.91769) -- (-0.65657,2.91315) -- (-0.55556,2.89186) -- (-0.45455,2.90504) -- (-0.35354,2.82232) -- (-0.25253,2.87945) -- (-0.15152,2.87613) -- (-0.05051,2.88327) -- (0.05051,2.83039) -- (0.15152,2.83185) -- (0.25253,2.81362) -- (0.35354,2.76375) -- (0.45455,2.65685) -- (0.55556,2.71922) -- (0.65657,2.81531) -- (0.75758,2.8255) -- (0.85859,2.85869) -- (0.9596,2.69301) -- (1.06061,2.74582) -- (1.16162,2.80765) -- (1.26263,2.89149) -- (1.36364,2.84779) -- (1.46465,2.80026) -- (1.56566,2.73391) -- (1.66667,2.77422) -- (1.76768,2.9602) -- (1.86869,2.89192) -- (1.9697,2.81848) -- (2.07071,2.80998) -- (2.17172,2.92835) -- (2.27273,2.91033) -- (2.37374,2.86712) -- (2.47475,2.86637) -- (2.57576,2.91476) -- (2.67677,2.89957) -- (2.77778,2.84981) -- (2.87879,2.8016) -- (2.9798,2.69138) -- (3.08081,2.6259) -- (3.18182,2.61733) -- (3.28283,2.71827) -- (3.38384,2.78428) -- (3.48485,2.81179) -- (3.58586,2.86994) -- (3.68687,2.92897) -- (3.78788,2.97864) -- (3.88889,3.02978) -- (3.9899,2.99245) -- (4.09091,3.04467) -- (4.19192,3.00238) -- (4.29293,3.11621) -- (4.39394,3.0664) -- (4.49495,3.06023) -- (4.59596,3.15167) -- (4.69697,3.06794) -- (4.79798,3.12199) -- (4.89899,3.121) -- (5,3.037);
\fill[fill=magenta] (-2.47475,3.10858) circle[radius=0.07];
\fill[fill=magenta] (1.9697,3.16457) circle[radius=0.07];
\fill[fill=teal] (-2.47475,2.87402) circle[radius=0.07];
\fill[fill=teal] (1.9697,2.81848) circle[radius=0.07];
\end{tikzpicture}%
\qquad%
\begin{tikzpicture}[scale=0.7]
\colorlet{colora}{magenta}
\colorlet{colorb}{teal}
\colorlet{colorc}{orange}
\colorlet{colord}{blue}
\draw[dashed,thick] (-2.47475,-1.1) -- (-2.47475,4);
\draw[dashed,thick] (1.9697,-1.1) -- (1.9697,4);
\draw[thick] (-5,-0.5) -- (5,-0.5);
\draw (-2.77475,-0.8) node {$\tau_\ell$};
\draw (2.2697,-0.8) node {$\tau_r$};
\draw[thick,colora] (-5,0) -- (-4.89899,-0.07914) -- (-4.79798,-0.02241) -- (-4.69697,-0.05041) -- (-4.59596,0.00655) -- (-4.49495,-0.02758) -- (-4.39394,0.07223) -- (-4.29293,0.0737) -- (-4.19192,0.14667) -- (-4.09091,-0.03266) -- (-3.9899,-0.01653) -- (-3.88889,-0.03666) -- (-3.78788,-0.05354) -- (-3.68687,-0.02013) -- (-3.58586,0.05229) -- (-3.48485,0.08212) -- (-3.38384,0.07378) -- (-3.28283,0.15761) -- (-3.18182,0.21682) -- (-3.08081,0.15884) -- (-2.9798,0.2592) -- (-2.87879,0.27165) -- (-2.77778,0.20364) -- (-2.67677,0.12366) -- (-2.57576,0.08241) -- (-2.47475,0.22533);
\draw[thick,colorb] (-5,0) -- (-4.89899,-0.03513) -- (-4.79798,-0.12263) -- (-4.69697,-0.04131) -- (-4.59596,-0.19071) -- (-4.49495,-0.21967) -- (-4.39394,-0.23424) -- (-4.29293,-0.29077) -- (-4.19192,-0.39942) -- (-4.09091,-0.41873) -- (-3.9899,-0.28739) -- (-3.88889,-0.18827) -- (-3.78788,-0.14035) -- (-3.68687,-0.16009) -- (-3.58586,-0.149) -- (-3.48485,-0.10989) -- (-3.38384,-0.02762) -- (-3.28283,0.02569) -- (-3.18182,-0.07797) -- (-3.08081,-0.07541) -- (-2.9798,-0.12636) -- (-2.87879,-0.1086) -- (-2.77778,-0.18277) -- (-2.67677,-0.12656) -- (-2.57576,-0.17731) -- (-2.47475,-0.05998);
\draw[thick,colora] (1.9697,0.16003) -- (2.07071,0.29282) -- (2.17172,0.36183) -- (2.27273,0.21199) -- (2.37374,0.1817) -- (2.47475,0.1505) -- (2.57576,0.07604) -- (2.67677,0.08291) -- (2.77778,0.0168) -- (2.87879,-0.03994) -- (2.9798,-0.0222) -- (3.08081,-0.15706) -- (3.18182,-0.14493) -- (3.28283,-0.17105) -- (3.38384,-0.05926) -- (3.48485,-0.13212) -- (3.58586,-0.05611) -- (3.68687,-0.26411) -- (3.78788,-0.2938) -- (3.88889,-0.20986) -- (3.9899,-0.2011) -- (4.09091,-0.17243) -- (4.19192,-0.14853) -- (4.29293,-0.17532) -- (4.39394,-0.20191) -- (4.49495,-0.14987) -- (4.59596,-0.08922) -- (4.69697,0.05924) -- (4.79798,0.0801) -- (4.89899,-0.01881) -- (5,0.06988);
\draw[thick,colorb] (1.9697,-0.02658) -- (2.07071,0.05377) -- (2.17172,-0.00878) -- (2.27273,0.15913) -- (2.37374,0.1617) -- (2.47475,0.25056) -- (2.57576,0.22359) -- (2.67677,0.32266) -- (2.77778,0.25263) -- (2.87879,0.18062) -- (2.9798,0.24331) -- (3.08081,0.08891) -- (3.18182,-0.02803) -- (3.28283,-0.13581) -- (3.38384,-0.13902) -- (3.48485,-0.12597) -- (3.58586,-0.21186) -- (3.68687,-0.24398) -- (3.78788,-0.17209) -- (3.88889,-0.20537) -- (3.9899,-0.18179) -- (4.09091,-0.13994) -- (4.19192,-0.15419) -- (4.29293,-0.16685) -- (4.39394,-0.22336) -- (4.49495,-0.22277) -- (4.59596,-0.18498) -- (4.69697,-0.21961) -- (4.79798,-0.18946) -- (4.89899,-0.09895) -- (5,-0.28498);
\draw[very thick,colorc] (-2.47475,0.22533) -- (-2.37374,0.09353) -- (-2.27273,0.09353) -- (-2.17172,0.06026) -- (-2.07071,0.14144) -- (-1.9697,0.06721) -- (-1.86869,0.05571) -- (-1.76768,-0.04221) -- (-1.66667,0.01394) -- (-1.56566,0.0031) -- (-1.46465,0.09178) -- (-1.36364,0.16299) -- (-1.26263,0.1612) -- (-1.16162,0.23403) -- (-1.06061,0.25734) -- (-0.9596,0.25905) -- (-0.85859,0.38207) -- (-0.75758,0.4566) -- (-0.65657,0.52858) -- (-0.55556,0.50038) -- (-0.45455,0.4529) -- (-0.35354,0.42027) -- (-0.25253,0.53111) -- (-0.15152,0.5226) -- (-0.05051,0.62778) -- (0.05051,0.40552) -- (0.15152,0.36131) -- (0.25253,0.35387) -- (0.35354,0.30411) -- (0.45455,0.3398) -- (0.55556,0.30721) -- (0.65657,0.32631) -- (0.75758,0.28878) -- (0.85859,0.34895) -- (0.9596,0.34717) -- (1.06061,0.3492) -- (1.16162,0.29116) -- (1.26263,0.16955) -- (1.36364,0.19622) -- (1.46465,0.22186) -- (1.56566,0.12973) -- (1.66667,0.12939) -- (1.76768,0.04758) -- (1.86869,0.04758) -- (1.9697,0.16003);
\draw[very thick,colord] (-2.47475,-0.05998) -- (-2.37374,-0.14761) -- (-2.27273,-0.14761) -- (-2.17172,-0.12022) -- (-2.07071,-0.15873) -- (-1.9697,-0.23345) -- (-1.86869,-0.1702) -- (-1.76768,-0.11827) -- (-1.66667,-0.1799) -- (-1.56566,-0.00583) -- (-1.46465,-0.11825) -- (-1.36364,-0.19265) -- (-1.26263,-0.17055) -- (-1.16162,-0.2264) -- (-1.06061,-0.08582) -- (-0.9596,-0.12513) -- (-0.85859,-0.28563) -- (-0.75758,-0.19189) -- (-0.65657,-0.17758) -- (-0.55556,-0.10153) -- (-0.45455,-0.09311) -- (-0.35354,0.06829) -- (-0.25253,-0.00991) -- (-0.15152,-0.08422) -- (-0.05051,-0.15943) -- (0.05051,-0.22236) -- (0.15152,-0.12896) -- (0.25253,-0.18864) -- (0.35354,-0.28057) -- (0.45455,-0.25426) -- (0.55556,-0.33099) -- (0.65657,-0.29088) -- (0.75758,-0.24146) -- (0.85859,-0.27014) -- (0.9596,-0.29492) -- (1.06061,-0.27498) -- (1.16162,-0.23529) -- (1.26263,-0.22231) -- (1.36364,-0.23198) -- (1.46465,-0.24917) -- (1.56566,-0.17267) -- (1.66667,-0.11026) -- (1.76768,-0.09596) -- (1.86869,-0.09596) -- (1.9697,-0.02658);
\fill[fill=magenta] (-2.47475,0.22533) circle[radius=0.07];
\fill[fill=magenta] (1.9697,0.16003) circle[radius=0.07];
\fill[fill=teal] (-2.47475,-0.05998) circle[radius=0.07];
\fill[fill=teal] (1.9697,-0.02658) circle[radius=0.07];
\draw[thick,colora] (-5,1) -- (-4.89899,0.95811) -- (-4.79798,0.97491) -- (-4.69697,0.99605) -- (-4.59596,0.96936) -- (-4.49495,1.0253) -- (-4.39394,0.95602) -- (-4.29293,1.09908) -- (-4.19192,1.01587) -- (-4.09091,1.08926) -- (-3.9899,1.09107) -- (-3.88889,1.09796) -- (-3.78788,1.15283) -- (-3.68687,1.18289) -- (-3.58586,1.10533) -- (-3.48485,1.14243) -- (-3.38384,1.12249) -- (-3.28283,1.08819) -- (-3.18182,1.11577) -- (-3.08081,1.15754) -- (-2.9798,1.10094) -- (-2.87879,1.09887) -- (-2.77778,1.01731) -- (-2.67677,0.93324) -- (-2.57576,1.01365) -- (-2.47475,1.09816);
\draw[thick,colorb] (-5,1) -- (-4.89899,0.90614) -- (-4.79798,0.73849) -- (-4.69697,0.79498) -- (-4.59596,0.78689) -- (-4.49495,0.65636) -- (-4.39394,0.65322) -- (-4.29293,0.73808) -- (-4.19192,0.68456) -- (-4.09091,0.71279) -- (-3.9899,0.67792) -- (-3.88889,0.79087) -- (-3.78788,0.89769) -- (-3.68687,1.00866) -- (-3.58586,0.9729) -- (-3.48485,0.86579) -- (-3.38384,0.91158) -- (-3.28283,1.01724) -- (-3.18182,0.96226) -- (-3.08081,0.97263) -- (-2.9798,0.97131) -- (-2.87879,1.09551) -- (-2.77778,0.9801) -- (-2.67677,0.91527) -- (-2.57576,0.91239) -- (-2.47475,0.90417);
\draw[thick,colora] (1.9697,1.19047) -- (2.07071,1.19173) -- (2.17172,1.0141) -- (2.27273,0.96793) -- (2.37374,1.05082) -- (2.47475,1.11656) -- (2.57576,1.05146) -- (2.67677,0.98389) -- (2.77778,1.03852) -- (2.87879,1.04246) -- (2.9798,1.1166) -- (3.08081,1.01988) -- (3.18182,1.03532) -- (3.28283,1.04575) -- (3.38384,0.969) -- (3.48485,1.03654) -- (3.58586,1.17411) -- (3.68687,1.14241) -- (3.78788,1.25302) -- (3.88889,1.28298) -- (3.9899,1.13516) -- (4.09091,1.06261) -- (4.19192,1.04696) -- (4.29293,1.12372) -- (4.39394,1.1081) -- (4.49495,1.2323) -- (4.59596,1.20369) -- (4.69697,1.24253) -- (4.79798,1.27642) -- (4.89899,1.28739) -- (5,1.27768);
\draw[thick,colorb] (1.9697,1.03535) -- (2.07071,1.09312) -- (2.17172,1.04077) -- (2.27273,1.03258) -- (2.37374,1.08002) -- (2.47475,1.12456) -- (2.57576,1.17326) -- (2.67677,1.25912) -- (2.77778,1.24291) -- (2.87879,1.20194) -- (2.9798,1.3128) -- (3.08081,1.33978) -- (3.18182,1.27716) -- (3.28283,1.36171) -- (3.38384,1.31857) -- (3.48485,1.23505) -- (3.58586,1.20327) -- (3.68687,1.2487) -- (3.78788,1.15019) -- (3.88889,1.04962) -- (3.9899,1.11691) -- (4.09091,1.11014) -- (4.19192,0.96222) -- (4.29293,0.95917) -- (4.39394,0.96685) -- (4.49495,1.11247) -- (4.59596,1.12363) -- (4.69697,1.21684) -- (4.79798,1.21563) -- (4.89899,1.1747) -- (5,1.15081);
\draw[very thick,colorc] (-2.47475,1.09816) -- (-2.37374,1.05757) -- (-2.27273,1.05757) -- (-2.17172,1.08081) -- (-2.07071,1.16313) -- (-1.9697,1.06433) -- (-1.86869,1.1101) -- (-1.76768,1.24502) -- (-1.66667,1.23193) -- (-1.56566,1.19605) -- (-1.46465,1.17352) -- (-1.36364,1.13725) -- (-1.26263,1.1125) -- (-1.16162,1.09212) -- (-1.06061,1.10246) -- (-0.9596,1.11663) -- (-0.85859,1.00356) -- (-0.75758,1.06846) -- (-0.65657,0.93002) -- (-0.55556,0.92533) -- (-0.45455,0.95962) -- (-0.35354,1.09718) -- (-0.25253,1.06556) -- (-0.15152,1.14348) -- (-0.05051,1.196) -- (0.05051,1.25913) -- (0.15152,1.29612) -- (0.25253,1.21884) -- (0.35354,1.31276) -- (0.45455,1.31043) -- (0.55556,1.34738) -- (0.65657,1.32459) -- (0.75758,1.28922) -- (0.85859,1.33092) -- (0.9596,1.26318) -- (1.06061,1.24573) -- (1.16162,1.30932) -- (1.26263,1.29939) -- (1.36364,1.1779) -- (1.46465,1.13174) -- (1.56566,1.20158) -- (1.66667,1.17753) -- (1.76768,1.20209) -- (1.86869,1.20209) -- (1.9697,1.19047);
\draw[very thick,colord] (-2.47475,0.90417) -- (-2.37374,0.85682) -- (-2.27273,0.85682) -- (-2.17172,0.80555) -- (-2.07071,0.93663) -- (-1.9697,1.00615) -- (-1.86869,1.04245) -- (-1.76768,0.96352) -- (-1.66667,0.9423) -- (-1.56566,0.85717) -- (-1.46465,0.91183) -- (-1.36364,0.95332) -- (-1.26263,0.89956) -- (-1.16162,0.86381) -- (-1.06061,0.93904) -- (-0.9596,0.88371) -- (-0.85859,0.7843) -- (-0.75758,0.77323) -- (-0.65657,0.76565) -- (-0.55556,0.75021) -- (-0.45455,0.85309) -- (-0.35354,0.70407) -- (-0.25253,0.79127) -- (-0.15152,0.68035) -- (-0.05051,0.71671) -- (0.05051,0.7428) -- (0.15152,0.68969) -- (0.25253,0.70502) -- (0.35354,0.66131) -- (0.45455,0.60222) -- (0.55556,0.68013) -- (0.65657,0.70114) -- (0.75758,0.69546) -- (0.85859,0.80643) -- (0.9596,0.71303) -- (1.06061,0.7911) -- (1.16162,0.78997) -- (1.26263,0.70343) -- (1.36364,0.81597) -- (1.46465,0.87264) -- (1.56566,0.90218) -- (1.66667,1.01235) -- (1.76768,1.00843) -- (1.86869,1.00843) -- (1.9697,1.03535);
\fill[fill=magenta] (-2.47475,1.09816) circle[radius=0.07];
\fill[fill=magenta] (1.9697,1.19047) circle[radius=0.07];
\fill[fill=teal] (-2.47475,0.90417) circle[radius=0.07];
\fill[fill=teal] (1.9697,1.03535) circle[radius=0.07];
\draw[thick,colora] (-5,2) -- (-4.89899,1.92105) -- (-4.79798,1.98195) -- (-4.69697,2.17228) -- (-4.59596,2.18787) -- (-4.49495,2.23316) -- (-4.39394,2.35502) -- (-4.29293,2.32874) -- (-4.19192,2.32405) -- (-4.09091,2.29398) -- (-3.9899,2.27969) -- (-3.88889,2.16662) -- (-3.78788,2.1358) -- (-3.68687,2.14898) -- (-3.58586,2.12682) -- (-3.48485,2.05589) -- (-3.38384,2.01989) -- (-3.28283,1.95135) -- (-3.18182,1.98941) -- (-3.08081,2.00954) -- (-2.9798,1.97954) -- (-2.87879,1.81824) -- (-2.77778,1.87033) -- (-2.67677,1.97199) -- (-2.57576,1.94668) -- (-2.47475,2.0082);
\draw[thick,colorb] (-5,2) -- (-4.89899,1.97783) -- (-4.79798,1.91389) -- (-4.69697,2.08033) -- (-4.59596,2.07625) -- (-4.49495,2.1533) -- (-4.39394,2.00684) -- (-4.29293,2.0066) -- (-4.19192,1.97531) -- (-4.09091,1.98428) -- (-3.9899,2.05416) -- (-3.88889,2.11913) -- (-3.78788,2.32307) -- (-3.68687,2.45284) -- (-3.58586,2.40648) -- (-3.48485,2.37392) -- (-3.38384,2.24361) -- (-3.28283,2.20925) -- (-3.18182,2.17734) -- (-3.08081,2.20224) -- (-2.9798,1.99398) -- (-2.87879,1.99471) -- (-2.77778,1.97932) -- (-2.67677,1.8988) -- (-2.57576,1.75118) -- (-2.47475,1.89352);
\draw[thick,colora] (1.9697,2.10271) -- (2.07071,2.14494) -- (2.17172,2.0135) -- (2.27273,2.00346) -- (2.37374,1.86928) -- (2.47475,1.89125) -- (2.57576,1.83225) -- (2.67677,1.77149) -- (2.77778,1.79649) -- (2.87879,1.7203) -- (2.9798,1.79316) -- (3.08081,1.88618) -- (3.18182,1.84822) -- (3.28283,1.76733) -- (3.38384,1.82243) -- (3.48485,1.85982) -- (3.58586,1.93749) -- (3.68687,1.99624) -- (3.78788,2.01309) -- (3.88889,2.04876) -- (3.9899,2.00743) -- (4.09091,1.93172) -- (4.19192,1.83035) -- (4.29293,1.83953) -- (4.39394,1.89433) -- (4.49495,1.75597) -- (4.59596,1.80239) -- (4.69697,1.78571) -- (4.79798,1.67782) -- (4.89899,1.65453) -- (5,1.78188);
\draw[thick,colorb] (1.9697,1.91426) -- (2.07071,1.95974) -- (2.17172,1.89747) -- (2.27273,1.79465) -- (2.37374,1.75373) -- (2.47475,1.61105) -- (2.57576,1.64544) -- (2.67677,1.80823) -- (2.77778,1.77022) -- (2.87879,1.8925) -- (2.9798,1.95956) -- (3.08081,1.97212) -- (3.18182,2.0396) -- (3.28283,2.03395) -- (3.38384,2.06828) -- (3.48485,2.07288) -- (3.58586,2.08699) -- (3.68687,2.00482) -- (3.78788,2.061) -- (3.88889,2.04703) -- (3.9899,2.04235) -- (4.09091,2.02793) -- (4.19192,1.99598) -- (4.29293,1.97776) -- (4.39394,1.97936) -- (4.49495,2.03649) -- (4.59596,2.00169) -- (4.69697,1.97747) -- (4.79798,1.97051) -- (4.89899,1.98188) -- (5,1.95588);
\draw[very thick,colorc] (-2.47475,2.0082) -- (-2.37374,2.04592) -- (-2.27273,2.04592) -- (-2.17172,2.09995) -- (-2.07071,2.06054) -- (-1.9697,2.01024) -- (-1.86869,2.02132) -- (-1.76768,2.12711) -- (-1.66667,2.07476) -- (-1.56566,2.01735) -- (-1.46465,2.0929) -- (-1.36364,2.15151) -- (-1.26263,2.27138) -- (-1.16162,2.24432) -- (-1.06061,2.12721) -- (-0.9596,2.22456) -- (-0.85859,2.21162) -- (-0.75758,2.22221) -- (-0.65657,2.15374) -- (-0.55556,2.20909) -- (-0.45455,2.22854) -- (-0.35354,2.27417) -- (-0.25253,2.30949) -- (-0.15152,2.24549) -- (-0.05051,2.2769) -- (0.05051,2.34039) -- (0.15152,2.26683) -- (0.25253,2.23649) -- (0.35354,2.21623) -- (0.45455,2.14324) -- (0.55556,2.1487) -- (0.65657,2.1249) -- (0.75758,2.25401) -- (0.85859,2.24785) -- (0.9596,2.2039) -- (1.06061,2.25277) -- (1.16162,2.25851) -- (1.26263,2.16591) -- (1.36364,2.19144) -- (1.46465,2.12385) -- (1.56566,2.06133) -- (1.66667,2.05835) -- (1.76768,2.08138) -- (1.86869,2.08138) -- (1.9697,2.10271);
\draw[very thick,colord] (-2.47475,1.89352) -- (-2.37374,1.89333) -- (-2.27273,1.89333) -- (-2.17172,1.79156) -- (-2.07071,1.78112) -- (-1.9697,1.78411) -- (-1.86869,1.79929) -- (-1.76768,1.73604) -- (-1.66667,1.72666) -- (-1.56566,1.65312) -- (-1.46465,1.68243) -- (-1.36364,1.66369) -- (-1.26263,1.6557) -- (-1.16162,1.75089) -- (-1.06061,1.70561) -- (-0.9596,1.81058) -- (-0.85859,1.76271) -- (-0.75758,1.65261) -- (-0.65657,1.63955) -- (-0.55556,1.64832) -- (-0.45455,1.63911) -- (-0.35354,1.798) -- (-0.25253,1.81191) -- (-0.15152,1.92148) -- (-0.05051,1.93924) -- (0.05051,1.82524) -- (0.15152,1.79405) -- (0.25253,1.70908) -- (0.35354,1.71204) -- (0.45455,1.68649) -- (0.55556,1.82389) -- (0.65657,1.9247) -- (0.75758,1.93687) -- (0.85859,1.82941) -- (0.9596,1.86203) -- (1.06061,1.93216) -- (1.16162,1.90473) -- (1.26263,1.86319) -- (1.36364,1.85054) -- (1.46465,1.73972) -- (1.56566,1.87369) -- (1.66667,1.99933) -- (1.76768,1.98948) -- (1.86869,1.98948) -- (1.9697,1.91426);
\fill[fill=magenta] (-2.47475,2.0082) circle[radius=0.07];
\fill[fill=magenta] (1.9697,2.10271) circle[radius=0.07];
\fill[fill=teal] (-2.47475,1.89352) circle[radius=0.07];
\fill[fill=teal] (1.9697,1.91426) circle[radius=0.07];
\draw[thick,colora] (-5,3) -- (-4.89899,2.94297) -- (-4.79798,2.8128) -- (-4.69697,2.833) -- (-4.59596,2.88338) -- (-4.49495,2.7553) -- (-4.39394,2.78597) -- (-4.29293,2.82327) -- (-4.19192,2.90075) -- (-4.09091,2.81606) -- (-3.9899,2.93204) -- (-3.88889,3.07821) -- (-3.78788,3.0485) -- (-3.68687,3.17675) -- (-3.58586,3.1403) -- (-3.48485,3.08373) -- (-3.38384,3.09733) -- (-3.28283,3.15799) -- (-3.18182,3.19138) -- (-3.08081,3.15084) -- (-2.9798,3.13138) -- (-2.87879,3.01541) -- (-2.77778,2.99298) -- (-2.67677,3.09027) -- (-2.57576,3.14074) -- (-2.47475,3.10858);
\draw[thick,colorb] (-5,3) -- (-4.89899,3.04426) -- (-4.79798,3.0339) -- (-4.69697,3.10163) -- (-4.59596,3.00095) -- (-4.49495,3.11372) -- (-4.39394,3.13436) -- (-4.29293,3.09077) -- (-4.19192,3.0227) -- (-4.09091,2.96097) -- (-3.9899,2.98766) -- (-3.88889,2.90391) -- (-3.78788,2.93068) -- (-3.68687,2.99752) -- (-3.58586,3.03593) -- (-3.48485,3.03516) -- (-3.38384,3.05231) -- (-3.28283,2.96511) -- (-3.18182,2.94408) -- (-3.08081,2.90698) -- (-2.9798,2.88833) -- (-2.87879,2.85712) -- (-2.77778,2.95143) -- (-2.67677,2.97224) -- (-2.57576,2.90848) -- (-2.47475,2.87402);
\draw[thick,colora] (1.9697,3.16457) -- (2.07071,3.12164) -- (2.17172,3.10441) -- (2.27273,3.19174) -- (2.37374,3.1243) -- (2.47475,3.09522) -- (2.57576,2.96667) -- (2.67677,2.98204) -- (2.77778,2.98577) -- (2.87879,3.0284) -- (2.9798,3.0982) -- (3.08081,3.13808) -- (3.18182,3.09762) -- (3.28283,3.11817) -- (3.38384,3.04829) -- (3.48485,3.05073) -- (3.58586,3.01549) -- (3.68687,3.02829) -- (3.78788,2.87358) -- (3.88889,2.88586) -- (3.9899,2.89037) -- (4.09091,2.75496) -- (4.19192,2.87104) -- (4.29293,2.80198) -- (4.39394,2.85684) -- (4.49495,3.03303) -- (4.59596,3.00562) -- (4.69697,2.93692) -- (4.79798,2.95515) -- (4.89899,3.00088) -- (5,2.89474);
\draw[thick,colorb] (1.9697,2.81848) -- (2.07071,2.80998) -- (2.17172,2.92835) -- (2.27273,2.91033) -- (2.37374,2.86712) -- (2.47475,2.86637) -- (2.57576,2.91476) -- (2.67677,2.89957) -- (2.77778,2.84981) -- (2.87879,2.8016) -- (2.9798,2.69138) -- (3.08081,2.6259) -- (3.18182,2.61733) -- (3.28283,2.71827) -- (3.38384,2.78428) -- (3.48485,2.81179) -- (3.58586,2.86994) -- (3.68687,2.92897) -- (3.78788,2.97864) -- (3.88889,3.02978) -- (3.9899,2.99245) -- (4.09091,3.04467) -- (4.19192,3.00238) -- (4.29293,3.11621) -- (4.39394,3.0664) -- (4.49495,3.06023) -- (4.59596,3.15167) -- (4.69697,3.06794) -- (4.79798,3.12199) -- (4.89899,3.121) -- (5,3.037);
\draw[very thick,colorc] (-2.47475,3.10858) -- (-2.37374,3.1101) -- (-2.27273,3.1101) -- (-2.17172,3.19963) -- (-2.07071,3.15016) -- (-1.9697,3.13447) -- (-1.86869,3.16032) -- (-1.76768,3.24403) -- (-1.66667,3.29854) -- (-1.56566,3.3177) -- (-1.46465,3.26884) -- (-1.36364,3.23026) -- (-1.26263,3.23051) -- (-1.16162,3.20627) -- (-1.06061,3.22728) -- (-0.9596,3.22732) -- (-0.85859,3.16025) -- (-0.75758,3.26082) -- (-0.65657,3.35363) -- (-0.55556,3.41909) -- (-0.45455,3.47089) -- (-0.35354,3.52359) -- (-0.25253,3.5135) -- (-0.15152,3.43967) -- (-0.05051,3.30267) -- (0.05051,3.21385) -- (0.15152,3.24577) -- (0.25253,3.31379) -- (0.35354,3.21988) -- (0.45455,3.19648) -- (0.55556,3.07558) -- (0.65657,2.98762) -- (0.75758,3.07212) -- (0.85859,3.14922) -- (0.9596,3.21685) -- (1.06061,3.21407) -- (1.16162,3.1704) -- (1.26263,3.1433) -- (1.36364,3.17252) -- (1.46465,3.19577) -- (1.56566,3.17058) -- (1.66667,3.13536) -- (1.76768,3.10392) -- (1.86869,3.10392) -- (1.9697,3.16457);
\draw[very thick,colord] (-2.47475,2.87402) -- (-2.37374,3.00444) -- (-2.27273,3.00444) -- (-2.17172,2.88778) -- (-2.07071,2.86693) -- (-1.9697,2.75644) -- (-1.86869,2.80489) -- (-1.76768,2.83566) -- (-1.66667,2.90989) -- (-1.56566,2.9312) -- (-1.46465,2.83381) -- (-1.36364,2.95262) -- (-1.26263,2.95062) -- (-1.16162,2.77457) -- (-1.06061,2.66168) -- (-0.9596,2.68828) -- (-0.85859,2.62738) -- (-0.75758,2.59735) -- (-0.65657,2.58363) -- (-0.55556,2.63991) -- (-0.45455,2.71219) -- (-0.35354,2.80113) -- (-0.25253,2.75468) -- (-0.15152,2.76883) -- (-0.05051,2.81437) -- (0.05051,2.72812) -- (0.15152,2.73099) -- (0.25253,2.67896) -- (0.35354,2.55364) -- (0.45455,2.62929) -- (0.55556,2.64921) -- (0.65657,2.61179) -- (0.75758,2.6704) -- (0.85859,2.64212) -- (0.9596,2.65665) -- (1.06061,2.70533) -- (1.16162,2.73432) -- (1.26263,2.76663) -- (1.36364,2.81484) -- (1.46465,2.8843) -- (1.56566,2.76204) -- (1.66667,2.81845) -- (1.76768,2.89192) -- (1.86869,2.89192) -- (1.9697,2.81848);
\fill[fill=magenta] (-2.47475,3.10858) circle[radius=0.07];
\fill[fill=magenta] (1.9697,3.16457) circle[radius=0.07];
\fill[fill=teal] (-2.47475,2.87402) circle[radius=0.07];
\fill[fill=teal] (1.9697,2.81848) circle[radius=0.07];
\end{tikzpicture}
\caption{An illustration of the resampling argument outlined. On the left, we have two independent samples of $\mu^{0}_{n,T}$ (pink) and $\mu^{f}_{n,T}$ (green) with $n=4$. The stopping domain where the pink data dominates the green data is denoted by $[\tau_{\ell}, \tau_r]$. On the right, we have the resampled ensembles (denoted by orange and blue respectively) under the monotone coupling causing the orange curves to deterministically be above the blue curves.}
\label{fig:key105}
\end{figure}
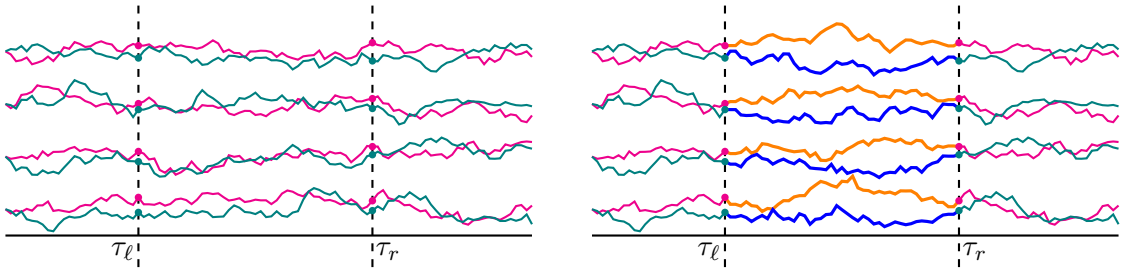

Note that the existence of the stopping domain relies on the independent fluctuations of the two ensembles to find a random time such that the zero boundary ensemble is above the free ensemble. This necessitates $T$ to be large enough as a function of $n$ to allow enough room to fluctuate. The order in which limits are taken in Theorem \ref{th:DLZ} permits taking $T$ much larger compared to $n$ and hence this strategy suffices. On the other hand, for stronger claims such as Theorem \ref{th:uniqueness} and Corollary \ref{freeunique}, this does not work as $n$ can be in principle infinite for finite $T$. This is where we need to refine our approach. A finer coupling of the free and zero boundary measures is also at the heart of our ergodicity results Theorem \ref{th:mixing} and Theorem \ref{th:corrdecay}. Without diving further into details let us only  mention that a crucial input at this point is a strong one point tail estimate for general uniformly tight ensembles. For the zero boundary ensemble this is expressed by Theorem \ref{th:stretched}. To deliver such estimates we set up a bootstrap program, which is broadly divided into two parts.
The first part proves that a uniformly tight hypothesis (recall Definition \ref{ut}) can be upgraded to a uniform confinement property, a notion we introduce later (see Definition \ref{ust}) which essentially amounts to satisfying first moment bounds as in Corollary \ref{cor:logtight}. In the second part, such first moment bounds are then upgraded further to the sought tail bounds. 

\subsection{Organization of the article}
In the upcoming section we implement the reverse coupling strategy outlined above for an alternative proof of Theorem \ref{th:DLZ}. 
Section \ref{sec:tails} is devoted to the one point tail estimate, Theorem \ref{th:stretched}. We then investigate mixing properties and prove Theorems 
 \ref{th:mixing} and \ref{th:corrdecay} in Section \ref{sec:ergodicity}. The argument developed in this section also allows us to prove Theorem \ref{lowertail}. Finally, we prove Theorem \ref{th:uniqueness} and Corollary \ref{freeunique} in Section \ref{sec:uniqueness} which also contains the proof of stretched exponential tail bounds under the assumption of UT.

\subsection{Acknowledgements}
The authors thank Evgeni Dimitrov and Senya Shlosman for useful comments on an earlier draft of the paper. The authors also thank Mriganka Basu Roy Chowdhury for help with the figures.  
PC was supported in part by the Miller Institute at UC Berkeley.  SG was partially supported by NSF grant DMS-$1855688$, NSF CAREER Award
DMS-$1945172$, and a Sloan Research Fellowship. This work was initiated when PC was visiting UC Berkeley as a Visiting Miller Professor.

\section{A first coupling argument}\label{sec:coupling}
In this section we introduce the vanilla version of the crucial coupling argument which as indicated in Section \ref{iop} will be a key device in many of our arguments. While we will need refined versions of this in the proofs of our main results, here as a warmup we introduce the key idea which already will suffice to provide an alternate proof of Theorem \ref{th:DLZ}. 

In what follows $a=1$ and $\l>1$ are fixed, and the constants 
appearing below are allowed to depend on their value, but they may not depend on other parameters.
 Let $\uX_{n,T}$ and $\uY_{n,T}$ denote the random lines with law $ \mu^f_{n,T}$ and $ \mu^0_{n,T}$ respectively. By monotonicity there exits a coupling of $\uX_{n,T},\uY_{n,T}$ such that $\uX_{n,T}\succeq \uY_{n,T}$ with probability one. The following coupling lemma  states that one can find another coupling such that if $T$ is large enough then, in the bulk,  the inequality can be reversed with large probability.   
\begin{lemma} 
\label{lem:coupling} 
For all $n\in\bbN$, and $T>0$ there exists a coupling $\bbP_{n,T}$ of $\uX_{n,T}$ and $\uY_{n,T}$ such that 
\begin{align}\label{eq:coup1} 
\bbP_{n,T}\left(\uX_{n,T}(t)\le \uY_{n,T}(t)\,,\;\forall\,t\in[-T/2,T/2]\right)  \ge 1-\e(T)\,, 
\end{align}
for some $\e(T)\to0$, as $T\to\infty$. 
\end{lemma}
\begin{proof}
Let $Y^n_{n,T}$ denote the $n$-th path in $\uY_{n,T}$. For any $u>0$, consider the random times 
\begin{align}\label{eq:coup2} 
\t_\ell(u) = \inf\{t>-T: \; Y_{n,T}^n(t)\ge u\}\,,\quad 
\t_r(u) = \sup\{t<T: \; Y_{n,T}^n(t)\ge u\}\,.
\end{align}
Consider now $Y_{n,1}^n$, that is the lowest path in the interval $[-1,1]$ with zero boundary conditions, and let 
\begin{align}\label{eq:pnueq} 
\d_n(u) = \bbP(Y_{n,1}^n(0)\ge u).
\end{align}
For any fixed $n\in\bbN$, and $u>0$, one has $\d_n(u)>0$.  
This is a direct consequence of the definition of the ensemble $\mu^0_{n,T}$. We refer however the interested reader to Lemma \ref{lem:pkvu} below for a quantitative lower bound on such probabilities that will be crucial in some of our later arguments involving more refined couplings. 

We seek to bound from below the probability 
\begin{align}\label{eq:coup3} 
\bbP\(\t_\ell(u)<-T/2\,,\;\t_r(u)>T/2 \).
\end{align}
By monotonicity, see Lemma \ref{lem:mono} and Remark \ref{rem:nonct}, we can replace $\uY_{n,T}$ by the ensemble obtained by pinning all  paths at zero height at the endpoints of the  intervals 
\[I_j = [-T+2(j-1), -T + 2j]\,,\quad j=1,\dots,j_{\max},
\] 
where $j_{\max}=\lfloor T\rfloor$.
Call $\widehat\uY_{n,T}$ this pinned process, and let $\widehat Y^n$ denote its $n$-th path.
Note that each interval $I_j$ has size $2$ with midpoint $s_j=-T + (2j-1)$. We consider the index $j_\ell$ defined as the smallest $j$ such that $\widehat Y^n(s_j)\ge u$ (see Figure \ref{fig:stopdomain}). 
Since the intervals are independent, one has 
\begin{align}\label{eq:coup41} 
\bbP(\t_\ell(u)\ge -T/2)\le (1-\d_n(u))^{\lfloor T/4\rfloor}\le e^{-\d_n(u)\lfloor T/4\rfloor}.
\end{align}
By symmetry, the same bound applies to 
$\bbP(\t_r(u)\le T/2)$. It follows that
\begin{align}\label{eq:coup4a} 
\bbP(\t_\ell(u)<-T/2\,,\;\t_r(u)>T/2 )\ge 1 - 2e^{-\d_n(u)\lfloor T/4\rfloor}
\end{align}
\begin{figure}[h]
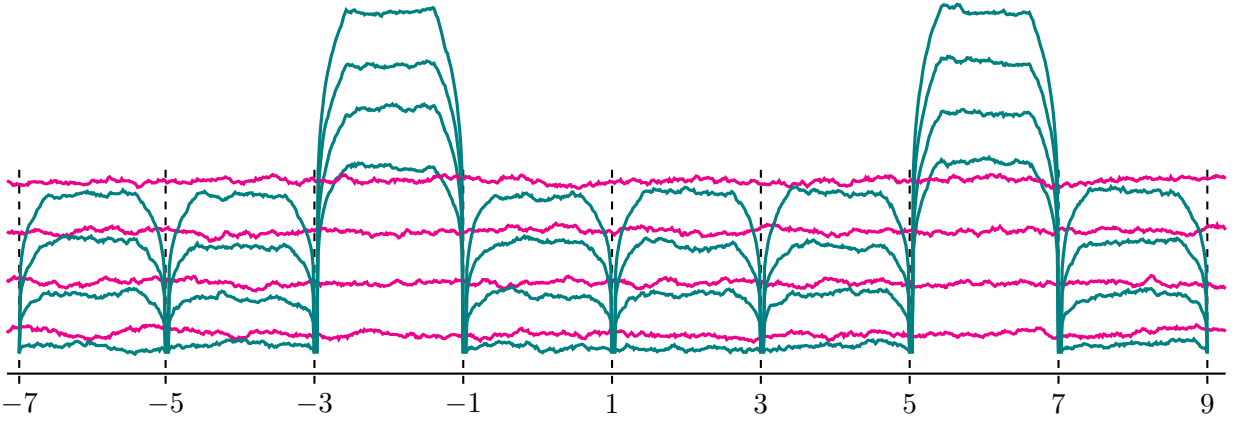


\caption{The ensemble pinned at zero at gaps of $2$ with finitely many lines (four in the figure) owing to random fluctuations will be above the infinite ensemble with the same number of lines at the random locations $\tau_{\ell}, \tau_r,$ as in the proof.}
\label{fig:stopdomain}
\end{figure}

If we condition on $\{\uY_{n,T}(t), t\le \t_\ell(u)\}$ and $\{\uY_{n,T}(t), t\ge \t_r(u)\}$,
then by the  strong  BG property  we may resample 
\[
\{\uY_{n,T}(t),\; t\in[\t_\ell(u),\t_r(u)]\}
\]
 by using the law with left boundary data $\ux:=\uY_{n,T}(\t_\ell(u))$ and right boundary data $\uy:=\uY_{n,T}(\t_r(u))$. In particular, we know that the boundary values are all higher than $u$. 

Consider now an independent sample of $\uX_{n,T}$, and define the event 
\[A=\{X^1_{n,T}(\t_\ell(u))\le u\,,\;  X^1_{n,T}(\t_r(u))\le u\},\] 
where $X^1_{n,T}$ denotes the top path in $\uX_{n,T}$. We can now construct the desired coupling $\bbP_{n,T}$. We start with two independent sample of $\uX_{n,T}, \uY_{n,T}$, as above.   If the event 
\[
B=A\cap\{\t_\ell(u)<-T/2\,,\;\t_r(u)>T/2\}
\] occurs then we resample $\{(\uX_{n,T}(t),\uY_{n,T}(t)),\; t\in[\t_\ell(u),\t_r(u)]\}$
according to the monotone coupling which by construction guarantees that 
\begin{align}\label{eq:coup4a4} 
\uX_{n,T}(t)\le \uY_{n,T}(t), \qquad t\in[-T/2,T/2],
\end{align} with probability one. If instead $B$ does not occur, then we keep the independent samples of $(\uX_{n,T},\uY_{n,T})$ everywhere. The previous observations and the strong BG property  guarantee that this is a valid coupling. Thus, we have shown that 
there exists a coupling $\bbP_{n,T}$ of $(\uX_{n,T},\uY_{n,T})$ such that 
\eqref{eq:coup4a4} holds 
with probability at least $1- q(n,T,u) - q'(n,T,u)$, where  
\begin{align}\label{eq:coup5} 
q(n,T,u)=2e^{-\d_n(u)\lfloor T/4\rfloor}\,\,\qquad  q'(n,T,u) =2 \sup_{t\in[-T,T]}\bbP( X^1_{n,T}(t)> u).
\end{align}
Next we observe that, for every fixed $n\in\bbN$, $q'(n,T,u)\to 0$ as $u\to\infty$ uniformly in $T$.  Indeed, this follows from Corollary \ref{cor:logtight} 
and Markov's inequality. Therefore, we may take $T\to\infty$ so that $q(n,T,u)\to0$ and then take $u\to\infty$, which concludes the proof. 
 \end{proof}

\subsection{Free vs.\ zero boundary conditions 
}
The coupling from the above lemma provides a rather crude comparison  of the free and zero boundary paths $\uX_{n,T}, \uY_{n,T}$.
However, this is already sufficient to provide an alternative proof of Theorem \ref{th:DLZ}.
\begin{corollary}\label{cor:DLZ}
Let $\mu^0_n$ denote the weak limit of $\mu^0_{n,T}$, as $T\to\infty$, as in \eqref{eq:CIWa}. Then \eqref{eq:DLZ} holds, that is $\mu^0_n$ is also the weak limit of $\mu^f_{n,T}$, as $T\to\infty$. 
\end{corollary}
\begin{proof}
In view of Theorem \ref{tightLE12}, the measures $\mu_{n,T}^f$ indexed by $\{n,T\}$ are tight. It remains  to show that for each fixed $n$, finite dimensional distributions converge, as $T\to\infty$, to the finite dimensional distributions of  $\mu^0_n$. 
To this end, fix $m\in\bbN$, let $\calS=(s_1,\dots,s_m)\in \bbR^m$, $\calI=(i_1,\dots,i_m)\in\{1,\dots,n\}^m$, and let $\calT=(t_1,\dots,t_m)\in \bbR_+^m$. Consider the event
\begin{equation}\label{eq:monpis}
E=E(\calS,\calI,\calT)=\{\uX\in\Omega:\; X^{i_j}(s_j) > t_j\,,\; j=1,\dots,m\}.
\end{equation}
Using the tightness of the measures $\mu^f_{n,T}$, by standard measure theoretic arguments it follows that it suffices to show that for each $n\in\bbN$, and for each choice of $\calS,\calI,\calT$,
\begin{align}\label{eq:fdconv} 
 \mu^f_{n,T}(E)-\mu^0_{n,T}(E)\to 0\,,\qquad T\to\infty.
 \end{align}
 
Since the event $E$ is increasing and $ \mu^0_{n,T}$ is stochastically dominated by $ \mu^f_{n,T}$, it follows that 
\[
\mu^f_{n,T}(E)\ge \mu^0_{n,T}(E).
\]
 On the other hand, if $T$ is large enough so that all points in $\cT$ are contained in $[-T/2,T/2]$, by Lemma \ref{lem:coupling} we
have
\begin{align}\label{eq:fdconv2} 
 \mu^f_{n,T}(E) &= \bbP_{n,T}\left(X^{i_j}(s_j) > t_j\,,\; j=1,\dots,m\right)
\\ &\le \bbP_{n,T}\left(Y^{i_j}(s_j) > t_j\,,\; j=1,\dots,m\right) + \e(T) = \mu^0_{n,T}(E) + \e(T)
. 
 \end{align}
Letting $T\to\infty$ we obtain \eqref{eq:fdconv} .
\end{proof}
 
\section{One point tail estimate}\label{sec:tails}
The proof of Theorem \ref{th:stretched} is divided into two parts. The first part establishes a weaker statement, namely a non optimal stretched exponential bound, while the second part bootstraps the first argument to reach the optimal $3/2$ exponent.  
\subsection{Proof of Theorem \ref{th:stretched}: Part I}\label{sec:thp1}
For later purposes, it is convenient to consider the following general setup. Recall that a $\l-$tilted LE is a BG measure as in Definition \ref{def:BG}. 
\begin{definition}[Uniformly confined LE]
\label{ust} A $\l-$tilted LE $\underline X$ (not necessarily stationary) is said to be {\em uniformly  confined} (UC) if there exists a constant $C$ such that
for all integers $ k\ge 0$, for all $s\in\bbR$,
\begin{equation}\label{eq:tight3}
\bbE\left[X^{k+1}(s)\right]\le C\l^{-k/3}, 
\end{equation}
and such that, for all $s\in\bbR$, $S>0$ and all $k\ge 0$
\begin{equation}\label{eq:tight3a}
\bbE\left[\max_{u\in[-S,S]}X^{k+1}(s+u) \right]\le C\l^{-k/3}[1+\log(1+|S\l^{2k/3}|)].
\end{equation}
\end{definition}
Recalling the definition of UT from Definition \ref{ut}, observe that  
\begin{equation}\label{implication}
UC\implies UT.
\end{equation}
Note also that
the UC property  
is invariant under time translation, that is if $\nu$ is UC and $\nu_s$ is the law induced by $\nu$ on the translated paths $\uX(\cdot-s)$, then $\nu_s$ is also UC, with the same constants in \eqref{eq:tight3} and \eqref{eq:tight3a}.  Lastly, the zero boundary $\l$-tilted LE $\mu^0$ is UC by Remark \ref{freetozero}.

We will now show that the UC 
assumption can already be bootstrapped to  establish the significantly stronger stretched exponential tails. We first start with the statement for the top line.

\begin{lemma}[Stretched exponential bound]\label{stretchpart1}
For any UC $\l-$tilted LE $\underline X$, there exist positive constants $\a,c$ and $C$ such that for all $t>0$,
\begin{equation}\label{eq:stretchedexp1}
\bbP\left[X^1(0)> t\right]\le C\,e^{-c\,t^\a}.
\end{equation}
Moreover, by the invariance under translation of the UC property, the tail bound \eqref{eq:stretchedexp1} holds for $X^1(s)$ uniformly in $s \in \bbR$.
\end{lemma}

The proof is technical, so we begin with a brief roadmap. Recall that we only have one point as well as curved max first moment bounds at our disposal. This allows us, simply by Markov's inequality, to argue that up to a small failure probability, paths with a large enough index (polynomial in $t$) are below a slightly raised floor. Further, Markov's inequality again allows us some, albeit weak, control on the entry and exit data of the remaining top paths (stretched exponential in $t$). Thus we have a finite problem at our hand with  polynomial in $t$ many paths whose entry and exit data are bounded by stretched exponential in $t$, on a domain we choose to be suitably large (stretched exponential in $t$). We then obtain a lower bound on the partition function of this system to argue that the paths are likely to come down to height smaller than $t$ with stretched exponentially small failure probability which finishes the proof. Crucial ingredients in the proof include the BG property, monotonicity, comparison to the Ferrari-Spohn diffusion and tail estimates of its maximum.

\begin{proof}
Fix some $y>0$ large, define $K=C_\l \log(y)$, where $C_\l>0$ is a constant depending only on $\l$ to be taken large enough, and $T=y^{10}$ and consider the events 
\begin{equation}\label{eq:stretchedexp2}
A=\left\{\max_{s\in[-T,T]}X^K(s)\le \frac1y
\right\}\,,\qquad B=\left\{\max_{s\in[-T,T]}X^1(s)\le y
\right\}.
\end{equation}
Let us first show that 
\begin{equation}\label{eq:stretchedexp30}
\bbP(A\cap B) \ge 1- \frac1{y^\b}\,,
\end{equation}
for any constant $\b\in(0,1)$ provided $y$ is large enough. Indeed, for any $k\in\bbN$, using Markov's inequality along with \eqref{eq:tight3a} (which holds by hypothesis), 
 \begin{equation}\label{eq:stretchedexp4}
\bbP\(\max_{s\in[-T,T]}X^{k+1}(s)> u\)\le \frac{C\l^{-k/3}(k+\log(T))}u .
\end{equation}
If $k=0$ and $u=y$ we have $\bbP(B^c)\le Cy^{-1}\log(y)$ for some new constant $C>0$. If $k+1=C_\l \log(y)$, where $C_\l>0$ is sufficiently large depending only on $\l$, and $u=1/y$, we have 
$\bbP(A^c)\le Cy^{-1}$. This proves \eqref{eq:stretchedexp30}.

Next, we observe that, thanks to \eqref{eq:stretchedexp30}, to prove \eqref{eq:stretchedexp1} it is sufficient to show that there exist constants $a> 1,b>0$ such that for all $y>0$ large enough one has 
\begin{equation}\label{eq:stretchedexp5}
\bbP\left(X^1(0)>2\log^a(y) \tc A\cap B\right)\le \frac1{y^b}.
\end{equation}
By monotonicity and the BG property, it follows that it suffices to prove the estimate 
\begin{equation}\label{eq:stretchedexp6}
\bbP\left(\xi^1_y(0)>\log^a(y) \right)\le \frac1{y^b},
\end{equation}
where $\xi^1_y(0)$ is the height at zero of the top line of the line ensemble $\underline \xi_y = (\xi^1_y,\dots,\xi^K_y)$, consisting  of $K$ lines in the interval $[-T,T]$ with boundary conditions all equal to $y$, with a ceiling at height $y$ and a floor at zero, with 
all lines subject to an area tilt with the same coefficient $\l\equiv 1$.   

To prove \eqref{eq:stretchedexp6}, we introduce the nested shapes defined by the trapezoids 
$\cT_i=\cT_{S_i,\ell,h_i,y}$, where
\begin{align}\label{eq:trapezoids}
 \cT_i(s)=\begin{cases}
y & s\in[-T ,-S_i]\cup[S_i,T]
\\
y-(y-h_i)\tfrac{s+S_i}{\ell} & s\in[-S_i,-S_i+\ell]\\
h_i & s\in[-S_i+\ell,S_i-\ell]\\
y-(y-h_i)\tfrac{S_i-s}{\ell} &s\in[S_i-\ell,S_i]
\end{cases}
\end{align}
and we set 
\[
S_i:=T-2(K-i+1)\ell,\quad \ell:=y^{9}, \quad h_i:=(K-i+1)\log^5(y)\,.
\]
Note that $\cT_{i+1}\prec \cT_i$, $i=1,\dots K-1$, see Figure \ref{fig:trapa2}.

To prove \eqref{eq:stretchedexp6}, 
by monotonicity we may replace the ensemble $\{\xi^i_y\}$ by the ensemble $\{\bar \xi^i_y\}$ obtained by restricting, for each $i$, the $i$-th path $\xi^i_y$ on the interval $[-S_i-\ell,S_i+\ell]$ with boundary conditions $y$ with floor at zero and ceiling at $y$, and with area tilt $\l=1$.
Define the events 
\begin{align}\label{eq:aevpkv} 
 A_i=\{  \bar \xi^i_y\prec\cT_i\}\,,\qquad i=1,\dots,K.
\end{align} 
\begin{figure}[h]
\center
\begin{tikzpicture}[scale=0.8]
\draw [black] (0,0) -- (6,0) -- (6,4) -- (0,4) -- cycle;
\draw (1/2,4-0) -- (6-1/2,4-0) -- (5,4-3) -- (1,4-3) -- cycle;
\draw (3/2,4-0) -- (6-3/2,4-0) -- (4,4-2) -- (2,4-2) -- cycle;

\draw [dashed,black]  (1,0) -- (1,4) ;
\draw [dashed,black]  (4.5,0) -- (4.5,4) ;
\draw [dashed,black]  (5.5,0) -- (5.5,4) ;
\draw [dashed,black]  (0,4-3) -- (1,4-3) ;
\draw [dashed,black]  (0,2) -- (2,2);

  \node[shape=circle, draw=black, fill = black, scale = .2]  at (0,4-3) {}; 
  \node[scale = .7]  at (-.5,4-3) {$h_{i+1}$};
   \node[shape=circle, draw=black, fill = black, scale = .2]  at (0,2) {}; 
  \node[scale = .7]  at (-.5,2) {$h_{i}$};

 \node[shape=circle, draw=black, fill = black, scale = .2]  at (0,4) {}; 
  \node[scale = .7]  at (-.3,4) {$y$};
  \node[scale = .8]  at (-.3,-.2) {$-T$};
   \node[scale = .8]  at (6.1,-.2) {$T$};

\node[shape=circle, draw=black, fill = black, scale = .22]  at (1,0) {}; 
\node[scale = .7]  at (.85,-1/4) {$-S_{i+1}+\ell$};
\node[shape=circle, draw=black, fill = black, scale = .22]  at (6-3/2,0) {}; 
\node[shape=circle, draw=black, fill = black, scale = .22]  at (5.5,0) {}; 
\node[scale = .7]  at (4.5,-1/4) {$S_{i}$};
\node[scale = .7]  at (5.5,-1/4) {$S_{i+1}$};
\node[shape=circle, draw=black, fill = black, scale = .22]  at (3,0) {}; 
\node[scale = .7]  at (3,-1/4) {$0$};

\end{tikzpicture}
\caption{The trapezoids $\cT_i$ and $\cT_{i+1}$.}
\label{fig:trapa2}
\end{figure}
By construction, if e.g.\ $a=6$, and $ \xi^1_y(0)>\log^a(y) $, then there must exist $i=1,\dots,K$ such that $A_i$ did not occur while all $A_j$, $j=i+1,\dots,K$ have occurred. Therefore,
\begin{equation}\label{eq:stretchedexp6a6}
\bbP\left(\xi^1_y(0)>\log^a(y) \right)\le 
\sum_{i=1}^K\bbP\left(A_i^c\tc A_{i+1}\right)\,,
\end{equation}
where $A_{K+1}$ denotes the certain event. 

To prove an upper bound on $\bbP\left(A_i^c\tc A_{i+1}\right)$, by monotonicity we may estimate from below the probability that a \emph{single} random path $Z^i$  on $[-S_i-\ell,S_i+\ell]$ with boundary condition $y$, floor at $h_{i+1}$, and area tilt $\l=1$, satisfies $Z^i\prec \cT_i$, so that 
\begin{equation}\label{eq:stretchedexp6a6a}
\bbP\left(\xi^1_y(0)>\log^a(y) \right)\le 
\sum_{i=1}^K\(1-\bbP\left(Z^i \prec  \cT_i\right)\)\,,
\end{equation}
where we set $h_{K+1}=0$ for the floor of the $K$-th path $Z^K$. We emphasize that each line $Z^i$ is now analyzed separately, that is we have reduced the problem from an ensemble of $K$ lines to a single line. Next, we are going to prove
\begin{equation}\label{eq:stretchedexp8}
\bbP\left(Z^i \prec  \cT_i\right)
\ge 1- \frac1{y^b}\,,
\end{equation}
for some $b>0$, and for all $i=1,\dots,K$. Combined with \eqref{eq:stretchedexp6a6a}, and adjusting the value of the constant $b>0$,  this proves the desired claim \eqref{eq:stretchedexp6} since $K=O(\log y)$.
 
We are going to prove \eqref{eq:stretchedexp8}
in the case $i=K$,  but the same argument works for $i=1,\dots,K-1$ with no modification. Thus, we now have $Z=Z^K$, a single path on $[-S,S]:=[-T+\ell,T-\ell]$ with ceiling at $y$, floor at $0$, and area tilt $\l=1$; see Figure \ref{fig:trap2}.

\begin{figure}[h]
\center
\begin{tikzpicture}[scale=0.8]
\draw [black] (0,0) -- (6,0) -- (6,4) -- (0,4) -- cycle;
\draw (1/2,4) -- (1,1) -- (5,1) -- (6-1/2,4) -- cycle;
\draw [dashed,black]  (1/2,0) -- (1/2,4) ;
\draw [dashed,black]  (0,1) -- (1,1) ;
\draw [dashed,black]  (1,0) -- (1,1) ;

  \node[shape=circle, draw=black, fill = black, scale = .17]  at (0,1) {}; 
  \node[scale = .7]  at (-.3,1) {$h$};
 \node[shape=circle, draw=black, fill = black, scale = .17]  at (0,4) {}; 
  \node[scale = .7]  at (-.3,4) {$y$};
  \node[scale = .6]  at (-.3,-.2) {$-S$};
   \node[scale = .6]  at (6.1,-.2) {$S$};

  \node[shape=circle, draw=black, fill = black, scale = .17]  at (1/2,0) {}; 
  \node[scale = .5]  at (.5,-1/4) {$-S+\ell$};
\node[shape=circle, draw=black, fill = black, scale = .17]  at (1,0) {}; 
\node[scale = .5]  at (1.04,.2) {$-S+2\ell$};

\draw[thick,magenta,rounded corners=0.9mm] (0,4-0)--(0.1,4-0.05)--(0.3,4-0.4)--(0.4,4-1.07)--(0.5,4-1.63)--(0.55,4-1.33)--(0.65,4-1.73)--(0.75,4-1.9)--(0.8,4-2.19)--(.9,4-3.04)--(1.1,4-3.1)--(1.3,4-3.05)--(1.6,4-3.18)--(1.8,4-3.15)--(2.2,4-3.1)--(2.66,4-3.37)--(2.96,4-3.27)--(3.1,4-3.47)--(3.36,4-3.3)--(3.68,4-3.2)--(3.96,4-3.17)--(4.1,4-3.1)--(4.36,4-3.2)--(4.84,4-3.05)--(5.1,4-3.1)--(5.32,4-2.47)--(5.45,4-2.67)--(5.56,4-1.3)--(5.63,4-1.5)--(5.8,4-0.6)--(5.85,4-0.9)--(5.96,4-0.07)--(6,4-0);
\end{tikzpicture}
\caption{A sketch of the event in \eqref{eq:stretchedexp8} for $i=K$. Here  $S:=T-\ell$, $h=\log^5(y)$.
}
\label{fig:trap2}
\end{figure}
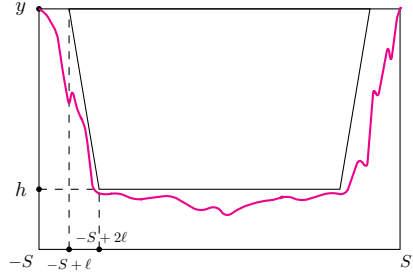

Define the stopping times
\begin{equation}\label{eq:stretchedexp9}
\t_1 = \inf\{s>-S:\;Z(s)\le h/2\}\,,\quad \t_2 = \sup\{s<S:\;Z(s)\le h/2\}.
\end{equation}
Since $y=T^{1/10}\le \sqrt S$ and $\ell=y^9\ge S^\a$ for some $\a\in(3/4,1)$, monotonicity and Lemma \ref{lem:comingdown} below show that 
\begin{gather}\label{eq:stretchedexp10} 
\bbP\left(\t_1<-S+\ell\,,\;\t_2>S-\ell \right)\ge 1- \frac1{y^b}\,.
\end{gather} 

Next, we show that 
\begin{gather}
\bbP\left(\max_{s\in[-S+\ell,S-\ell]}Z(s)\le h\,\,\Big|\,\, \t_1<-S+\ell\,,\;\t_2>S-\ell \right)\ge 1- \frac1{y^b}.
\label{eq:stretchedexp11}
\end{gather} 
Note that \eqref{eq:stretchedexp10}  and \eqref{eq:stretchedexp11}, by adjusting the value of $b$, are sufficient to conclude the proof of the desired bound \eqref{eq:stretchedexp8} at $i=K$ for all $y$ sufficiently large. 
On the other hand it is not difficult to see that, by using the strong BG property, and then monotonicity, the estimate \eqref{eq:stretchedexp11} follows from the tail bound on the maximum for the Ferrari-Spohn diffusion proven in Lemma \ref{fsmax10} below. Indeed, we may impose a floor at $h/2$ so that the probability of the event $\max_{s\in[-S+\ell,S-\ell]}Z(s)$ can be estimated by the probability of the event $\max_{s\in[-S+\ell,S-\ell]}Y_{FS}(s)\ge h/2$. This concludes the proof of Lemma \ref{stretchpart1}.
\end{proof}

\begin{lemma}\label{fsmax10}
There exist constants $c,C>0$ such that the  stationary Ferrari-Spohn diffusion $Y_{FS}$ satisfies, for all $T\ge 1$, $t>0$,
\begin{gather}\label{eq:stretchedexp14} 
\bbP
\left(
\max_{s\in[-T,T]} Y_{FS}(s) \ge t
\right)\le C\,T\exp{\left(-c
\,t^{3/2}\right)}\,.
\end{gather} 
\end{lemma}
\begin{proof}
We assume for simplicity that $T$ is an integer, and let $s_1,\dots, s_N$, with $N=2T$, denote a mesh of the interval $[-T,T]$, such that $s_{i+1}-s_i=1$. Further, let $E$ be the event that $Y_{FS}(s_i)\le t/2$ for all $i=1,\dots,N$. Then by the tail estimate \eqref{eq:domino1} and a union bound, the complement of $E$ has probability at most $C\,T\exp{\left(-c\,t^{3/2}\right)}$ for some constants $c,C$.  Thus, we may assume that $E$ holds in \eqref{eq:stretchedexp14}. By the BG property, using monotonicity and  a union bound it is then sufficient to consider the probability that at least one of $2T$ independent Brownian bridges on the interval $[0,1]$, conditioned to stay nonnegative, with boundary $0$ at both ends, has maximum larger than $t/2$. By a union bound, and using well known estimates on the maximum of Brownian excursions, 
one has that this probability is at most $ C\,T \exp{\left(-c\,t^{2}\right)}$. 
\end{proof}

The following technical lemma quantifies the pull to the floor, as a consequence of the area tilt, even in the presence of high boundary conditions. 
\begin{lemma}\label{lem:comingdown}
Let $Y_T$ denote the Brownian bridge on $[-T,T]$ with left and right boundary data at height $\sqrt T$, with floor at zero and area tilt $1$, and fix $\a\in(3/4,1)$. For $L>0$,  let 
\begin{equation}\label{eq:stretchedexp99}
\t_{\ell,L} = \inf\{s>-T:\;Y_T(s)\le L\}\,,\quad \t_{r,L} = \sup\{s<T:\;Y_T(s)\le L\}.
\end{equation}
There exist constants  $c>0$ and $L>0$ such that for all $T$ large enough
\begin{equation}\label{eq:leri}
\bbP(\t_{\ell,L} >-T+T^\a)\le e^{-T^c}\,,  \qquad \bbP(\t_{r,L} <T-T^\a)\le e^{-T^c}.
\end{equation}
\end{lemma}
\begin{proof}
By symmetry it suffices to prove the bound on the left random time $\t_{\ell,L}$. By adding a floor at height $\sqrt T$, and using the fact that a Ferrari-Spohn diffusion with zero boundary data  exceeds height $a$ on $[-T,T]$ with probability at most $O(Te^{-c\, a^{3/2}})$, see Lemma \ref{fsmax10}, we know that $Y_T$ will stay below $2\sqrt T$ with probability at least  $1-Ce^{-T^c}$ for some constants $c,C>0$, for all $T$ large enough. Thus, at a negligible cost, we may add a ceiling at $2\sqrt T$ and we may consider the path $\tilde Y_T$ defined as 
the single line with area tilt $1$ with both boundary data at $2\sqrt T$  on the restricted time interval $[-T,-T+T^\a]$ with length $T^\a$. Thus we may reduce to 
the probability  that the minimum height of the path $\tilde Y_T$ exceeds $L$. 
We have
\begin{gather}\label{eq:stret} 
\bbP\left(\t_{\ell,L}\ge -T+T^\a\right)\le \bbP\left(\min_{s}\tilde Y_T(s)>L\right) + Ce^{-T^c}.
\end{gather} 
Moreover, 
\begin{gather}
\bbP\left(\min_{s}\tilde Y_T(s)>L\right) \le 
\frac{\bbE\left[ e^{-A(B)}\ind_{A(B)\ge LT^\a}\right]}{\bbE\left[ e^{-A(B)}\right]}
\le \frac{e^{-LT^\a}}{\bbE\left[ e^{-A(B)}\right]}\,,
\label{eq:stretchedexp301} 
\end{gather} 
where $\bbE$ denotes the normalized expectation over the standard Brownian bridge $B(\cdot)$ in 
$[-t_\a,t_\a]$, $t_\a:=T^\a/2$, with boundary height $2\sqrt T$ at both ends, conditioned to satisfy $B(s)\in[0, 2\sqrt T]$ for all $s\in[-t_\a,t_\a]$, and with $A(B)=\int_{-t_\a}^{t_\a} B(s)\dd t$. We 
  have used the fact that $\min_{s}\tilde Y_T(s)>L$ implies that $A(B)\ge 2Lt_\a=LT^\a$. It remains to provide a lower bound on the denominator in \eqref{eq:stretchedexp301} . 
  
  Let $\ell= T^\b$, $0<\b <\a$, and $h\ge 1$ be  parameters to be fixed, and consider the  event that the path comes down from height $2\sqrt T$ to height $h$ within distance $\ell$ from both left and right of the interval $ [-t_\a,t_\a]$ and that it stays below height $2h$ for the rest of the time. More precisely, define 
  \[
  F_{\ell,h}=\{B(-t_\a+\ell)\le h,\; B(t_\a-\ell)\le h) \,,\; B(s)\le 2h, \;\forall s\in [-t_\a+\ell ,t_\a-\ell] \}\,.
  \] 
On the event $F_{\ell,h}$ we have 
\[
A(B)\le 2(2\sqrt T)\ell + 2h(2t_\a-\ell)\le 4 T^{\b+\tfrac12} + 2hT^\a\le (4+2h)T^\a,
\]   
provided 
$\b\le\a-1/2$, so that $\bbE\left[ e^{-A(B)}\right]\ge \bbP\left(  F_{\ell,h}\right)e^{-(4+2h)T^\a}$. Next, we  show that 
 \begin{gather}\label{eq:leri1} 
\bbP\left(  F_{\ell,h}\right)\ge e^{-T^\a}\,,
\end{gather} 
for suitable values of the constants $\b$ and $h$.
By considering the joint distribution of $X=B(-t_\a+\ell)$, and $Y=B(t_\a-\ell)$, one can check by gaussian computations that the event $\{X\in [0,h], Y\in [0,h]\}$
has probability at least $c\ell^{-1/2}e^{-C T/\ell}$, for suitable constants $c,C$. Moreover, conditioned on the occurrence of this event, a computation for the  probability of the tube event  $0\le B(s)\le 2h$, $s\in[-t_\a+\ell,t_\a-\ell]$ 
for a Brownian bridge shows that  $F_{\ell,h}$ satisfies 
 \begin{gather}\label{eq:leri2} 
\bbP\left(  F_{\ell,h}\right)\ge c\ell^{-1/2}e^{-C T/\ell}e^{-C T^\a/h^2}\,.
\end{gather} 
Since $\a>3/4$ we may take $\ell=T^\b$ with  $\b>1-\a$ and $\b<\a-1/2$, and we may choose the constant $h$ large enough, independently of $T$, to ensure that $\bbP\left(  F_{\ell,h}\right)\ge e^{-T^\a}$, for all $T$ sufficiently large. This proves \eqref{eq:leri1}.
Summarizing, using \eqref{eq:stret}  and \eqref{eq:stretchedexp301}, if we fix $L=2h+6$ we have obtained, for all $T$ large enough, 
\begin{gather}
\bbP\left(\t_{\ell,L}\ge -T+T^\a\right)\le 
e^{-(L-2h-5)T^\a}+Ce^{-T^c}\le e^{-T^\a}+Ce^{-T^c}\,.
\end{gather} 
\end{proof}

\subsection{Proof of Theorem \ref{th:stretched}: Part II}
Next, we want to  bootstrap the stretched exponential behavior in \eqref{eq:stretchedexp1} to obtain the optimal exponent $3/2$.  Let us fix $t$ large, and consider the time interval $T=t^{R}$, with $R$ a large constant to be fixed later. Thanks to 
\eqref{eq:stretchedexp1} the event 
\begin{equation}\label{eq:stretcha1}
A=\left\{X^1(-T)\le t^{2/\a}\,,\;X^1(T)\le t^{2/\a}\right\},
\end{equation}
has probability at least $1-e^{-ct^2}$ for some constant $c>0$ and all $t>1$. Moreover, recall by Remark \ref{rem:mon}, under $\mu^0_{n,T},$ conditional on the top curve, the law of the second curve is stochastically dominated by that of the unconditional first curve. Using this, that $\lim_{n,T \to \infty}\mu^0_{n,T}=\mu^0,$ and Corollary \ref{cor:logtight}, it follows by using standard measure theoretic arguments that 
the maximum of the second line $X^2$, 
\begin{equation}\label{eq:stretcha2}
M^2_T= \max_{s\in[-T,T]}X^2(s)\,,
\end{equation}
satisfies 
{\begin{equation}\label{eq:stretcha3}
\bbP
\left(M^2_T > C \log T \tc X^1(0)>t \right) \le \frac12.  
\end{equation}}
It follows that 
\begin{equation}\label{eq:stretcha4}
\bbP
\left(X^1(0)>t \right) \le \bbP
\left( X^1(0)>t \,,\; M^2_T \le C \log T \right) + \frac12\,\bbP
\left(X^1(0)>t \right).  
\end{equation}
Therefore, using also the previous observation about the event $A$ in \eqref{eq:stretcha1} one has 
\begin{equation}\label{eq:stretcha5}
\bbP
\left(X^1(0)>t \right) \le 2\,\bbP
\left( X^1(0)>t \,,\; M^2_T \le C \log T\,,\; A \right) + e^{-ct^2}.
\end{equation}
Now, on the event $\{M^2_T \le C \log T\}\cap  A$, by stochastic domination one can replace $X^1(s)$, $s\in[-T,T]$ by the  single Brownian bridge $Z(s)$, $s\in [-T,T]$, 
with boundary conditions $y:=t^{2/\a}$, with area tilt $1$ and with floor at $y_0:=  C \log T$. 
From Lemma \ref{lem:comingdown}, if $T=t^R$ with $R$ large enough, it follows  that with probability at least $1- 2e^{-t^2}$ one has $Z(s)\le y_0+\log(t)\le 2y_0$ for some stopping times $\t_1<-T/2$ and $\t_2>T/2$. On the latter event one can use the strong BG property to resample on the stopping domain $[\t_1,\t_2]\supset [-T/2,T/2]$ with boundary condition $2y_0
$. Thus, letting $Y_{FS}(\cdot)$ denote the stationary Ferrari-Spohn diffusion, by monotonicity and using \eqref{eq:domino1},  it follows that 
 \begin{align}\label{eq:stretcha6}
\bbP
\left(X^1(0)>t \right) &\le 2\,\bbP
\left( Z(0)>t  \right) + e^{-ct^2}\\
& \le 
2\,\bbP
\left( Y_{FS}(0)>t - 2y_0  \right) + 3e^{-ct^2} \\
& = \exp{\left(-\left(\tfrac{2\sqrt 2}{3}
+ o(1)\right)t^{3/2}\right)}\,,\qquad t\to\infty.
\end{align} 
This ends the proof of \eqref{eq:stretchedexp}.
\qed

Below we record a tail bound for the maximum of $X^1,$ which is a straightforward consequence of the above argument.
\begin{corollary}\label{cor:roughS}
There exist constants $c,C>0$, such that for all $S\ge 1$, for all $t>0$, 
\begin{equation}\label{eq:stretchedexpmax}
\bbP\left(\max_{ s\in[-S,S]}X^1(s)> C\log(S)+ t\right)\le CS\exp{\left(-c\, 
t^{3/2}\right)}\,.
\end{equation}
\end{corollary}
\begin{proof}
Fix $t>0$ large  and $T\ge 1$. If $T\ge T_R:=t^R$ for some large enough constant $R>0$, the exact same argument as above yields
\begin{align}
\bbP
\left(\max_{s\in [-T/2,T/2]}X^1(s)>t +2y_0 \right) &
\le 2\,\bbP
\left(\max_{s\in [-T/2,T/2]}Z(s)>t \right) + e^{-ct^2}\\
& \le 
2\,\bbP
\left( \max_{s\in [-T/2,T/2]}Y_{FS}(s)>t   \right) +3e^{-ct^2} \\
& \le C\,T \exp{\left(-c\,t^{3/2}
\right)}\,. 
\label{eq:pfofmax}\end{align}
with the last inequality following from Lemma \ref{fsmax10}. This proves \eqref{eq:stretchedexpmax} in the case $T\ge T_R$.  If $T\le T_R$ one may
estimate 
 \begin{align}
\bbP
\left(\max_{s\in [-T/2,T/2]}X^1(s)>t +C\log(T) \right)&\le \bbP
\left(\max_{s\in [-T_R/2,T_R/2]}X^1(s)>t \right)\\
& \le C\,t^R \exp{\left(-c\,(t-C\log(t))^{3/2}
\right)}\,,
\label{eq:pfofmaxa}
\end{align}
where the last bound follows from \eqref{eq:pfofmax} at $T=T_R$, for a suitable constant $C>0$. Taking $t$ sufficiently large and adjusting the value of the constants shows that for any $1\le T\le T_R$ 
the last expression is bounded by $CT \exp{(-c\,t^{3/2})}$ for suitable constants $c,C>0$. This finishes the proof.
\end{proof}

\smallskip

Corollary \ref{cor:stretched} is also a quick consequence of Theorem \ref{th:stretched}.
\begin{proof}[Proof of Corollary \ref{cor:stretched}]
By monotonicity and rescaling, as in Remark \ref{rem:mon}, for each $k$ we have 
 \begin{align}
\bbP
\left(X^{k+1}(0)>\l^{-k/3}t \right)&=  \bbP
\left(X^{i+1}(0)>\l^{-k/3}t\,,\;\forall i\le k \right)\\
& \le \prod_{i=0}^k\bbP
\left(X^{1}(0)>\l^{-(k-i)/3}t \right)\\
& = \exp{\left(-\left(c_k 
+ o(1)\right)t^{3/2}\right)}\,,\qquad t\to\infty,
\label{eq:cklab}
\end{align}
where $c_k:= \tfrac{2\sqrt 2}{3}\sum_{i=0}^k \l^{-i/2}$, which satisfies $c_k\to c_\infty(\l) = \tfrac{2\sqrt {2}}{3}\tfrac{\sqrt \l}{\sqrt {\l}-1}$ as $k\to\infty$.
\end{proof}

While this in principle finishes the proof of our results about the upper tail, for later applications we show how to prove a counterpart of the above corollary and hence an extension of Lemma \ref{stretchpart1} to general $k$, simply under the assumption UC from Definition \ref{ust}.  Note that the proof of the above corollary relied on the monotonicity and scale invariance exhibited by the zero boundary LE $\mu^0$ which a priori might not hold for a general LE (though a posteriori it does as a consequence of Theorem \ref{th:uniqueness}).
\begin{lemma}\label{univtail1}
For any UC $\l-$tilted LE $\underline X$, there exist constants  $\a>0$, $c>0$ and $C>0$, such that for any $s\in \R$ and any $k\ge 0$ and $y>0$,
\begin{equation}\label{eq:stretchedexp100}
\bbP\left[X^{k+1}(s)> \l^{-k/3}y\right]\le C\,e^{-c\,y^\a}.
\end{equation}
\end{lemma}
\begin{proof} The proof is a straightforward adaptation of the arguments in Lemma \ref{stretchpart1} and so we will be brief. First of all, without loss of generality we will take $s=0$, since the UC 
property is translation invariant. 
Fix $k\in\bbN$,  and some $y>0$ large. Define $K=C_\l \log(y)$, where $C_\l>0$ is a constant depending only on $\l$ to be taken large enough, and $T=\lambda^{-2k/3}y^{10}$, and consider the events 
\begin{equation}\label{eq:stretchedexp2a}
A=\left\{\max_{s\in[-T,T]}X^{k+1+K}(s)\lambda^{k/3}\le \frac1y
\right\}\,,\qquad B=\left\{\max_{s\in[-T,T]}X^{k+1}(s)\lambda^{k/3}\le y
\right\}.
\end{equation}
It follows from the definition of UC in \eqref{eq:tight3} and \eqref{eq:tight3a} along with an application of Markov's inequality that
\begin{equation}\label{eq:stretchedexp3aw}
\bbP(A\cap B) \ge 1- \frac1{y^\b}\,,
\end{equation}
for any constant $\b\in(0,1)$ provided $y$ is large enough. 

Further, on $A\cap B,$ by the BG property, and using monotonicity, it suffices to ignore the top $k$ curves and hence we simply have to analyze the ensemble of $K$ lines $X^{k+1},\dots,X^{k+K}$, on the domain $[-T,T]$ starting and ending at $\lambda^{-k/3}y$ with a floor at $\lambda^{-k/3}/y$ and a ceiling at $\lambda^{-k/3}y.$ We can now appeal to Brownian scaling and instead consider the ensemble on $[-y^{10},y^{10}]$ (recall that $T=\lambda^{-2k/3}y^{10}$) with $K$ lines with the usual geometric area tilts, starting and ending at $y$ with a floor at $1/y$ and a ceiling at $y.$ Let us denote this new ensemble by $\{Z^i\}_{i\le K}.$
This is exactly the setting of the proof of Lemma \ref{stretchpart1}, the arguments of which imply e.g.\ that $$\bbP\(\max_{s\in [-y^{5},y^{5}]}Z^{1}(s)\ge \log^a(y)\)\le \frac{1}{y^{b}},$$ for some constants $a,b>0$, for all $y$ large enough.  Since $X^{k+1}(0)$ is stochastically dominated by $Z^1(0)$, 
this finishes the proof. 
\end{proof}

The proof of Theorem \ref{lowertail} quantifying the entropic repulsion effect on the lower tail behavior relies on an argument developed in detail in Section \ref{sec:ergodicity}, particularly a construction outlined \eqref{eq:trapezoid}, and hence is postponed to the end thereof. 

\section{Ergodicity and mixing}\label{sec:ergodicity}
This section is devoted to the proofs of Theorems \ref{th:mixing} and \ref{th:corrdecay}.
We begin by recalling that the stationary, infinite zero boundary LE is denoted by $\uX$ and for any $t\in\bbR $, $T_t$ denotes the shift operator  
defined by $T_t\uX (\cdot ) = \uX (t + \cdot)$. 
\begin{proof}[Proof of Theorem \ref{th:mixing}]
An application of the $\pi-\l$ theorem, see e.g.\ \cite{durrett2019probability}, shows that it is sufficient to prove \eqref{eq:ergo1} for all sets $A,B$ of the form \eqref{eq:monpis}. Therefore, adopting the notation from \eqref{eq:monpis}, we fix $A=E(\calS,\calI,\calT)$ and $B=E(\calS',\calI',\calT')$, where $m\in\bbN$, 
$\calS=(s_1,\dots,s_m)\in \bbR^m$, $\calI=(i_1,\dots,i_m)\in\bbN^m$, and $\calT=(t_1,\dots,t_m)\in \bbR_+^m$ are given, and {similarly for $\calS',\calI',\calT'$}. 
Let $X'_t = \ind_{T_t A}$ and $X=\ind_{B}$, so that $\bbE\left[X'_t\right]=\bbE\left[X'_0\right]=\mu^0
\left(A \right)$ and $
\bbE
\left[X\right]  = \mu^0\left( B\right)$. 
We are going to find two independent random variables $W'_t\in[0,1]$, $W_t\in[0,1]$ and a coupling of $(X'_t,X)$ and $(W'_t,W_t)$ such that 
a.s.\ $X'_t\ge W'_t$, $X\ge W_t$  and such that, setting $\D_t= X-W_t$, $\D'_t=X'_t-W'_t$, one has
\begin{gather}\label{eq:ergo2} 
\lim_{t\to\infty}\bbE
\left[\D_t\right] = \lim_{t\to\infty}\bbE
\left[\D'_t\right] = 0
\,.
\end{gather}
By stationarity, \eqref{eq:ergo2}  is equivalent to $\lim_{t\to \infty}\bbE \left[W'_t\right]=\bbE\left[X'_0\right]$ and  $\lim_{t\to \infty}\bbE \left[W_t\right]=\bbE\left[X\right]$. 
Assuming the existence of such a coupling, we see that
\begin{align}\label{eq:ergo3} 
&\mu^0
(T_t A \cap B) = \bbE
\left[X'_tX\right]  \nonumber \\
 &\qquad  = \bbE\left[W'_t\right]
\bbE\left[ W_t\right] + \bbE
\left[W_t \D'_t\right] + \bbE
\left[W'_t \D_t\right] +\bbE
\left[\D_t\D'_t\right]    \nonumber \\
 &\qquad   \longrightarrow \bbE\left[X'_0\right]\bbE
\left[X\right]  = \mu^0
(A)\mu^0 (B)\,,\quad t\to\infty
\,,
\end{align} 
 where we use \eqref{eq:ergo2}, the independence of $W'_t,W_t$,   and the fact that $W_t,W'_t,\D_t,\D'_t\in[0,1]$.

To construct the desired coupling we take $s>0$ so that $|s_i|\le s$ and $|s'_i|\le s$, and $\ell\in\bbN$ such that $\ell\ge \max\{i_j,i'_{j'}\}$ for all $j=1,\dots, m,$ and $j'=1,\dots, m'$, so that the event $B$ concerns the first $\ell$ lines in the time interval $[-s,s]$, while the event $T_tA$ concerns  the first $\ell$ lines in the time interval $[t-s,t+s]$.  

\begin{figure}[h]
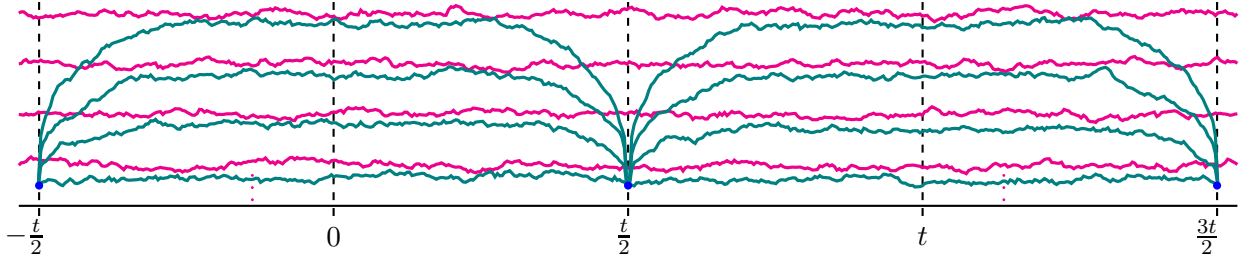


\caption{Illustration of the proof strategy for Theorem \ref{th:mixing}. The ensemble $\uY$ pinned at $\pm t/2$ and ${3t}/{2}$ (green lines) is monotonically coupled to $\uX$ with law $\mu^{0}$ (pink lines). The weak convergence guarantees that the gap between the top green lines and the top pink lines goes to zero in expectation on compact domains centered around $0$ and $t.$ Further, the green lines on $[-t/2,t/2]$ and $[t/2,3t/2]$ are independent due to the pinning.}
\end{figure}
To construct the independent proxy random variables $W'_t$ and $W_t$, consider the line ensemble $\uY_{\infty,t}$ with zero boundary condition on $[-t/2,t/2]$ and with infinitely many lines. Note that this is well defined by monotonicity as the increasing limit of the zero boundary ensemble with finitely many lines, as noted in \cite{CIW19}. Call $W_t$ the indicator function
that the first $\ell$ lines of $\uY_{\infty,t}$ satisfy the event $B$. Similarly, let $\uY'_{\infty,t}$ be the line ensemble with zero boundary condition on $[t/2,3t/2]$, with infinitely many lines, and call $W'_t$ the indicator function of the event 
that the first $\ell$ lines of $\uY'_{\infty,t}$ satisfy the event $T_tA$. 
Note that owing to the pinning, the line ensembles $\uY_{\infty,t}$ and $\uY'_{\infty,t}$ and hence the random variables $W_t,W'_t$ are independent. 

Now observing that the LE 
\begin{align}\label{eq:ley}
\uY:=(\uY_{\infty,t},\uY'_{\infty,t})
\end{align} 
obtained by concatenating $\uY_{\infty,t}$ and $\uY'_{\infty,t}$ is stochastically dominated by $\uX$, and that the events involved are all increasing, there exists a coupling of the random variables $(X'_t,X),(W'_t,W_t)$ such that 
a.s.\ $X'_t\ge W'_t$, $X\ge W_t$. It remains to prove  $ \bbE
\left[W_t \right]\uparrow  \mu^0
(B)$ and $   \bbE
\left[W'_t \right]\uparrow  \mu^0
(A)$. That $\bbE
\left[W_t \right]$ and  $\bbE
\left[W'_t \right]$ increases and that the limits are upper bounded by $\mu^0
(B)$ and $\mu^0
(A)$ respectively is a simple consequence of monotonicity and the increasing nature of $A,B$. Equality now follows by Portmanteau's lemma about weak convergence, and the fact that $A$ and $B$ are open, which implies $\liminf\bbE
\left[W_t \right]\ge \mu^0(B)$ and similarly $\liminf\bbE
\left[W'_t \right]\ge \mu^0(A)$, which finishes the proof. 

\end{proof}
 
While the above provides a qualitative proof of mixing, we next seek to obtain the quantitative bound stated in Theorem \ref{th:corrdecay}. The key technical step 
 is to estimate the convergence rate in \eqref{eq:ergo2}. We begin with a brief overview. Recall the coupling argument based on resampling presented in Section \ref{sec:coupling} where in the case of finitely many lines, a reverse coupling was constructed under which with high probability the zero boundary LE dominates the free boundary LE. In the case of infinitely many lines, this is hard to argue but nonetheless we show that  indeed such a reverse coupling exists where the requisite reverse domination holds up to a small shift, which is enough for our purposes.

To accomplish this we rely on \eqref{eq:tight3a} which allows us to ignore the curves with index larger than a chosen threshold on a finite domain by raising the floor by an appropriately chosen amount which effectively reduces the case to finitely many lines. The reverse coupling strategy can now be employed as long as we can show that with high probability there exist random times in the finite domain where the top few curves in a pinned ensemble, say, $\uY$ as in the previous proof, are higher than that of an independently chosen infinite ensemble $\uX$. Note now that there is a tradeoff in the choice of the domain size $t$. The smaller it is, the lower the floor is since by \eqref{eq:tight3a} the growth of the maximum is controlled on such a domain. On the other hand, choosing the domain to be too small makes it harder for the top few lines of $\uY$ to have large enough fluctuations in order to dominate the corresponding lines in $\uX$ at some time.

To ensure that a desirable domain exists, we will rely on the stretched exponential tail estimates developed in Corollary  \ref{cor:stretched}.

\begin{proof}[Proof of Theorem \ref{th:corrdecay}]
We start by arguing as in \eqref{eq:ergo3}.
Namely, define the random variables $X'_t=X^1(t)$, $X=X^1(0)$, where $\uX$ denotes the LE with law $\mu^0$. As before, consider the line ensemble $\uY$ defined by \eqref{eq:ley}, where 
$\uY_{\infty,t}$ has zero boundary condition on $[-t/2,t/2]$, and $\uY'_{\infty,t}$ has zero boundary condition on $[t/2,3t/2]$. Call $W_t=Y^1(0)$ the height at zero of the top line of $\uY_{\infty,t}$ and call $W'_t=Y^1(t)$ the height at $t$ of the top line of $\uY'_{\infty,t}$. Clearly, $W'_t,W_t$ are independent. Moreover, by stochastic domination and translation invariance, 
\begin{gather}\label{eq:corre02}
\bbE\left[W_t \right]=\bbE\left[W'_t \right]\le  \bbE
\left[X\right]=\bbE
\left[X'_t\right] ,
\\
\bbE\left[W^2_t \right]=\bbE\left[(W'_t)^2 \right]\le  \bbE
\left[X^2\right]=\bbE
\left[Y^2_t\right] .
\end{gather}
Now for any coupling of $(X,X'_t),(W'_t,W_t)$, by independence of $W'_t,W_t$ one has 
\begin{align}\label{eq:corre2}
&\cov\left(X^1(0),X^1(t)\right)=\bbE
\left[X X'_t\right]-\bbE
\left[X\right]\bbE\left[ X'_t\right]\\
&\qquad =  \bbE
\left[W_t\right]\bbE\left[ W'_t\right]-\bbE
\left[X\right]\bbE\left[ X'_t\right] + \bbE
\left[W_t \D'_t\right] + \bbE
\left[W'_t \D_t\right] +\bbE
\left[\D_t\D'_t\right]\,, 
\end{align}
where $\D_t=X-W_t$ and $\D'_t=X'_t-W'_t$.
Therefore, taking absolute values, using \eqref{eq:corre02} and translation invariance, the left hand side of \eqref{eq:corre1} is bounded by 
\begin{equation}\label{eq:corre21}
2\bbE[X]
\bbE\left[ \D_t\right] + 2\left|\bbE\left[ W_t\D'_t\right]\right|
+\left|\bbE
\left[\D_t\D'_t\right]\right|  .
\end{equation}
We note that this is valid for any coupling of $(X,X'_t),(W'_t,W_t)$, and that while  the inequalities \eqref{eq:corre02} ensure that $\bbE\left[ \D_t\right]=\bbE\left[ \D'_t\right] \ge 0$, not all couplings have $\D_t,\D'_t\ge 0$ pointwise. 
From \eqref{eq:corre21}, Theorem \ref{th:stretched}, and Schwarz' inequality, we see that \eqref{eq:corre1} follows once we prove that for suitable constants $c,C>0$, 
\begin{equation}\label{eq:corre3}
\bbE
\left[(\D_t)^2\right]   \le C\varphi(t)^2,
\end{equation}
where we define 
\begin{equation}\label{eq:defphit}
 \varphi(t) := \exp\left[-c(\log t)^{3/7}\right].
 \end{equation} 
By \eqref{eq:corre02}, 
\begin{equation}\label{eq:corre31}
\bbE
\left[(\D_t)^2\right] =  \bbE[(X-W_t)^2]\le 2\bbE[X^2] - 2 \bbE\left[X W_t\right] = 2\bbE\left[ X\D_t\right]\,.
\end{equation}
We are going to show that there exists a coupling of $(X,X'_t),(W'_t,W_t)$ and an event $E_t$ such that 
under $E_t$ one has $\D_t\le \varphi(t)^2$ and such that 
\begin{equation}\label{eq:corre4}
\bbP
\left(E_t^c\right)\le C\varphi(t)^4\,.
\end{equation}
In this case, 
\begin{align}\label{eq:corre32}
\bbE\left[ X\D_t\right] &\le \varphi(t)^2\bbE\left[ X\right]+\bbE\left[ X\D_t\ind_{E_t^c}\right]  \\
& \le C\varphi(t)^2 + 
\sqrt{\bbP
\left(E_t^c\right)\bbE\left[X^2(\D_t)^2\right]}\le 2C\varphi(t)^2 \,,
\end{align}
where in the first equality we use the positivity of $X$. In the final inequality we use $\bbE\left[X^2(\D_t)^2\right]\le C$, which follows e.g.\ from the crude bound $|\D_t| \le X+ W_t$ and applying Theorem \ref{th:stretched}. From \eqref{eq:corre32} and \eqref{eq:corre31} we infer
\eqref{eq:corre3} with a different constant $C$. 
Thus, it remains to prove \eqref{eq:corre4} and hence the proof is now complete with the aid of the following lemma. We end by pointing out that the same argument verbatim works for any fixed $k$ and hence shows that Theorem \ref{th:corrdecay} continues to hold in this case as well.\end{proof}

Though we only need to produce a coupling of the top lines $X^1$ and $Y^1$, for a later application in the proof of Theorem \ref{th:uniqueness}, as in Theorem \ref{th:mixing}, we will design a coupling of $\uX$ with $\uY_{\infty,t}$, the $\infty$-LE 
with zero boundary conditions on $[-t/2,t/2]$, such that with high probability the top $i$ lines  are close to each other, for any fixed $i$, provided $t$ is large enough.

\begin{lemma}\label{reversecoupling}
 Given any $i$ and $S\ge 0$, for all large enough $t$, there exists a coupling of $\uY_{\infty,t}$  and $\uX$, and an event $E_t$ such that on $E_t$ one has 
\[
X^j(s)-\uY^j_{\infty,t}(s) 
\le  \varphi(t)^2\,,\qquad j=1,\dots,i\,,\quad s\in[-S,S],
\]
and such that $\bbP(E_t^c)\le \varphi(t)^2$
 is satisfied, where $\varphi(t)$ is defined in \eqref{eq:defphit}. 
 \end{lemma}
\begin{proof} 
Let $\uY_{k,t}$ denote the $k$-line ensemble with zero boundary condition on $[-t/2,t/2]$.  We take $t>0$ large and $k\in\bbN$ as a function of $t$ defined by  $k = \lfloor \left(\tfrac1A \log t\right)^{1/a}\rfloor$,  where $a>0$ and $A>0$ are constants to be tuned later, so that 
\begin{equation}\label{eq:tandk}
\exp(Ak^a)\le t\le \exp(A(k+1)^a)\,.
\end{equation}
Since,  by monotonicity, \[
Y^j_{\infty,t}\ge Y^j_{k,t},\]
it will be sufficient to show that  
\begin{equation}\label{eq:clama}
X^j(s)-Y^j_{k,t}(s) 
\le  \varphi(t)^2\,,\qquad j=1,\dots,i\,,\quad s\in[-S,S].
\end{equation}

We are going to use an enhanced version of the coupling argument in the proof of Lemma \ref{lem:coupling}.
We start by ensuring that  all lines below the $k$-th line $X^k$ are below a certain height throughout the whole interval. From Markov inequality and Corollary \ref{cor:logtight}, for any $u>0$ and $t$ large enough one has
\begin{equation}\label{eq:ft1}
\bbP(F_t)\ge 1- \frac{C k\log(t)}{u}\,,\qquad 
F_t := \left\{\max_{s\in[-t/2,t/2]}X^{k+1}(s) \le u\l^{-k/3}
\right\}\,.
\end{equation}
Using \eqref{eq:tandk} 
and setting $u:=\l^{k/6}$, one has 
\begin{equation}\label{eq:ft2}
\bbP(F_t)\ge 1- C\l^{-k/7},
\end{equation}
for all $t$ large enough. In other words, the $(k+1)$-th line exceeds the  height $\l^{-k/6}$ in $[-t/2,t/2]$ with probability $O(\l^{-k/7})$. 

Next, consider the discrete time steps $s_j=-t/2+2j-1$, $j=1,\dots,j_{\max}$,
where $j_{\max}=\lfloor t/2\rfloor$. Similarly, let $u_j = t/2 -2j+1$, $j=1,\dots,j_{\max}$. 
Set $v:=k^b$, with $b>0$ to be fixed later, and consider the indexes 
\begin{align}\label{eq:coupp2} 
&\ell_* = \inf\{j\in\{1,\dots,j_{\max}\}: Y_{k,t}^i(s_j)\ge v\l^{-(i-1)/3}\,,\;\forall i=1,\dots,k\}\,,\\
&r_* = \inf\{j\in\{1,\dots,j_{\max}\}: \; Y_{k,t}^i(u_j)\ge v\l^{-(i-1)/3}\,,\;\forall i=1,\dots,k\}\,.
\end{align}
Accordingly, we define the random times $\t_\ell=s_{\ell_*}$ and $\t_r=u_{r_*}$, and  
consider the events 
\begin{align}\label{eq:coupp3} 
&B_t=\{\t_\ell<-t/4\,,\;\t_r>t/4 \}\,,\qquad G_t=\cap_{i=1}^kG^{\,i}_t\,,\\
&G^{\,i}_t = \{X^i(\t_\ell) \le v\l^{-(i-1)/3}\}\cap \{X^i(\t_r) \le v\l^{-(i-1)/3}\}. 
\end{align}

We start with independent samples of the infinite line ensemble $\uX$ and the $k$-line ensemble $\uY_{k,t}$ with zero boundary conditions on $[-t/2,t/2]$, and consider the event 
\[
E_t=F_t\cap G_t\cap B_t.
\] 
If $E_t$ occurs, then by definition of the events $F_t, G_t$ and $B_t$,  at the stopping domain $[\t_\ell(u),\t_r(u)]$ the boundary values of $\uY_{k,t}$ are higher than the boundary values of the top $k$ lines of $\uX$ and the $(k+1)$-th line of $\uX$ does not exceed height $\l^{-k/6}$. Thus, on this event we may resample $\{(\uX^{i}(s),\uY^{i}_{k,t}(s)),\; s\in[\t_\ell(u),\t_r(u)]\}$
according to the monotone coupling which by construction guarantees that 
\begin{align}\label{eq:coupp31} 
X^i(s)\le Y^i_{k,t}(s) + \l^{-k/6}, \qquad s\in[-t/4,t/4],\;\;\;i=1,\dots, k,
\end{align}
with probability one. This follows from the fact that the resampling of the top $k$ lines of $\uX$, on the event $F_t$, using stochastic domination, can be performed with an effective floor at height $\l^{-k/6}$, and therefore one  can compare with the lines $Y^i_{k,t}$, $i=1,\dots,k$ with floor at zero, with an overall shift by $ \l^{-k/6}$.  
If instead $E_t$ does not occur, then we keep the independent samples of $(\uX,\uY_{k,t})$ everywhere. The previous observations and the strong BG property  guarantee that this is a valid coupling.  Note that above we crucially used the fact that 
$[\t_\ell(u),\t_r(u)]$ is a stopping domain.

From \eqref{eq:coupp31} it follows that, if $  \l^{-k/6} \le \varphi(t)^2$, then on the event $E_t$ we have \eqref{eq:clama}.
It remains to show that  $\bbP(E_t^c)\le \varphi(t)^2$ holds. One has 
\begin{align}\label{eq:coupp4} 
\bbP(E^c_t) & \le  \bbP(F_t^c)+\bbP(G_t^c) + \bbP( B_t^c) .
\end{align} 
Now, \eqref{eq:ft2} shows that $ \bbP(F_t^c)\le C\l^{-k/7}$, while using the independence of $(\t_\ell,\t_r)$ and the line ensemble $\uX$, Corollary \ref{cor:stretched} and a union bound imply 
 \begin{align}\label{eq:couppe4} 
\bbP(G_t^c) \le2 k \exp(-ck^{3b/2}),
\end{align}
for some constant $c>0$. 
We turn to an upper bound on $\bbP(B_t^c)$.

Note that the points $s_j$ are the midpoints of the intervals $I_j$ of size $2$, where 
\[
I_j = [-t/2+2(j-1), -t/2 + 2j]\,,\quad j=1,\dots,j_{\max}.
\] 
By monotonicity, see Lemma \ref{lem:mono} and Remark \ref{rem:nonct}, we can replace $\uY_{k,T}$ by the ensemble obtained by pinning all $k$ paths at zero height at the endpoints of the intervals $I_j$.
Call $Z^i$, $i=1,\dots,k$, the lines of this pinned process.
We define the index $\ell_*'$ as the smallest $j$ such that $Z^i(s_j)\ge v\l^{-(i-1)/3},\; \forall i=1,\dots,k$. 
Since the intervals are independent, one has 
\begin{align}\label{eq:coup4} 
\bbP(\t_\ell\ge -t/4)\le \bbP(\ell_*'>\lfloor t/8\rfloor)\le (1-p_k(v))^{\lfloor t/8\rfloor}\le e^{-p_k(v)\lfloor t/8\rfloor},
\end{align}
where $p_k(v)$ is defined as the probability that the $\l$-tilted  $k$-line ensemble $\uY_{k,2}$  with zero boundary conditionson the interval $[-1,1]$ satisfies 
\begin{align}\label{eq:pkvu0} 
Y^i_{k,2}(0)\ge v\l^{-(i-1)/3},\qquad i=1,\dots,k.
\end{align} 
By symmetry, the same bound applies to 
$\bbP(\t_r\le t/4)$. It follows that
\begin{align}\label{eq:coupas4} 
\bbP(B_t^c)\le \bbP(\t_\ell\ge -t/4) + \bbP(\t_r\le t/4)
\le  2e^{-p_k(v)\lfloor t/8\rfloor}.
\end{align}
In Lemma \ref{lem:pkvu} below we show that for all $k\in\bbN, v\ge 1$,
$p_k(v)\ge  e^{ - Ckv^2 }$,
for some constant $C>0$, see also the remark following that lemma for a discussion on the true asymptotics. Taking $v=k^b$, 
we obtain
\begin{align}\label{eq:trap4c} 
p_k(v)\ge  e^{ - Ck^{1+2b} } .
\end{align}
Summarizing, since $t\ge e^{Ak^a}$ as in \eqref{eq:tandk}, if $a=1+2b$ and $A$ is large enough, it follows that 
 \begin{align}\label{eq:coupass4} 
\bbP(B_t^c)\le   2\exp{(-c\,e^{k^a})},
\end{align}
for some constant $c>0$. 

Therefore, from \eqref{eq:ft2}, \eqref{eq:couppe4} and \eqref{eq:coupass4} we have 
 \begin{align}\label{eq:coupas5} 
\bbP(E_t^c)\le   \l^{-k/7} + k \exp(-ck^{3b/2}) + 2\exp{(-c\,e^{k^a})}\,.
\end{align}
By choosing $b=2/3$, $a = 1+2b=7/3$, we obtain \eqref{eq:corre4} with $\varphi(t)=e^{-ck}$ for some new constant $c>0$.  By adjusting the value of $c$ we can also ensure that $  \l^{-k/6} \le \varphi(t)^2$ as required to have $\D_t\le  \varphi(t)^2$ on the event $E_t$.  By our definition of $k$ in \eqref{eq:tandk}, this proves the desired correlation decay with $\varphi(t) =  \exp\left[-c(\log t)^\d\right]$, and $\d=1/a=3/7$. 
\end{proof}

\begin{lemma}\label{lem:pkvu} 
There exists a constant $C>0$, such that for any $k\in\bbN$, $v\ge 1$, the probability $p_k(v)$ defined in \eqref{eq:pkvu0}  satisfies
\begin{align}\label{eq:pkvu2} 
p_k(v)\ge  e^{ - Ckv^2 }.
\end{align}
\end{lemma}

Note that the above bound is expected to be suboptimal when $k$ is much larger than  $\log v,$ where the one point estimate in Corollary \ref{cor:stretched} is expected to yield a bound $e^{ - Ckv^{3/2+o(1)}}.$ On the other hand when $k$ is much smaller, say $k=1$, this is the right behavior as on a fixed domain size, the one point tail of the top line is indeed Gaussian and not given by Theorem \ref{th:stretched} .

\begin{proof}
We use a geometric construction involving the nested trapezoidal shapes defined as follows. Given parameters $0<a_i<b_i<1$ and $h_i>0$, consider the $i$-th trapezoid $\Pi_i$ defined by the piecewise linear function 
\begin{align}\label{eq:trapezoid}
 \Pi_i(t)=\begin{cases}
0 & t\in[-1,-b_i]\cup[b_i,1]
\\
\tfrac{h_i(b_i+t)}{b_i-a_i} & t\in[-b_i,-a_i]\\
h_i & t\in[-a_i,a_i]\\
\tfrac{h_i(b_i-t)}{b_i-a_i} & t\in[a_i,b_i]
\end{cases}
\end{align}
We choose
\[
b_i:=\l^{-2(i-1)/3}, \quad a_i:=\frac12(\l^{2/3}+1)\l^{-2i/3}, \quad h_i:=v\l^{-(i-1)/3}.
\]
 Note that $a_i$ is the midpoint in the interval $[b_{i+1},b_i]$, and one has $\Pi_{i+1}\prec \Pi_i$; see Figure \ref{fig:trap10}.

\begin{figure}[h]
\center
\begin{tikzpicture}[scale=0.8]
\draw [black] (0,0) -- (6,0) -- (6,4) -- (0,4) -- cycle;
\draw (1/2,0) -- (6-1/2,0) -- (5,3) -- (1,3) -- cycle;
\draw (3/2,0) -- (6-3/2,0) -- (4,2) -- (2,2) -- cycle;

\draw [dashed,black]  (1,0) -- (1,3) ;
\draw [dashed,black]  (5,0) -- (5,3) ;
\draw [dashed,black]  (0,3) -- (1,3) ;
\draw [dashed,black]  (0,2) -- (2,2);

  \node[shape=circle, draw=black, fill = black, scale = .2]  at (0,3) {}; 
  \node[scale = .7]  at (-.3,3) {$h_i$};
   \node[shape=circle, draw=black, fill = black, scale = .2]  at (0,2) {}; 
  \node[scale = .7]  at (-.5,2) {$h_{i+1}$};

 \node[shape=circle, draw=black, fill = black, scale = .2]  at (0,4) {}; 
  \node[scale = .7]  at (-.3,4) {$2v$};
  \node[scale = .8]  at (-.3,-.2) {$-1$};
   \node[scale = .8]  at (6.1,-.2) {$1$};

  \node[shape=circle, draw=black, fill = black, scale = .22]  at (1/2,0) {}; 
  \node[scale = .7]  at (.27,-1/4) {$-b_i$};
  \node[shape=circle, draw=black, fill = black, scale = .22]  at (1/2,0) {}; 
  \node[scale = .7]  at (1.65,-1/4) {$-b_{i+1}$};
\node[shape=circle, draw=black, fill = black, scale = .22]  at (1,0) {}; 
\node[scale = .7]  at (.85,-1/4) {$-a_i$};
\node[shape=circle, draw=black, fill = black, scale = .22]  at (3/2,0) {}; 
\node[shape=circle, draw=black, fill = black, scale = .22]  at (6-3/2,0) {}; 
\node[shape=circle, draw=black, fill = black, scale = .22]  at (5,0) {}; 
\node[shape=circle, draw=black, fill = black, scale = .22]  at (5.5,0) {}; 
\node[scale = .7]  at (4.5,-1/4) {$b_{i+1}$};
\node[scale = .7]  at (5,-1/4) {$a_{i}$};
\node[scale = .7]  at (5.5,-1/4) {$b_{i}$};
\node[shape=circle, draw=black, fill = black, scale = .22]  at (3,0) {}; 
\node[scale = .7]  at (3,-1/4) {$0$};

\end{tikzpicture}
\caption{The trapezoids $\Pi_i$ and $\Pi_{i+1}$.}
\label{fig:trap10}
\end{figure}
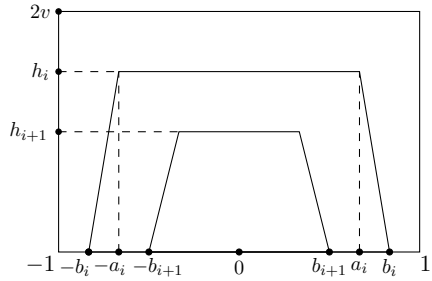

Consider the $\l$-tilted  $k$-line ensemble $\uY_{k,2}$ on the interval $[-1,1]$ with zero boundary conditions, and let  $A$ denote the event that for each $i=1,\dots,k$ one has the $i$-th line above the $i$-th trapezoid:
\begin{align}\label{eq:aevpkva} 
 A = \{  
Y^i_{k,2}\succ\Pi_i,\;\; i=1,\dots,k\}.
\end{align} 
By construction, one has $p_k(v)\ge \bbP(A)$. By monotonicity we may replace the ensemble $\uY_{k,2}$ by the ensemble $\tilde \uY_{k}$ obtained by pinning the $i$-th path $Y^i_{k,2}$ at zero outside of the interval $[-a_{i-1},a_{i-1}]$, for each $i=1,\dots,k$, where we set $a_0=1$. Thus, we have 
\begin{align}\label{eq:aevpkv1} 
 p_k(v)\ge \bbP\left(  
\tilde Y^i_{k}\succ\Pi_i,\;\; i=1,\dots,k\right),
\end{align} 
see Figure \ref{fig:traps2}.

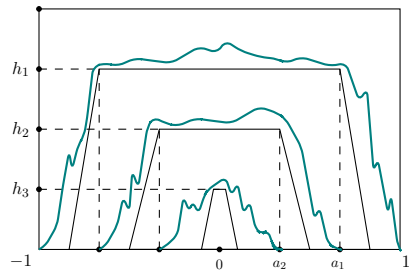
\begin{figure}[h]
\center
\begin{tikzpicture}[scale=0.8]
\draw [black] (0,0) -- (6,0) -- (6,4) -- (0,4) -- cycle;
\draw (1/2,0) -- (6-1/2,0) -- (5,3) -- (1,3) -- cycle;
\draw (3/2,0) -- (6-3/2,0) -- (4,2) -- (2,2) -- cycle;
\draw (2.7,0) -- (6-2.7,0) --  (3.1,1) -- (2.9,1) -- cycle;

\draw [dashed,black]  (1,0) -- (1,3) ;
\draw [dashed,black]  (0,3) -- (1,3) ;
\draw [dashed,black]  (4,0) -- (4,2) ;
\draw [dashed,black]  (5,0) -- (5,3) ;
\draw [dashed,black]  (2,0) -- (2,2) ;
\draw [dashed,black]  (0,1) -- (3,1) ;
\draw [dashed,black]  (0,2) -- (2,2) ;

  \node[shape=circle, draw=black, fill = black, scale = .17]  at (0,3) {}; 
   \node[shape=circle, draw=black, fill = black, scale = .17]  at (4,0) {}; 
     \node[shape=circle, draw=black, fill = black, scale = .17]  at (5,0) {}; 
  \node[shape=circle, draw=black, fill = black, scale = .17]  at (3,0) {}; 
\node[shape=circle, draw=black, fill = black, scale = .17]  at (0,1) {}; 
\node[shape=circle, draw=black, fill = black, scale = .17]  at (0,2) {};

  \node[scale = .6]  at (-.3,3) {$h_1$};
  \node[scale = .6]  at (-.3,1) {$h_3$};
  \node[scale = .6]  at (-.3,2) {$h_2$};
 \node[shape=circle, draw=black, fill = black, scale = .17]  at (0,4) {}; 
  \node[scale = .6]  at (-.3,-.2) {$-1$};
   \node[scale = .6]  at (6.1,-.2) {$1$};

  \node[shape=circle, draw=black, fill = black, scale = .17]  at (2,0) {}; 
 \node[shape=circle, draw=black, fill = black, scale = .17]  at (1,0) {}; 
\node[scale = .5]  at (4,-1/4) {$a_2$};
\node[scale = .5]  at (4.97,-1/4) {$a_1$};
\node[scale = .5]  at (3,-1/4) {$0$};

\draw[thick,teal,rounded corners=0.9mm] (0,0)--(0.1,0.05)--(0.3,0.4)--(0.4,1.07)--(0.5,1.63)--(0.55,1.33)--(0.65,1.73)--(0.75,1.9)--(0.8,2.19)--(.9,3.04)--(1.1,3.1)--(1.3,3.05)--(1.6,3.18)--(1.8,3.15)--(2.2,3.1)--(2.66,3.37)--(2.96,3.27)--(3.1,3.47)--(3.36,3.3)--(3.68,3.2)--(3.96,3.17)--(4.1,3.1)--(4.36,3.2)--(4.84,3.05)--(5.1,3.1)--(5.32,2.47)--(5.45,2.67)--(5.56,1.3)--(5.63,1.5)--(5.8,0.6)--(5.85,0.9)--(5.96,0.07)--(6,0);

\draw[thick,teal,rounded corners=0.9mm] (1,0)--(1.1,0.05)--(1.3,0.4)--(1.4,0.87)--(1.5,0.63)--(1.55,1.3)--(1.65,1.03)--(1.75,1.7)--(1.8,2.03)--(1.9,2.2)--(2.1,2.1)--(2.3,2.05)--(2.6,2.18)--(2.8,2.15)--(3.2,2.1)--(3.66,2.37)--(3.96,2.27)--(4.1,2.17)--(4.36,1.7)--(4.58,0.5)--(4.66,0.7)--(4.8,0.1)--(4.92,0.02)--(5,0);

\draw[thick,teal,rounded corners=0.8mm] (2,0)--(2.1,0.05)--(2.3,0.24)--(2.4,0.87)--(2.5,0.63)--(2.55,0.8)--(2.65,0.6)--(2.75,0.9)--(2.8,0.99)--(3.1,1.2)--(3.2,0.9)--(3.3,1.1)--(3.4,.6)--(3.5,.79)--(3.6,.34)--(3.7,0.25)--(3.8,.32)--(3.88,0.1)--(3.9,0.05)--(3.95,0.005)--(4,0);

\end{tikzpicture}
\caption{A sketch of the event in \eqref{eq:aevpkv1}  in the case $k=3$.}
\label{fig:traps2}
\end{figure}
Let $E_i$ denote the event that $\tilde Y^i_{k}\succ\Pi_i$. Conditioned on $E_i$, to compute the probability of $E_{i+1}$ one may, by monotonicity, impose a ceiling at height $h_i$ and remove all paths below the $(i+1)$-th path. This shows that 
\begin{align}\label{eq:aevpkv2} 
 p_k(v)\ge \prod_{i=1}^k \g_i(v) \,,
 \end{align} 
where $\g_i(v)$ is defined as the probability that the random path $Z_i$ on the interval $[-a_{i-1},a_{i-1}]$, with zero boundary condition, with a floor at zero, a ceiling at height $h_{i-1}$, with area tilt $\l^{i-1}$, satisfies $Z_i\succ\Pi_i$. Here we may set $h_0=\l^{1/3}v$ for the ceiling acting on the top path. It remains to show that 
there exists an absolute constant $C>0$ such that for all $v\ge 1$, for all $i\in\bbN$, 
 \begin{align}\label{eq:aevpkv3} 
 \g_i(v) \ge e^{-Cv^2}\,.
 \end{align} 
To prove \eqref{eq:aevpkv3} we observe that by the scaling from Lemma \ref{lem:scaling} applied to a single path, one has that $ \g_i(v)$ is the probability that the random path $\xi$  on the interval $[-a,a]$, with zero boundary condition, with a floor at zero, a ceiling at height $\l^{1/3}v$, with area tilt $1$, satisfies $\xi\succ\Pi$, where $\tilde \Pi$ is the trapezoid defined by \eqref{eq:trapezoid} with $a_i,b_i,h_i$ replaced by $\tilde a,\tilde b,\tilde h$ respectively, and where 
\begin{gather*}
 \tilde b =\l^{2(i-1)/3}b_i=1, \quad \tilde h =\l^{(i-1)/3}h_i=v,\\
a=\l^{2(i-1)/3}a_{i-1} = \tfrac12(\l^{2/3}+1), \quad \tilde a = \l^{2(i-1)/3}a_{i}= \tfrac12(1+\l^{-2/3}). 
\end{gather*}
Therefore,
  \begin{align}\label{eq:aevpkv4} 
 \g_i(v) \ge  \bbP(  
\xi\succ\tilde \Pi)\,,
 \end{align} 
 see Figure \ref{fig:trap1}.

\begin{figure}[h]
\center
\begin{tikzpicture}[scale=0.8]
\draw [black] (0,0) -- (6,0) -- (6,4) -- (0,4) -- cycle;
\draw (1/2,0) -- (6-1/2,0) -- (5,3) -- (1,3) -- cycle;
\draw [dashed,black]  (1,0) -- (1,3) ;
\draw [dashed,black]  (0,3) -- (1,3) ;

  \node[shape=circle, draw=black, fill = black, scale = .17]  at (0,3) {}; 
  \node[scale = .6]  at (-.3,3) {$v$};
 \node[shape=circle, draw=black, fill = black, scale = .17]  at (0,4) {}; 
  \node[scale = .6]  at (-.6,4) {$\l^{1/3}v$};
  \node[scale = .6]  at (-.2,-.2) {$-a$};
   \node[scale = .6]  at (6.1,-.2) {$a$};

  \node[shape=circle, draw=black, fill = black, scale = .17]  at (1/2,0) {}; 
  \node[scale = .6]  at (.5,-1/4) {$-1$};
  \node[shape=circle, draw=black, fill = black, scale = .17]  at (6-1/2,0) {}; 
  \node[scale = .6]  at (6-.5,-1/4) {$1$};
\node[shape=circle, draw=black, fill = black, scale = .17]  at (1,0) {}; 
\node[scale = .6]  at (1.05,-1/4) {$-\tilde a$};

\draw[thick,teal,rounded corners=0.9mm] (0,0)--(0.1,0.05)--(0.3,0.4)--(0.4,1.07)--(0.5,1.63)--(0.55,1.33)--(0.65,1.73)--(0.75,1.9)--(0.8,2.19)--(.9,3.04)--(1.1,3.1)--(1.3,3.05)--(1.6,3.18)--(1.8,3.15)--(2.2,3.1)--(2.66,3.37)--(2.96,3.27)--(3.1,3.47)--(3.36,3.3)--(3.68,3.2)--(3.96,3.17)--(4.1,3.1)--(4.36,3.2)--(4.84,3.05)--(5.1,3.1)--(5.32,2.47)--(5.45,2.67)--(5.56,1.3)--(5.63,1.5)--(5.8,0.6)--(5.85,0.9)--(5.96,0.07)--(6,0);
\end{tikzpicture}
\caption{The rescaled trapezoid $\tilde \Pi$ and a sketch of the event in \eqref{eq:aevpkv4}.}
\label{fig:trap1}
\end{figure}
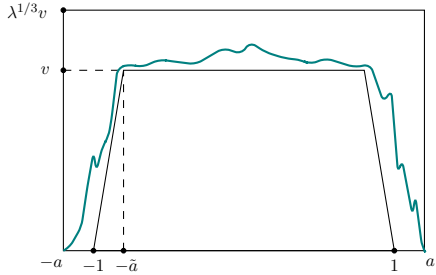
To estimate the probability $  \bbP(  
\xi\succ\tilde \Pi)$ we may write 
\begin{align}\label{eq:aevpkv55} 
 \bbP(  
\xi\succ\tilde \Pi) &= \frac{\bbE(  e^{-\int_{-a}^a\xi_0(s)\dd s};\;
\tilde \Pi\prec \xi_0\prec \l^{1/3}v)}{\bbE(  e^{-\int_{-a}^a\xi_0(s)\dd s};\;
\xi_0\prec \l^{1/3}v)}\\
& \ge e^{-2a\l^{1/3}v} \bbP(  
\tilde \Pi\prec \xi_0\prec \l^{1/3}v)
\,,
 \end{align} 
 where now $\xi_0$ is the simple Brownian excursion on $[-a,a]$, without area tilt and without ceiling.  Standard computations for  Brownian excursions now imply that 
 \begin{align}\label{eq:aevpkv6} 
\bbP(  
\tilde \Pi\prec \xi_0\prec \l^{1/3}v)\ge e^{-Cv^2}, 
\end{align} 
for all $v\ge1$, and for some constant $C=C(\l)>0$ for all $\l>1$. By adjusting the value of the constant $C$ we arrive at the desired estimate \eqref{eq:aevpkv3}. 

\end{proof}

We conclude this section with the proof of Theorem \ref{lowertail} relying on the construction in \eqref{eq:trapezoid}. 

\subsection{Proof of Theorem \ref{lowertail}} We will employ monotonicity and the strategy outlined in Figure \ref{fig:traps2}. 
Note first that as a simple consequence of ordering, 
\begin{equation}\label{push}
\bbP(X^{1}(0)\le \e)=\bbP(X^{i}(0)\le \e, \,\, i\le K)
\end{equation}
 for any $K.$
Now suppose there was no interaction between the paths, i.e., if the $i$th path was an independent FS diffusion with tilting factor $\lambda^{i-1}$. Then from the above display, the proof would follow  if we can prove that for each such FS diffusion $Y_{FS}^i$,
$$\bbP(Y_{FS}^i(0)\le \e)\le \e^2, \text{\,say}.$$ 
Now note that the above bound cannot possibly be true for all $i$ since for any $i$, the typical
value of $Y_{FS}^i(0)$ is of order $\lambda^{-(i-1)/3}.$
However our proof will indeed proceed by showing that an approximate independence holds true and that the above estimate also continues to hold, as long as $i\le K=c\log (1/\e)$ for some appropriate $c.$

We start with a generalized version of \eqref{FSlowertail}.
\begin{lemma}\label{lem:ecube}
There exists a universal constant $C>0,$ such that for any $i\in\bbN$,
$$\bbP\(Y_{FS}^i(0)\le \e\)\le C\e^3\l^{i-1}.$$
In particular, $\bbP\(Y_{FS}^i(0)\le \e\)\le C\e^2$, for all $i$ such that $\l^{i-1}\le 1/\e.$
\end{lemma}
\begin{proof}
By monotonicity, $\bbP(Y_{FS}^i(0)\le \e)\le \bbP(\widetilde{Y}_{FS}^i(0)\le \e)$ where $\widetilde{Y}_{FS}^i$ has the same law as $Y_{FS}^i$ except now it is pinned at zero at $\pm \lambda^{-2(i-1)/3}.$
By Brownian scaling, 
\[
\bbP\(\widetilde{Y}_{FS}^i(0)\le \e\)=\bbP\(\widetilde{Y}_{FS}^0(0)\le \e \lambda^{(i-1)/3}\),
\] 
where $\widetilde{Y}_{FS}^0$ is simply a FS diffusion with tilting parameter $1$ and pinned at zero at $\pm 1.$ 
Thus it is a Brownian excursion $B(s)$ on $[-1,1]$ with the tilting factor
\begin{equation}\label{areatiltfactore}
\exp{\(-\int_{-1}^1 B(s)\dd s\)}.
\end{equation}
By Gaussian tails of the maximum of a Brownian excursion on $[-1,1]$, the partition function is a constant,  say $1/C_1$, and since the above density term is bounded by $1$, 
$$
\bbP\(\widetilde{Y}_{FS}^0(0)\le \e \lambda^{(i-1)/3}\)
\le C_1 \bbP\(B(0)\le \e \lambda^{(i-1)/3}\)\le C_2 \e^3\lambda^{i-1},$$
for some constant $C_2>0$, 
where the last estimate follows from standard repulsion estimates for Brownian excursion which e.g. can be found in \cite{ito1996diffusion}.
\end{proof}

Equipped with this estimate we return to the proof of Theorem \ref{lowertail}.
Now by monotonicity, the right hand side of \eqref{push}, can only increase on removal of all lines beyond $K$ as well as introducing boundary conditions where now the $i^{th}$ curve is pinned at $0$ at $\pm x_i$ where $x_K= c_\l\lambda^{-2K/3}$ and $x_{i-1}-x_i= C_\l \lambda^{-2i/3}$ where $c_\l,C_\l$ are constants depending on $\l$ to be fixed later. 

An advantage of this pinning is that the independent Ferrari-Spohn diffusions $\widetilde Y^i_{FS}$ with pinnings at $\pm x_i$, have a non-trivial chance of being non-intersecting. {As depicted in Figure \ref{fig:traps2}}  for the $i^{th}$ curve consider a trapezoid $\Pi_i$ of width and height comparable to $\lambda^{-2i/3}$ and $\lambda^{-i/3}$ respectively. The exact construction can be done as in \eqref{eq:trapezoid}, with the choice of the parameter $v$ being $1$. By construction $\Pi_i$ are nested and by scaling arguments it is easy to see that $\widetilde Y^i_{FS}$ stays within $\Pi_{i-1}\setminus \Pi_{i}$ with a probability bounded away from $0$ independently of $i,$ say $c.$

Thus, using also Lemma \ref{lem:ecube},
\begin{align*}
\bbP\(X^{i}(0)\le \e, \,\, i\le K\)\le \frac{\prod_{i=1}^K \bbP\(\widetilde Y_{FS}^i (0)\le \varepsilon\)}{\prod_{i=1}^K \bbP\(\widetilde Y_{FS}^i (0)\in \Pi_{i-1}\setminus \Pi_{i}\)}\le \frac{\varepsilon^{2K}}{c^{K}},
\end{align*}
for some universal constant $c,$ as long as $\lambda^{-K}\ge {\varepsilon},$ which finishes the proof.

%
%

\section{Uniqueness}\label{sec:uniqueness}
This section is devoted to the proof of Theorem \ref{th:uniqueness}.
We will again employ the reverse coupling strategy in the way outlined before the proof of Theorem \ref{th:corrdecay}, with the difference that now $\uX$ is replaced by the paths, denoted $\uZ$, 
with law $\nu$, where $\nu$ stands for a generic uniformly tight $\l$-tilted measure satisfying the asymptotic pinning condition in \eqref{eq:asymppin}.

Note that as commented before the proof of Theorem \ref{th:corrdecay}, one point tail estimates were crucial in the implementation of this strategy, for which we relied on Corollary  \ref{cor:stretched} which delivers such estimates for the zero boundary LE.  Thus  in this general case, our first step is to establish a counterpart tail bound simply under the assumption of UT. The most important part of the proof is a bootstrapping result which establishes that 
\[UT\implies UC,\] 
reversing \eqref{implication}, and hence showing their equivalence.
This will
allow us to appeal to our previously established bounds in Lemma \ref{univtail1} yielding the
tail estimates needed to carry out the program.
\subsection{Uniformly tight implies uniformly confined}
The first step delivers the above mentioned crucial input which shows that an asymptotically pinned at zero,  uniformly tight LE is also uniformly  confined (recall from Definitions  \ref{def:asympin}, \ref{ut}, and \ref{ust}). 
\begin{theorem}[$UT\implies UC$]\label{ergodicle}Suppose $\nu$ is the law of a UT $\l-$tilted LE satisfying \eqref{eq:asymppin} to be denoted by $\uZ=\{Z^{i}\}_{i\ge 1}$. Then $\nu$ is also UC. Namely, 
there exists a constant  $C>0$ such that for all $k\ge 0$ and $s\in\bbR$,
\begin{equation}\label{eq:tight2z}
\bbE\left[Z^{k+1}(s)\right]\le C\l^{-k/3} .
\end{equation}
Moreover,
 for all $k\ge 0,S>0$ and $s\in\bbR$,
\begin{equation}\label{eq:tightob}
\bbE\left[\max_{t\in[-S,S]}Z^{k+1}(s+t) \right]\le C\l^{-k/3}\log(1+|\l^{2k/3}S|).
\end{equation}
\end{theorem}
Once Theorem \ref{ergodicle} is available,  we obtain stretched exponential tails by Lemma \ref{univtail1}: 
There exist constants $c,C>0$ and $\a>0$ such that for any $k$ and any $s\in \bbR$, $x>0$
\begin{equation}\label{eq:stretchedexpa10}
\bbP\left[Z^{k+1}(s)> \lambda^{-k/3}x\right]\le C\,e^{-c\,x^\a}.
\end{equation}

To prove Theorem \ref{ergodicle}, we start with a tightness result where the ensemble is assumed to have bounded boundary data.  
\begin{proposition}\label{pro:tightmom}
Fix $L>0$ and let $Z_{\infty,T}^i(s)$, $i=1,2,\dots$, $s\in[-T,T]$ denote the infinite $\l$-tilted LE on $[-T,T]$ with boundary data all equal to $L$ defined as a monotone limit of the finite $\l$-tilted LE $Z_{n,T}^i(s)$, $i=1,2,\dots n$, $s\in[-T,T]$, with boundary data all equal to $L$. There exists a constant  $T_0=T_0(L)>0$ such that for all $T\ge T_0$, for all $k\ge 0$,
{\begin{equation}\label{eq:tight2y}
\bbE\left[Z^{k+1}_{\infty,T}(0)\right]\le C_0\l^{-k/3} ,
\end{equation}}
where $C_0$ is an absolute constant (independent of $L$).
Moreover,
 for all $0\le S\le T/2$,
{
\begin{equation}\label{eq:tightoa}
\bbE\left[\max_{s\in[-S,S]}Z^{k+1}_{\infty,T}(s) \right]\le C_0\l^{-k/3}\log(1+|\l^{2k/3}S|).
\end{equation}
}
\end{proposition}
Notice the important difference with the apparently similar estimate in Corollary \ref{cor:logtight}.
In the latter, the line ensemble has either free or zero boundary conditions, while here it has a fixed height boundary $L$, the same for every path, and this is a huge difference for low lying paths. 
  
{We reiterate that the field $\uZ_{\infty,T}$ appearing in the above statement is well defined, by taking limits $n\to\infty$ of an ensemble of $n$ paths and using monotonicity.}    Using a floor at $L$ and Theorem  \ref{th:logtight} one can immediately obtain rough upper bounds on the paths $Z^k_{\infty,T}$, but Proposition
\ref{pro:tightmom} is a much finer statement providing information on the rapid decrease of the $k$-th path to the correct scale $\l^{-k/3}$.

  Before proving Proposition \ref{pro:tightmom} we quickly finish the proof of Theorem \ref{ergodicle}.

\begin{proof}[Proof of Theorem \ref{ergodicle} assuming Proposition \ref{pro:tightmom}]  
Let $\uZ=\{Z^i,\,i\ge 1\}$ denote the random lines with law $\nu$. Given the uniform tightness hypothesis on $\nu$, for any $i\in\bbN$ let $L_i>0$ be such that $\sup_{t\in\bbR} \bbP(Z^{1}(t)\ge L_i)\le 2^{-i}.$
Now recalling the statement of Proposition \ref{pro:tightmom}, let $T_i=T_0(L_i)$ and define the events $$A_{i}:=\bigcap_{j=i}^\infty \{Z^{1}(\pm T_j)\le L_j\}.$$
By Borel-Cantelli lemma, $\mathbf{1}(A_i)\uparrow  \mathbf{1}, a.s.$, $i\to\infty$. 
Finally, fixing $S$ as in the second part of the theorem, let $i_0$ be such that $T_i>2S$ when $i\ge i_0.$

Now for any $i\ge i_0,$ observe that $A_i$ is measurable with respect to $\sigma(Z^{1}(s): |s|\ge T_i)$.
Further, to apply the asymptotic pinning hypothesis \eqref{eq:asymppin}, for any $\e >0,$ let $B(n,T,\e)$ be the event that  $\sup_{s\in [-T,T]}Z^{n+1}(s)\le \e$.

Then it follows by Proposition \ref{pro:tightmom} that for any $k<n-1$,
\begin{align*}
\bbE\left[Z^{k+1}(0) \mathbf{1}_{A_i} \mathbf{1}_{B (n,T_i,\e)}\right]&\le 
\bbE\left[Z^{k+1}(0)\mid \mathbf{1}_{A_i}=1,\mathbf{1}_{B (n,T_i,\e)}=1 \right]\le C_0\l^{-k/3} +\e \,\,\,\ \text{and}\\
\bbE\left[\max_{t\in[-S,S]}Z^{k+1}(t) \mathbf{1}_{A_i}\mathbf{1}_{B (n,T_i,\e)}\right]&\le 
\bbE\left[\max_{t\in[-S,S]}Z^{k+1}(t)\mid \mathbf{1}_{A_i}=1, \mathbf{1}_{B (n,T_i,\e)}=1 \right]\\&\le C_0\l^{-k/3}\log(1+|\l^{2k/3}S|)+\e.
\end{align*}

Note that above we used the BG property of $\uZ$ on the domain $[-T_i,T_i]$ with the top $n-1$ lines, and that on $A_i \cap B (n,T_i,\e)$ the hypothesis of Proposition \ref{pro:tightmom} is satisfied by the definition of $T_i$ as a function of $L_i$. The boundary condition induced by $Z^n$ can be raised to a floor at $\e$ (conditioning on $B (n,T_i,\e)$ permits that) 
leading to the extra $\e$ appearing in the bounds above.
Moreover,  note that the constant $C_0$ does not depend on $L_i$.

We will now finish the proof with a couple of applications of the monotone convergence theorem.
Note that since $\underline Z$ is assumed to satisfy \eqref{eq:asymppin}, for any $\e, T_i,$ we have $\mathbf 1_{B (n,T_i,\e)}\uparrow \mathbf{1}, a.s.$ as $n \to \infty.$
Therefore,
 \begin{align*}
\bbE\left[Z^{k+1}(0) \mathbf{1}_{A_i}\right]&\le C_0\l^{-k/3} +\e, \,\,\,\ \text{and,}\\
\bbE\left[\max_{t\in[-S,S]}Z^{k+1}(t) \mathbf{1}_{A_i}\right]&\le C_0\l^{-k/3}\log(1+|\l^{2k/3}S|)+\e.
\end{align*}

Since $\e$ was arbitrary, using  $\mathbf 1_{A_i}\uparrow \mathbf{1}, a.s.$  proves \eqref{eq:tight2z} and \eqref{eq:tightob} for the case $s=0$.  Applying the above argument to the LE shifted by any real number $s$ finishes the proof.
\end{proof}

We now turn to the proof of the key Proposition \ref{pro:tightmom}.
\subsection{Nested shapes and curved maxima}
To prove Proposition \ref{pro:tightmom} we want to define nested shapes $\phi_{i+1} \prec \phi_{i}$, where for each $i\ge 0$, $\phi_i:[-T,T]\mapsto \bbR_+$ is a function such that $\phi_i(s)$ restricted to $s\in[-T/2,T/2]$ is of the form $\l^{-i/3}\left(\log(1+|s\l^{2i/3}|)\right)$ up to a suitable vertical shift and such that with high probability $Z^{i+1}_{\infty,T}\prec \phi_i$ throughout $[-T,T]$.  To do this we will consider the ensemble $Z^i_{n,T}$ with $n$ lines and with the same boundary condition $L$ on $[-T,T]$, and will prove estimates that are uniform in the number of lines $n$. We start by defining a scaled version of the nested shapes $\phi_i$.  

Fix some parameters $\d\in(0,1/2)$, $\a\in(3/4,1)$, $L\in\bbR_+$. We define, for all $T$ large, and for any given choice of the vertical shift parameters $h_i\in\bbR_+$, 
 \begin{gather}\label{eq:psii}
\psi_i^{h_i}(s) = \begin{cases}
L_i + \e_i & s\in[-T_i,-T_i+T_i^\a]\cup[T_i-T_i^\a,T_i]\\
h_i + \Psi(s) & s\in[-T_i+T_i^\a,T_i-T_i^\a]
\end{cases}\\
\e_i:= \l^{i/4}\vee T^\d\,,\quad L_i:=L\l^{i/3}\,,\quad T_i:=T\l^{2i/3},\quad \Psi(s):=D\log(1+|s|),
\end{gather}
for some constant $D$ to be fixed below. We assume that the $h_i$ are such that the curved part of $\psi_i^{h_i}$ is lower than the flat part, that is 
 \begin{gather}\label{eq:psiia}
h_i + \Psi(T_i-T_i^\a) \le L_i + \e_i\,.
\end{gather}
Since $T$ will be taken large enough, this will be automatically satisfied 
in what follows.
Thus, for each $i\in\bbN$, $\psi^{h_i}_i$ is a function on the interval $[-T_i,T_i]$ which looks 
approximately like Figure \ref{fig:shape1}.

\begin{figure}[h]
\center
\begin{tikzpicture}[scale=0.8]
\draw [black] (0,0) -- (0,4);
\draw [black] (1,4) -- (0,4);
\draw [black] (1,1.8) -- (1,4);
\draw [black] (6,0) -- (6,4);
\draw [black] (5,4) -- (6,4);
\draw [black] (5,1.8) -- (5,4);
\draw (0,0) -- (6,0);
\draw [dashed,black]  (0,.6) -- (3,.6) ;
\draw [dashed,black]  (3,0) -- (3,.6) ;
\draw [dashed,black]  (0,3) -- (1,3) ;
\draw [dashed,black]  (1,1.8) -- (1,0) ;

 \node[shape=circle, draw=black, fill = black, scale = .17]  at (0,.6) {}; 
   \node[scale = .7]  at (-.4,.6) {$h_i$};
   \node[shape=circle, draw=black, fill = black, scale = .17]  at (3,0) {}; 
   \node[scale = .7]  at (3,-.25) {$0$};
  \node[shape=circle, draw=black, fill = black, scale = .17]  at (0,3) {}; 
  \node[scale = .7]  at (-.4,3) {$L_i$};
 \node[shape=circle, draw=black, fill = black, scale = .17]  at (0,4) {}; 
  \node[scale = .7]  at (-.8,4) {$L_i+\e_i$};
  \node[scale = .6]  at (-.3,-.2) {$-T_i$};
   \node[scale = .6]  at (6.1,-.2) {$T_i$};

  \node[shape=circle, draw=black, fill = black, scale = .17]  at (0,0) {}; 
\node[shape=circle, draw=black, fill = black, scale = .17]  at (1,0) {}; 
\node[scale = .5]  at (1.08,-1/4) {$-T_i+T_i^\a$};
\node[shape=circle, draw=black, fill = black, scale = .17]  at (1,1.8) {};
\node[shape=circle, draw=black, fill = black, scale = .17]  at (5,1.8) {};

\draw [black] plot [smooth] coordinates { (1,1.8) (1.1,1.75) (1.5,1.68) (2,1.5) (2.5,1.2) (2.7,1) (2.8,.88) (3,.6)};
\draw [black] plot [smooth] coordinates { (5,1.8) (6-1.1,1.75) (6-1.5,1.68) (6-2,1.5) (6-2.5,1.2) (6-2.7,1) (6-2.8,.88) (6-3,.6)};
%
\end{tikzpicture}
\caption{A sketch of the function $\psi^{h_i}_i$ in \eqref{eq:psii}}
\label{fig:shape1}
\end{figure}
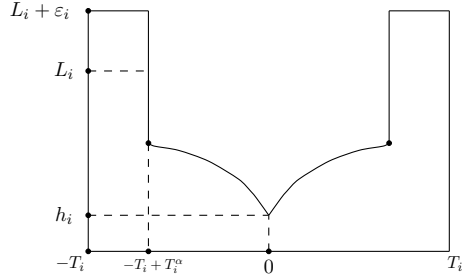

Next, we define the rescaled function $\phi_i^{h_i}(s) := \l^{-i/3}\psi_i^{h_i}(\l^{2i/3}s)$, $s\in[-T,T]$. In other words, 
 \begin{gather}\label{eq:phii}
\phi_i^{h_i}(s) = \begin{cases}
L + \tilde\e_i & s\in[-T,-T+t_i]\cup[T-t_i,T]\\
\tilde h_i +  \l^{-i/3}\Psi(s\l^{2i/3})  & s\in[-T+t_i,T-t_i]
\end{cases}\\
\tilde \e_i:= \l^{-i/3}\e_i,\qquad \tilde h_i:= \l^{-i/3}h_i\,,\qquad t_i:=T^\a\l^{-(1-\a)2i/3}.
\end{gather}
Notice that  for each $i\ge 0$, $\phi^{h_i}_i$ is a function on the interval $[-T,T]$.
The key technical step in the proof of Proposition \ref{pro:tightmom} can be stated as follows. 
\begin{lemma}\label{lem:th}
Let $Z_{n,T}^i(s)$, $i=1,\dots,n$, $s\in[-T,T]$ denote the $n$-line ensemble on $[-T,T]$ with boundary data all equal to $L$. There exist some absolute constants  $c,C>0$, a collection of random variables $h_0,h_1,\dots$ with $\sup_i\bbE[h_i]\le C$,  and a constant $T_0=T_0(L)$, such that for all 
$T\ge T_0$, for all $n\in\bbN$, the ensemble $\{Z^i_{n,T}\}$ can be coupled to $(h_0,\dots,h_{n-1})$ in such a way that for all $0\le i< n$, with probability at least $1-Ce^{-(T\l^i)^c}$,
\begin{equation}\label{eq:th101}
Z^{j+1}_{n,T}(s) \le \phi_j^{h_j}(s)\,,\qquad s\in[-T,T]\,,\;\;\;j=i,i+1,\dots,n-1.
\end{equation}
\end{lemma}
Let us first show how Proposition \ref{pro:tightmom} follows from Lemma \ref{lem:th}. By monotonicity, one has convergence of the line ensemble $\{Z^i_{n,T}\}_{i\ge 1}$ to $\{Z^i_{\infty,T}\}_{i\ge 1}$ and therefore it is sufficient to establish \eqref{eq:tight2y} and \eqref{eq:tightoa} for the ensemble $\{Z^i_{n,T}\}_{i\ge 1}$, 
with a constant $C$ independent of $n$. Let $A_i$ denote the event that \eqref{eq:th101} holds under the coupling introduced in Lemma \ref{lem:th}, so that  $\bbP(A_i^c)\le Ce^{-(T\l^i)^c}$. Notice that, 
\begin{align}\label{eq:th102a}
 \bbE\left[\max_{s\in[-T,T]}Z^{i+1}_{n,T} \,\ind_{A_i^c}\right]&  \le \bbP(A_i^c)^{1/2}\,\bbE\left[\left(\max_{s\in[-T,T]}Z^i_{\infty,T}\right)^2\right]^{1/2}\\
 & \le C \,e^{-\tfrac12(T\l^i)^c}\, T^\d,
\end{align}
for some new constants $C,\d>0$, where we have used a rough estimate on the second moment of the maximum of the top path in the infinite line ensemble $Z_{\infty,T}$. The latter can be obtained for instance adding a floor at the boundary height $L$ and then using the bounds in Corollary \ref{cor:roughS} for the zero boundary measure with a floor at zero. Note that since we are taking $T$ large enough depending on $L$, the above constant $C$ can be take independent of $L$. 

Thus, to prove Proposition \ref{pro:tightmom} it is sufficient to establish 
\begin{align}\label{eq:th102}
  \bbE\left[\max_{s\in[-S,S]}Z^{i+1}_{n,T}(s)\,\ind_{A_i}\ \right]\le C\l^{-i/3}\log(1+|\l^{2i/3}S|)\,
  \end{align}
  with an absolute constant $C$, for all $S\in[-T/2,T/2]$. On the event $A_i$, we know that  $Z^i_{n,T}(s)\le \phi_i^{h_i}(s) $, and since $t_i\le T/2$, one has 
\begin{gather}\label{eq:phii21}
 \l^{i/3}\,Z^{i+1}_{n,T}(s) \le  h_i + D\log(1+|s|\l^{2i/3})\,,\qquad    s\in[-T/2,T/2].
\end{gather}
Using $\sup_i\bbE[h_i]<\infty$ we obtain the desired estimate \eqref{eq:th102}.
 It remains to prove Lemma \ref{lem:th}. 
 
 \subsection{Proof of Lemma \ref{lem:th}}
 Let us begin with a high level description of the ideas of the proof. One consequence of our definitions is that the rescaled functions $\phi_i:=\phi^{h_i}_i$ defined in \eqref{eq:phii} describe nested shapes, see \eqref{eq:phii24} and Figure \ref{fig:evzi}. Towards the proof of \eqref{eq:th101}, by monotonicity, for each $i$  we may raise the boundary of the path $Z^{i+1}_{n,T}$ from $L$ to  $L+\tilde \e_{i+1}$, and 
 thus the probability of the event \eqref{eq:th101} is bounded below by the probability  that this modified path, for each $i$,  is contained in the region between the two shapes $\phi_{i}$ and $\phi_{i+1}$, see Figure \ref{fig:wiz}. If we start from the bottom path then a simple recursion together with monotonicity allows us to reduce the problem to the analysis of a single path, call it $Z_{i+1}$, with a floor at $\phi_{i+1}$, for each $i$. The key step is then an estimate of the probability that this path $Z_{i+1}$ satisfies $Z_{i+1} \prec  \phi_i$. To prove this, we use Lemma \ref{lem:comingdown} to ensure that $Z_{i+1}$ comes down rapidly enough from the boundary, so that we may resample in the bulk of our interval using boundary data that are only a small vertical shift from the floor $\phi_{i+1}$. At this point we apply a first moment estimate for the maximum of an area tilted path with a concave floor that was derived in \cite{CIW18}. It is precisely this confinement estimate, combined with scaling and recursion, that will allow us to construct the random variables $h_i$ representing the random vertical shifts in the definition of the nested shapes $\phi_i$, and to guarantee that their expected values satisfy the required uniform bounds $\sup_i \bbE[h_i]\le C$. 
 
 We now turn to the technical details of the proof.   
As mentioned above,  we proceed recursively, starting from the bottom path $Z^n_{n,T}$, and use monotonicity at each step to raise the boundary condition and to impose a floor. Namely, if we assume that $Z^{i+2}_{n,T}$ satisfies $Z^{i+2}_{n,T} \prec \phi_{i+1}$ in $[-T,T]$, then we can dominate $Z^{i+1}_{n,T}$ by a single random line $Z_{i+1}$ with area tilt $\l^i$, with boundary data $\phi_{i+1}(-T)=\phi_{i+1}(T)=L+\tilde\e_{i+1}$ and with a floor at $\phi_{i+1}(s)$, $s\in[-T,T]$. We are going to show that this random line $Z_{i+1}$ satisfies 
 \begin{gather}\label{eq:phii22}
 \bbP\left(Z_{i+1} \prec  \phi_{i}\right)\ge 1-Ce^{-(T\l^i)^c}\,,
\end{gather}
for some constants $c,C>0$ independent of $n$ and $i$. Once we have such a bound, we can obtain the desired estimate, starting from $Z_n$, where the floor is $\phi_{n}:= 0$, and then stopping at the desired path $i$, to obtain that the probability of the event $A_i$ in \eqref{eq:th101} satisfies, by the union bound, 
\begin{gather}\label{eq:phii23}
 \bbP\left(A_i \right)\ge 1-\sum_{j=i}^nCe^{-(T\l^j)^c}\ge 1-C'e^{-(T\l^i)^c}\,,
\end{gather}
for some new constant $C'$. Thus the proof of Lemma \ref{lem:th} has been reduced to the statement \eqref{eq:phii22}, for every $0\le i< n$. 
 
 To prove \eqref{eq:phii22} 
 we observe that since $Z_{i+1}$ is conditioned to stay above $\phi_{i+1}$, we must ensure  that 
 \begin{gather}\label{eq:phii24}
 \phi_i(s)\ge \phi_{i+1}(s), \qquad s\in[-T,T],
 \end{gather}
 otherwise the event in \eqref{eq:phii22} is empty.
The condition \eqref{eq:phii24} indeed holds  as soon as the variables $h_i$ satisfy a suitable relation. More precisely, notice that $t_{i+1}=\l^{-\frac23(1-\a)}t_i<t_i$, and using also \eqref{eq:psiia},  one has 
 \begin{gather}\label{eq:phii31}
\l^{i/3}(\phi_i(s) - \phi_{i+1}(s)) \ge \e_i - \l^{-1/3}\e_{i+1} \,,
\end{gather}
for $s\in[-T,-T+t_i]\cup[T-t_i,T]$, and 
\begin{gather}\label{eq:phii32}
\l^{i/3}(\phi_i(s) - \phi_{i+1}(s)) = h_i - \l^{-1/3}h_{i+1} + \Psi(s\l^{2i/3})-\l^{-1/3}\Psi(s\l^{2{(i+1)}/3})\,,
\end{gather}
for $s\in[-T+t_i,T-t_i]$. By definition of $\e_i$ we see that there exists $c_1=c_1(\l)>0$ such that 
\begin{gather}\label{eq:phii33}
 \e_i - \l^{-1/3}\e_{i+1}\ge c_1 \e_i .
 \end{gather} 
Notice also that if $C_1$ is a large enough constant then  
\begin{gather}\label{eq:phii34}
\l^{-1/3}\Psi(x\l^{2/3})\le C_1+ \l^{-1/4}\Psi(x)\,,\qquad x\in\bbR.
 \end{gather} 
 Therefore,
 \begin{gather}\label{eq:phii35}
\l^{i/3}(\phi_i(s) - \phi_{i+1}(s)) \ge  h_i - \l^{-1/3}h_{i+1} - C_1+ c_2\Psi(s\l^{2i/3}),
\end{gather}
for $s\in[-T+t_i,T-t_i]$, where $c_2=1-\l^{-1/4}>0$. In particular, if 
 \begin{gather}\label{eq:phii36}
 h_i \ge C_1+ \l^{-1/3}h_{i+1}
 \end{gather}
  we obtain $\phi_i\ge \phi_{i+1}$ as desired; see Figure \ref{fig:evzi}.

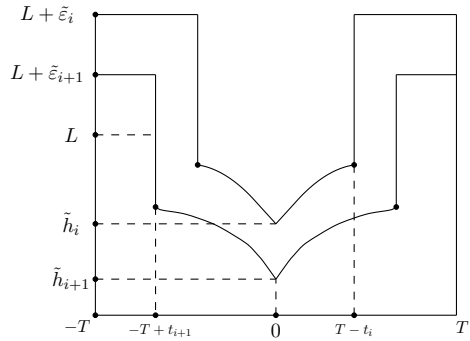
\begin{figure}[h]
\center
\begin{tikzpicture}[scale=0.8]
\draw [black] (0,0) -- (0,5);
\draw [black] (1.7,5) -- (0,5);
\draw [black] (1.7,2.5) -- (1.7,5);
\draw [black] (6,0) -- (6,5);
\draw [black] (6-1.7,5) -- (6,5);
\draw [black] (6-1.7,2.5) -- (6-1.7,5);
\draw [black] (1,4) -- (0,4);
\draw [black] (1,1.8) -- (1,4);
\draw [black] (5,4) -- (6,4);
\draw [black] (5,1.8) -- (5,4);
\draw (0,0) -- (6,0);
\draw [dashed,black]  (0,.6) -- (3,.6) ;
\draw [dashed,black]  (0,1.52) -- (3,1.52) ;
\draw [dashed,black]  (3,0) -- (3,.6) ;
\draw [dashed,black]  (0,3) -- (1,3) ;
\draw [dashed,black]  (1,1.8) -- (1,0) ;
\draw [dashed,black]  (6-1.7,2.5) -- (6-1.7,0) ;

 \node[shape=circle, draw=black, fill = black, scale = .17]  at (0,.6) {}; 
   \node[scale = .7]  at (-.4,.6) {$\tilde h_{i+1}$};
    \node[shape=circle, draw=black, fill = black, scale = .17]  at (0,1.52) {}; 
   \node[scale = .7]  at (-.4,1.52) {$\tilde h_{i}$};

   \node[shape=circle, draw=black, fill = black, scale = .17]  at (3,0) {}; 
   \node[scale = .7]  at (3,-.25) {$0$};
  \node[shape=circle, draw=black, fill = black, scale = .17]  at (0,3) {}; 
  \node[scale = .7]  at (-.4,3) {$L$};
 \node[shape=circle, draw=black, fill = black, scale = .17]  at (0,4) {}; 
  \node[scale = .7]  at (-.8,4) {$L+\tilde\e_{i+1}$};
   \node[shape=circle, draw=black, fill = black, scale = .17]  at (0,5) {}; 
  \node[scale = .7]  at (-.8,5) {$L+\tilde\e_{i}$};

  \node[scale = .6]  at (-.3,-.2) {$-T$};
   \node[scale = .6]  at (6.1,-.2) {$T$};

  \node[shape=circle, draw=black, fill = black, scale = .17]  at (0,0) {};
  \node[shape=circle, draw=black, fill = black, scale = .17]  at (1.7,2.5) {};
  \node[shape=circle, draw=black, fill = black, scale = .17]  at (6-1.7,2.5) {}; 
\node[shape=circle, draw=black, fill = black, scale = .17]  at (1,0) {}; 
\node[scale = .5]  at (1.08,-1/4) {$-T+t_{i+1}$};
\node[shape=circle, draw=black, fill = black, scale = .17]  at (6-1.7,0) {}; 
\node[scale = .5]  at (4.3,-1/4) {$T-t_{i}$};
\node[shape=circle, draw=black, fill = black, scale = .17]  at (1,1.8) {};
\node[shape=circle, draw=black, fill = black, scale = .17]  at (5,1.8) {};

\draw [black] plot [smooth] coordinates { (1,1.8) (1.1,1.75) (1.5,1.68) (2,1.5) (2.5,1.2) (2.7,1) (2.8,.88) (3,.6)};
\draw [black] plot [smooth] coordinates { (5,1.8) (6-1.1,1.75) (6-1.5,1.68) (6-2,1.5) (6-2.5,1.2) (6-2.7,1) (6-2.8,.88) (6-3,.6)};

\draw [black] plot [smooth] coordinates { (1.7,2.5) (1.8,2.48) (1.9,2.45) (2,2.41) (2.1,2.36)(2.2,2.3)(2.3,2.23)(2.4,2.15)(2.5,2.06)(2.6,1.96) (2.7,1.85)(2.8,1.74)(2.9,1.63) (3,1.52)};
\draw [black] plot [smooth] coordinates { (6-1.7,2.5) (6-1.8,2.48) (6-1.9,2.45) (6-2,2.41) (6-2.1,2.36)(6-2.2,2.3)(6-2.3,2.23)(6-2.4,2.15)(6-2.5,2.06)(6-2.6,1.96) (6-2.7,1.85)(6-2.8,1.74)(6-2.9,1.63) (6-3,1.52)};
\end{tikzpicture}
\caption{A sketch of the two shapes $\phi_{i+1}\prec \phi_i$. }
\label{fig:evzi}
\end{figure}

To prove \eqref{eq:phii22} we first apply the scaling as in Lemma \ref{lem:scaling}. It follows that 
the probability in  \eqref{eq:phii22} can be evaluated as 
\begin{gather}\label{eq:phii42}
 \bbP\left(Z_{i+1} \prec  \phi^{h_{i}}_{i}\right)= \bbP\left(W_{i+1} \prec  \psi^{h_{i}}_{i}\right)\,,
\end{gather}
where now $W_{i+1}$, $i=0,\dots,n-1$, denotes the single line on $[-T_i,T_i]$, with area tilt $1$, with boundary data 
\[
\l^{i/3}\phi^{h_{i+1}}_{i+1}(-T)=\l^{i/3}\phi^{h_{i+1}}_{i+1}(T)=L_i+\l^{i/3}\tilde\e_{i+1}=L_i+\hat\e_i,
\]
where we define $\hat\e_{i}:=\l^{i/3}\tilde\e_{i+1}=\l^{-1/3}\e_{i+1}$, 
and with floor given by 
\begin{gather}\label{eq:phii43}
\xi_i(s):=\l^{i/3}\phi^{h_{i+1}}_{i+1}(s\l^{-2i/3})=\l^{-1/3}\psi_{i+1}^{h_{i+1}}(s\l^{2/3})\,,\quad  s\in[-T_i,T_i].
\end{gather}
Note that $\xi_i(0)=\l^{-1/3}h_{i+1}$. We refer to Figure \ref{fig:wiz} for a drwaing of the event in \eqref{eq:phii42}.
\begin{figure}[h]
\center
\begin{tikzpicture}[scale=0.8]
\draw [black] (0,0) -- (0,5);
\draw [black] (1.7,5) -- (0,5);
\draw [black] (1.7,2.5) -- (1.7,5);
\draw [black] (6,0) -- (6,5);
\draw [black] (6-1.7,5) -- (6,5);
\draw [black] (6-1.7,2.5) -- (6-1.7,5);
\draw [black] (1,4) -- (0,4);
\draw [black] (1,1.8) -- (1,4);
\draw [black] (5,4) -- (6,4);
\draw [black] (5,1.8) -- (5,4);
\draw (0,0) -- (6,0);
\draw [dashed,black]  (0,1.8) -- (1,1.8) ;
\draw [dashed,black]  (3,0) -- (3,.6) ;
\draw [dashed,black]  (0,4.5) -- (1.7,4.5) ;
\draw [dashed,black]  (1.7,0) -- (1.7,2.5) ;
\draw [dashed,black]  (1,1.8) -- (1,0) ;

 \node[shape=circle, draw=black, fill = black, scale = .17]  at (1.7,2.5) {}; 

   \node[shape=circle, draw=black, fill = black, scale = .17]  at (3,0) {}; 
   \node[scale = .7]  at (3,-.25) {$0$};
 \node[shape=circle, draw=black, fill = black, scale = .17]  at (0,4) {}; 
  \node[scale = .7]  at (-.8,4) {$L_i+\hat\e_{i}$};
   \node[shape=circle, draw=black, fill = black, scale = .17]  at (0,5) {}; 
  \node[scale = .7]  at (-.8,5) {$L_i+\e_{i}$};

  \node[scale = .6]  at (-.3,-.2) {$-T_i$};
   \node[scale = .6]  at (6.1,-.2) {$T_i$};

  \node[shape=circle, draw=black, fill = black, scale = .17]  at (0,0) {};
  \node[shape=circle, draw=black, fill = black, scale = .17]  at (1.7,2.5) {};
  \node[shape=circle, draw=black, fill = black, scale = .17]  at (6-1.7,2.5) {}; 
\node[shape=circle, draw=black, fill = black, scale = .17]  at (1,0) {}; 
\node[scale = .55]  at (.8,-1/4) {$-\hat T_i$};
\node[shape=circle, draw=black, fill = black, scale = .17]  at (1.7,0) {}; 
\node[scale = .55]  at (1.8,-.3) {$- T_i+T_i^\a$};
\node[shape=circle, draw=black, fill = black, scale = .17]  at (1,1.8) {};
\node[shape=circle, draw=black, fill = black, scale = .17]  at (0,1.8) {};
  \node[scale = .7]  at (-.4,1.8) {$ v_i$};
  \node[shape=circle, draw=black, fill = black, scale = .17]  at (0,4.5) {};
  \node[scale = .7]  at (-.4,4.5) {$ w$};
\node[shape=circle, draw=black, fill = black, scale = .17]  at (5,1.8) {};

\draw [black] plot [smooth] coordinates { (1,1.8) (1.1,1.75) (1.5,1.68) (2,1.5) (2.5,1.2) (2.7,1) (2.8,.88) (3,.6)};
\draw [black] plot [smooth] coordinates { (5,1.8) (6-1.1,1.75) (6-1.5,1.68) (6-2,1.5) (6-2.5,1.2) (6-2.7,1) (6-2.8,.88) (6-3,.6)};

\draw [black] plot [smooth] coordinates { (1.7,2.5) (1.8,2.48) (1.9,2.45) (2,2.41) (2.1,2.36)(2.2,2.3)(2.3,2.23)(2.4,2.15)(2.5,2.06)(2.6,1.96) (2.7,1.85)(2.8,1.74)(2.9,1.63) (3,1.52)};
\draw [black] plot [smooth] coordinates { (6-1.7,2.5) (6-1.8,2.48) (6-1.9,2.45) (6-2,2.41) (6-2.1,2.36)(6-2.2,2.3)(6-2.3,2.23)(6-2.4,2.15)(6-2.5,2.06)(6-2.6,1.96) (6-2.7,1.85)(6-2.8,1.74)(6-2.9,1.63) (6-3,1.52)};

\draw[smooth,thick,teal,rounded corners=0.5mm] (0,4)--(0.1,4+0.05)--(0.3,4+0.2)--(0.4,4+0.3)--(0.5,4+0.27)--(0.55,3+1.23)--(0.65,3+1.1)--(0.75,3+1.05)--(0.8,3+1.15)--(.9,4.19)--(1.1,4.05)--(1.3,3.8)--(1.4,3.2)--(1.5,2.4)--(1.6,2)--(1.7,1.9)--(1.8,1.85)--(2.2,1.7)--(2.66,1.37)--(2.96,1.27)--(3.1,1.47)--(3.36,1.47)--(3.68,1.42)--(3.96,1.59)--(4.1,1.8)--(4.2,1.94)--(4.36,1.74)--(4.54,2.8)--(4.64,2.5)--(4.84,3.2)--(4.9,4.1)--(5.1,4.15)--(5.32,4.22)--(5.45,4.17)--(5.56,4.2)--(5.63,4.24)--(5.8,4.1)--(5.85,4.2)--(5.96,4.07)--(6,4);
\end{tikzpicture}
\caption{Sketch of a realization of the path $W_i$ satisfying the event $W_i\prec \psi^{h_i}_i$ in \eqref{eq:phii42}. The path is conditioned to stay above the floor $\xi_i$ from \eqref{eq:phii43}.}
\label{fig:wiz}
\end{figure}
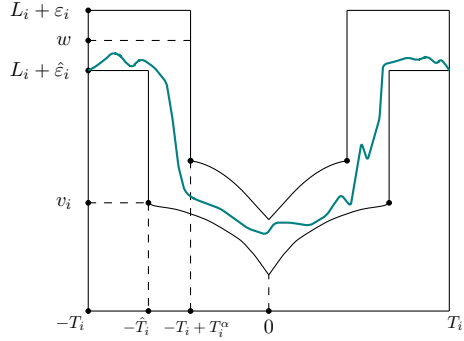

Note that from \eqref{eq:phii31} we know that $\e_i-\hat\e_i\ge c_1\e_i$. Define $w: = L_i + \hat \e_i + \frac12\,c_1\e_i$, where the constant $c_1$ is the same as in  \eqref{eq:phii31}. Let us also define 
\[
\hat T_i: = \l^{2{i}/3}(T-t_{i+1}) = T_i - T_i^\a\l^{2(1-\a)/3} ,
\]
and note that $\hat T_i > T_i-T_i^\a$; see Figure \ref{fig:wiz}. We first show that
\begin{gather}\label{eq:phii45}
\bbP\left(W_{i+1}(s) \le w \,,\;\forall s\in[-T_i, T_i]\right)\ge1- Ce^{-(T\l^i)^c}\,.
\end{gather}
This can be achieved by considering a global floor at height $L_i + \hat \e_i$ and using the bound in Lemma \ref{fsmax10} on the probability that the maximum of a Ferrari-Spohn diffusion exceeds the height $\frac12\,c_1\e_i$ in an interval of size $2T_i$.

Thus, by monotonicity we can now restrict to a path in $[-\hat T_i,\hat T_i]$ with boundary data at $w$, a ceiling at $w$ and a floor at $\xi_i(s)$, $s\in[-\hat T_i,\hat T_i]$. Call $\tilde W_i$ the associated path. Set $v_i=\xi_i(\hat T_i)$ as in Figure \ref{fig:wiz} and write 
\[
k_i=\max\{\tilde W_i(-T_i+T_i^\a),\tilde W_i(T_i-T_i^\a)\} - v_i.
\] 
We will show that 
\begin{gather}\label{eq:phii46}
\bbE\left[k_i\right]\le C\,,
\end{gather}
for some constant $C$. 
Note that we are considering a single line with area tilt 1, with boundary at $w$ and ceiling at $w$ and we ask for it to come down to $v_i+k_i$ within a time 
\[
\D:= T_i-T_i^\a - \hat T_i = T_i^\a(\l^{2(1-\a)/3} -1)\ge c_2T_i^\a ,
\]
for some constant $c_2>0$. We may add a floor at $v_i$ and prove the estimate \eqref{eq:phii46} using the argument from Lemma \ref{lem:comingdown}, 
applied in this case with $T$ replaced by $\hat T_i$ and with $T^\a$ replaced by $\D$ (and with a floor shifted upwards by $v_i$). 
We obtain that with probability at least $1 -Ce^{-cT_i^c}$ the height goes down from the boundary reaching below height $v_i+K$ within the intervals $[-\hat T_i,-T_i + T_i^\a]$ and
$  [\hat T_i,T_i - T_i^\a]$, for an absolute constant $K$ (which plays the role of the constant called $L$ in Lemma \ref{lem:comingdown}). Let us call $E$ this favorable event. On the complement of $E$ we may estimate $k_i$ by using the ceiling $w$ so that $k_i\le w\le L_i+ \e_i\le 2L_i$ so that 
\[
\bbE\left[k_i;E^c\right]\le 2CL_i e^{-cT_i^c}\le C,
\] 
for some absolute constant $C$ provided $T$ is large enough depending on $L$.
On the event $E$ one can use boundary conditions $v_i+K$ on a stopping domain 
which includes $[-T_i+T_i^\a,T_i-T_i^\a]$, and we can further impose a floor at $v_i+K$. We then obtain the estimate $\bbE[k_i;E]\le C$ by simply applying the one-point estimate for the single line with Ferrari-Spohn distribution. This proves \eqref{eq:phii46}.

Finally, once \eqref{eq:phii46} is available, we need to show that for a suitable choice of the random variables $h_i$ one has the desired domination within the region $[-T_i+T_i^\a,T_i-T_i^\a]$. This will be achieved by means of the curved maxima with concave floor established in \cite{CIW18}. By monotonicity and the definition of the random variables $k_i$, see also Figure \ref{fig:wiz}, 
we can consider a single line with left and right boundary conditions at height $v_i+k_i$ and with floor at $\hat k_i + \xi_i(s)$, for  $s\in[-T_i+T_i^\a,T_i-T_i^\a]$, where we define
\[
\hat k_i:= k_i + v_i - \xi_i(-T_i+T_i^\a).
\]
 In order to apply the result from \cite{CIW18} we shift it vertically by $-\hat k_i-\xi_i(0)$ and consider the path $\widehat W_{i+1}$ on $[-T_i+T_i^\a,T_i-T_i^\a]$, with area tilt $1$,  with floor at $\hat \xi_i(s):=\xi_i(s)-\xi_i(0)$, with left and right boundary conditions at height $\hat\xi_i(T_i-T_i^\a)$. Then we consider the curved maximum
\begin{gather}\label{eq:phii48}
\G_i:=\max_{s\in[-T_i+T_i^\a,T_i-T_i^\a]} \left[\widehat  W_{i+1} - D\log(1+|s|)
\right]_+\,,
\end{gather}
where $[\cdot]_+$ denotes the positive part. 
From the estimate established in \cite[Eq.\ (3.8)]{CIW18} it follows that 
\begin{gather}\label{eq:phii46a}
\bbE[\G_i]\le C\,,
\end{gather}
if the constant $D$ is large enough. As already noted after Theorem \ref{th:logtight} the bound in \cite{CIW18} is derived for curved maxima with respect to any concave power law function $\Psi(s)=|s|^\e$, $\e\in(0,1/2)$ but it can be extended, with the same proof, to the case of $\Psi(s)=D\log(1+|s|)$ if $D$ is large enough.  In conclusion, by setting
\begin{gather}\label{eq:phii49}
h_i = \xi_i(0)+\hat k_i + \G_i= \l^{-1/3} h_{i+1} + \hat k_i + \G_i
\,,
\end{gather}
we have obtained that the line $W_{i+1}$ from \eqref{eq:phii42}
satisfies  
\[
W_{i+1}(s)\le h_i+D\log(1+|s|) = \psi_i^{h_i}(s)\,,\qquad  s\in[-T_i,T_i],
\]
 where $\psi_i^{h_i}$ is defined in \eqref{eq:psii}. Recursively, starting at $h_n:=0$, \eqref{eq:phii49} defines the sequence $h_i$, and one obtains the desired bound $\sup_i \bbE[h_i]<\infty$, provided the same holds for $\hat k_i,\G_i$.  Noting that $\hat k_i\le 1+ k_i$ and using the bounds \eqref{eq:phii46} and \eqref{eq:phii46a}, combined with \eqref{eq:phii42}, this proves \eqref{eq:phii22}.

 \subsection{Proof of Theorem \ref{th:uniqueness}} 
 To finish the proof of Theorem \ref{th:uniqueness} we first need the counterpart result of Lemma \ref{reversecoupling}. We use the same notation, that is, we will use $\uY_{\infty,t}$ to denote the infinite LE 
with zero boundary conditions on $[-t/2,t/2]$ and note that  its infinite volume limit has law $\mu^0$.
Moreover, we write $\uZ$ for the  uniformly tight LE in the statement of Theorem \ref{th:uniqueness}. By monotonicity, the latter stochastically dominates the former. However, as in Lemma \ref{reversecoupling}, the next result shows that with high probability a reversing coupling can be constructed such that the top $i$ lines  are close to each other, for any fixed $i$, provided $t$ is large enough. 

\begin{lemma}\label{lem:cnamle}
There exists $\d>0$ such that given any $i$ and $S\ge 0$, for all large enough $t$, there exists a coupling of $ \uY_{\infty,t}$  and $\uZ$, and an event $E_t$ such that on $E_t$ one has 
\[
Z^j(s)-\uY^j_{\infty,t}(s) 
\le  \varphi(t)^2\,,\qquad j=1,\dots,i\,,\quad s\in[-S,S],
\]
and such that $\bbP(E_t^c)\le \varphi(t)^2$, where $\varphi(t):=\exp(-c\log(t)^{\d})$ for some absolute constant $c>0$. 

%
\end{lemma} 
Note that the definition of $\varphi(t)$ has been altered a bit compared to Lemma \ref{reversecoupling} where $\delta$ was taken to be $3/7$. This is because we will be relying on \eqref{eq:stretchedexpa10} which is a consequence of Theorem \ref{ergodicle} and Lemma \ref{univtail1} which does not exactly specify the exponent in the tail estimate, in contrast with Theorem \ref{th:stretched}, where the tail behavior is explicit. Nonetheless, this will not affect the outcome.

\begin{proof}[Proof of Lemma \ref{lem:cnamle}]The proof is verbatim the argument in the proof of Lemma \ref{reversecoupling} once two modifications are made. To deduce \eqref{eq:ft2}, the input from Corollary \ref{cor:logtight}  needs to be replaced by the counterpart input \eqref{eq:tightob}. Similarly \eqref{eq:stretchedexpa10} replaces Corollary \ref{cor:stretched} to deliver \eqref{eq:couppe4}. 
\end{proof}

We now have all the ingredients to complete the proof of Theorem \ref{th:uniqueness}.
Let $\uX$ denote the paths with law $\mu^0$. 
As already mentioned, under the monotone coupling of $\uZ$ and $\uX$, with probability one, for 
any $i\ge 1$, 
\begin{equation}\label{moncouple}
Z^i \succ X^i.
\end{equation}
On the other hand, by Lemma \ref{lem:cnamle} and the monotonicity $\uY_{\infty,t}\prec \uX$, for all fixed $i\in\bbN$ and $S>0$,  for any $\varepsilon> 0,$ there exists another coupling such that with probability at least $1-\varepsilon$, 
\begin{equation}\label{revcouple}
Z^j(s)\le X^j(s)+ 
  \varepsilon\,,\qquad j=1,\dots,i\,,\quad s\in[-S,S].
\end{equation}
The proof is essentially done at this point barring a few measure theoretic details which we seek to provide now. Borrowing notation from the proof of Corollary \ref{cor:DLZ}, fix $m\in\bbN$ and let $\calS=(s_1,\dots,s_m)\in [-t,t]^m$, $\calI=(i_1,\dots,i_m)\in\{1,\dots,i\}^m$, and let $\calT=(t_1,\dots,t_m)\in \bbR_+^m$ and consider the event
$$E=E(\calS,\calI,\calT)=\{\uY\in\Omega:\; Y^{i_j}(s_j) > t_j\,,\; j=1,\dots,m\}.
$$
It suffices to show that for each $m\in\bbN$, and for each choice of $\calS,\calI,\calT$,
\begin{align}\label{eq:fdconv1} 
 \mu^0(E)=\bbP(\uX\in E)=\bbP(\uZ\in E)=\nu(E).
 \end{align}
 Since the event $E$ is increasing, by \eqref{moncouple} it follows that
\[
\bbP(\uX\in E)\le \bbP(\uZ\in E).
\]
Inequality \eqref{revcouple} on the other hand implies that for any $\varepsilon>0$
\[
\bbP(\uX\in E_\varepsilon)\ge \bbP(\uZ\in E)-\e.
\]
where  $$E_{\varepsilon}=\{\uY\in\Omega:\; Y^{i_j}(s_j) > t_j-\varepsilon\,,\; j=1,\dots,m\}.
$$
Choosing $t_j$ such that $E$ is a continuity set for $\uX$ (note that this might not hold for at most a countable collection of $t_j$'s, having fixed the $s_j$'s), we get 
\[
\bbP(\uX\in E)= \bbP(\uZ\in E).
\]
Since such continuity sets generate the entire sigma algebra of finite dimensional distributions, we are done.
\qed
  
  \bigskip
  
We finish with the proof of Corollary \ref{freeunique}.

\begin{proof} By Theorem \ref{tightLE12},  the family of measures $\mu^f_{n,T}$ as $n$ and $T$ increase to $\infty$ is tight. Further, by Corollary \ref{cor:logtight}, and a similar reasoning, say, using Fatou's lemma as in Remark \ref{freetozero}, it follows that any limit point as $n,T \to \infty,$ must be UC. Hence it is UT and satisfies \eqref{eq:asymppin}.  Thus by Theorem \ref{th:uniqueness} any limit point must agree with $\mu^0$. This finishes the proof. 
\end{proof}

\bibliographystyle{plain}

\bibliography{bib_tightness}

\end{document}